\documentclass[11pt,reqno]{amsart}
\usepackage{amsmath}
\usepackage{amssymb}
\usepackage{graphicx}
\usepackage[margin=1.4in]{geometry}
\usepackage{comment}

\usepackage{amsthm}\usepackage{amsthm}\usepackage[matrix,arrow]{xy}\usepackage{enumerate}\usepackage{amsfonts}\usepackage{amsthm}\usepackage{amscd}

\usepackage{epstopdf}
\newtheorem{theorem}{Theorem}
\newtheorem{lemma}[theorem]{Lemma}
\newtheorem{prop}[theorem]{Proposition}
\newtheorem{cor}[theorem]{Corollary}

\newtheorem{defin}[theorem]{{Definition}}
\newtheorem{definition}[theorem]{{Definition}}
\newtheorem{claim}[theorem]{{Claim}}

\theoremstyle{remark}
\newtheorem{remark}[theorem]{Remark}

\newcommand{\fcal}{\mathcal{F}}
\newcommand{\gcal}{\mathcal{G}}\newcommand{\hcal}{\mathcal{H}}

\newcommand{\jcal}{\mathcal{J}}

 \def\calH{\hcal} 
\def\calF{\fcal}  
  
 \def\calJ{\jcal} \def\calG{\gcal}

 \def\O{\Omega}
 
\def\vp{\varphi} \def\eps{\epsilon}

\def\Cvx{{\operatorname{Cvx}}}
\def\co{{\operatorname{co}}}

\def\MA{Monge--Amp\`ere } 

\def\lb{\label}
\def\h#1{\hbox{#1}}
\def\q{\quad} \def\qq{\qquad}
\def\ra{\rightarrow}

\def\id{{\operatorname{id}}}

\def\half{{\frac{1}{2}}}
\def\div{\operatorname{div}}
\def\del{\partial}
 \newcommand{\RR}{\mathbb{R}}
 
 \newcommand{\NN}{{\mathbb N}}

\def\beq{\begin{equation}}
\def\eeq{\end{equation}}
\def\bpf{\begin{proof}}
\def\epf{\end{proof}}
\def\eaeq{\end{aligned}}
\def\baeq{\begin{aligned}}
\def\O{\Omega}\def\L{\Lambda}

\def\L{\Lambda}
\def\ve{\varepsilon}

\allowdisplaybreaks

\usepackage{graphicx}
\graphicspath{
    {converted_graphics/}
    {/}
}

\begin{document}

\title[Optimal transport via a Monge--Amp\`ere optimization problem]{Optimal transport via a Monge--Amp\`ere\\optimization problem}

\author[M. Lindsey]{Michael Lindsey}
\address{University of California, Berkeley}
\email{lindsey@berkeley.edu}

\author[Y. A. Rubinstein]{Yanir A.\ Rubinstein}
\address{University of Maryland}
\email{yanir@umd.edu}

\maketitle

\begin{abstract}
We rephrase Monge's optimal transportation (OT) problem with quadratic cost---via a Monge--Amp\`ere equation---as an infinite-dimensional optimization problem, which is in fact a convex problem when the 
target is a log-concave measure with convex support. We define a natural finite-dimensional discretization to the problem and 
associate a piecewise affine convex function
to the solution of this discrete problem. 
The discrete problems always admit a solution, which can be obtained by standard convex optimization algorithms whenever the target is a log-concave measure with convex support.
We show that under suitable regularity conditions the convex functions retrieved from 
the discrete problems converge to the convex 
solution of the original OT problem
furnished by Brenier's 
theorem. Also, we put forward an interpretation of our convergence result that suggests applicability to the convergence of a wider range of numerical methods for OT. Finally, we demonstrate the practicality of our convergence result by
providing visualizations of OT maps as well as of the dynamic OT problem obtained by solving the discrete problem numerically.
\end{abstract}

\tableofcontents

\section{Introduction}
\label{IntroSec}

In this article we develop a scheme for the numerical solution of the 
the Monge--Amp\`ere equation
governing optimal transport. 
Our main theorem provides natural and computationally feasible approximations of optimal transportation (OT) maps as well as of their Brenier convex potential, 
in the case of quadratic cost.

To achieve this we rephrase Monge's problem
as an infinite-dimensional optimization problem, 
which is, in fact, a convex problem when the target is a measure with convex support whose density $g$ is such that
$g^{-1/n}$ is convex.
Note that this class of measures {\it includes} all
log-concave measures with convex support.
We define a natural finite-dimensional discretization to the problem and 
associate a piecewise affine convex function
to the solution of this discrete problem. 
The discrete problems always admit a solution, which can be obtained by standard convex optimization algorithms whenever the target is a measure with convex support whose density $g$ is such that
$g^{-1/n}$ is convex.
We show that under suitable regularity conditions the convex functions retrieved from the discrete problems converge to the convex 
solution of the original OT problem
furnished by Brenier's 
theorem. 
While this result yields new insights
about optimal transport maps
it also has applications to numerical 
simulation of optimal transportation.
We illustrate this with a number of numerical examples.

\subsection{Yet another formulation of Monge's problem}
 \label{YetSubSec} 

Let $\Omega$ and $\Lambda$ be bounded open sets in $\mathbb{R}^{n}$
with $\Lambda$ convex, 
and let $f$ and $g$ be positive functions
on $\Omega$ and $\Lambda$, respectively, each bounded away from
zero and infinity. For simplicity,  assume that $f$ and $g$ are in $C^{0,\alpha}(\overline{\Omega})$ and $C^{0,\alpha}(\overline{\Lambda})$, respectively, and that they define positive measures $\mu$ and $\nu$ on $\Omega$ and
$\Lambda$, respectively, by 
$$
\mu=f\,dx, \quad \nu=g\,dx,
$$
where $dx$ denotes the Lebesgue measure on $\mathbb{R}^{n}$, 
and, assume that 
\begin{equation}
\begin{aligned}
\label{massEq}
\int_{\Omega}f\,dx=\int_{\Lambda}g\,dx.
\end{aligned}
\end{equation}
Then by results by Brenier and Caffarelli 
\cite{Brenier, Caf1992, VillaniOldNew}, 
there exists a unique solution of the corresponding Monge problem 
for the quadratic cost, i.e., 
$$
\underset{\{T:\O\ra\L\,:\, T_\#\mu=\nu\}}
{\mathrm{minimize}\;}  
\int_{\O}|T(x)-x|^{2}\,d\mu(x),
$$
and, moreover, $T$ is in $C^{1,\alpha}(\O)$.
Further, the solution is given by $T=\nabla\varphi$, for 
$\varphi$ convex and $C^{2,\alpha}(\O)$. 
In addition, $\varphi$ is the unique (up to an additive constant) 
Brenier (and, hence, also Alexandrov, or viscosity) solution of the second 
boundary value problem for the Monge--Amp\`ere equation 
\beq
\baeq
\label{MAEq}
\det\left(\nabla^{2}\varphi(x)\right)
&=\frac{f(x)}{g\left(\nabla\varphi(x)\right)},
\q x\in\Omega,
\cr
\del\vp(\O) &\subset\overline{\Lambda},
\eaeq
\eeq
where $\del\vp$ denotes the subdifferential map associated to the convex function $\vp$.

The following result rephrases Monge's problem as two different  infinite-dimensional 
optimization problems.
Theses problems can be considered `convex'  whenever, in particular, the target 
measure has log-concave (e.g., uniform) density with convex support. We refer
to \S\ref{InfDimConvSec} for a proof, as well as more details and intuition on the
aforementioned interpretation.

\begin{prop}
\label{InfDimProp}
With notation and hypotheses as in the above discussion, $\varphi$
is the unique solution of the following optimization problems: 
\begin{eqnarray*}
 & \underset{\psi\in \calJ}{\mathrm{minimize}\;}  
\calF(\psi):=\int_{\Omega} \calG_i(\psi,x)dx,\ i\in\{1,2\},
\end{eqnarray*}
 where 
$$
\calJ:=\{\psi\in C^2(\Omega)\,:\, \psi\ \mathrm{\ convex,\  and\  }
\nabla\psi(\Omega)\subseteq\Lambda\},
$$
and
$$
\baeq
\calG_1(\psi,x)
&:=
\max\left\{0,
-\log\det\left(\nabla^{2}\psi(x)\right)-\log g\left(\nabla\psi(x)\right)+\log f(x)\right\},
\cr
\calG_2(\psi,x)
&:=
\max
\left\{
0,
-\big(
\det\left(\nabla^{2}\psi(x)\right)
\big)^{1/n}
+ 
\Big(
f(x)/g\left(\nabla\psi(x)\right)
\Big)^{1/n}
\right\}.
\eaeq
$$
\end{prop}

The previous result, while not difficult to prove, provides
the key starting point for our discretization, that we now turn to discuss.

\subsection{A discrete Monge--Amp\`ere optimization problem}
\label{DMAOPIntroSubSec}

Let 
$$
x_{1},\ldots,x_{N}\in\overline{\Omega},
$$ 
be points,
and let 
$$
S_{1},\ldots,S_{M}\subset \overline{\Omega}
$$
be $n$-dimensional simplices whose set of vertices equals $\{x_{j}\}_{j=1}^N$ 
that together form an {\it almost-triangulation} of $\overline{\Omega}$. 
By this we mean that the intersection of any two of the $S_{i}$ is either empty or a common face (of any dimension) and that $\overline{\Omega} \backslash \bigcup_{i=1}^{M}S_{i}$,
has `small' volume. Note that the triangulation can be made perfect if
$\Omega$ is a polytope. We denote the vertices of the simplex 
$S_{i}$
by 
$$
x_{i_{0}},\ldots,x_{i_{n}},\q i_0,\ldots, i_n\in\{1,\ldots,N\},
$$
so that $S_i 
= 
\co\{
x_{i_{0}},\ldots,x_{i_{n}}
\}
$, 
where 
$$
\co\, A
$$ 
denotes the convex hull of a set $A$.
We denote by
$$
\eta_1,\ldots,\eta_N\in\overline{\Lambda}
$$
points in the closure of the target domain $\Lambda$, and we think of $\eta_j$ as the `image' 
of $x_j$, so that intuitively $S_i$ gets mapped to
$\co\{
\eta_{i_{0}},\ldots,\eta_{i_{n}}
\}
$.
We associate to the discrete map $x_j\mapsto \eta_j$ and the almost-triangulation 
$\{S_i\}_{i=1}^M$ a sort of discrete Jacobian
defined separately for each simplex. To define this, let
$A_{i}$, and $B_{i}$ be $n$-by-$n$  matrices defined by
\beq
\baeq
A_{i}&:=\left[\begin{array}{cccc}
\left(x_{i_{1}}-x_{i_{0}}\right) & \left(x_{i_{2}}-x_{i_{0}}\right) & \cdots & \left(x_{i_{n}}-x_{i_{0}}\right)\end{array}\right]^{T},
\cr
B_{i}&:=
\left[\begin{array}{cccc}
\left(\eta_{i_{1}}-\eta_{i_{0}}\right) & \left(\eta_{i_{2}}-\eta_{i_{0}}\right) & \cdots & \left(\eta_{i_{n}}-\eta_{i_{0}}\right)\end{array}\right]^{T}
\cr
\eaeq
\eeq
(here, $x_j$ and $\eta_j$ are represented by column vectors).
The {\it discrete Jacobian associated to the simplex $S_i$} is then the $n$-by-$n$  matrix
\beq
\lb{HiEq}
\baeq
H_{i}:=
H\Big(S_i,\eta_{i_{0}},\ldots,\eta_{i_{n}}\Big)
&:=
\frac{1}{2}\left(A_{i}\right)^{-1}B_{i}+\frac{1}{2}\left(\left(A_{i}\right)^{-1}B_{i}\right)^{T}.
\cr
\eaeq
\eeq
This name is further motivated by Lemma \ref{LindseyLemma}. Finally, denote the volume of the $i$-th simplex by 
\beq
\lb{volDef}
V_i := \mathrm{vol}(S_i).
\eeq
We now introduce a discrete analogue of the first optimization problem associated to the Monge--Amp\`ere equation \eqref{MAEq} introduced in Proposition \ref{InfDimProp}.

\begin{defin}
\label{DMAOPDef}
The {\it logarithmic discrete Monge--Amp\`ere optimization problem (LDMAOP)}
associated to the data $(\O,\L,f,g,\{x_j\}_{j=1}^N,\{S_i\}_{i=1}^M)$ is:
\begin{eqnarray}
 & \underset{\{\psi_{i}\in\RR,\eta_{i}\in\RR^n\}_{i=1}^N}{\mathrm{minimize}} & \sum_{i=1}^M V_i\cdot
\max\bigg\{0,-\log\det H_i
-\log g\Big({\textstyle
\frac{ \sum_{j=0}^n \eta_{i_j} }{n+1} }\Big)
+\log f\Big({\textstyle
\frac{ \sum_{j=0}^{n}x_{i_{j}}} {n+1} }\Big)\bigg\} \nonumber \\
 & \mathrm{subject\ to} & 
\psi_{j}\geq\psi_{i}+\left\langle \eta_{i},x_{j}-x_{i}\right\rangle ,\ i,j=1,\ldots,N,
\lb{Constraint1Eq}\\
 &  & \eta_{i}\in\overline{\Lambda},\ i=1,\ldots,N,
\lb{Constraint2Eq}\\
 &  & H_i
>0,\ i=1,\ldots,M.
\lb{Constraint3Eq}
\end{eqnarray}

\end{defin}

To ease the notation in the following, we define 
\beq
\lb{FiDef}
F_i\Big( \big\{\psi_j^{(k)},\eta_j^{(k)}\big\}_{j=1}^{N(k)} \Big)
:=
\max\bigg\{0,-\log\det H_i
-\log g\Big({\textstyle
\frac{ \sum_{j=0}^n \eta_{i_j} }{n+1} }\Big)
+\log f\Big({\textstyle
\frac{ \sum_{j=0}^{n}x_{i_{j}}} {n+1} }\Big)\bigg\}
\eeq
and
\beq
\label{FDef}
F\Big( \big\{\psi_j^{(k)},\eta_j^{(k)}\big\}_{j=1}^{N(k)} \Big)
:=
\sum_{i=1}^{M(k)} V_i \cdot F_i\Big( \big\{\psi_j^{(k)},\eta_j^{(k)}\big\}_{j=1}^{N(k)} \Big),
\eeq
so $F_i$ is a per-simplex penalty, and $F$ is the objective function of the DMAOP.

The variables of the LDMAOP are $\psi_1,\ldots,\psi_N$
and $\eta_1,\ldots,\eta_N$ (while $x_1,\ldots,x_N, \O,$  $\Lambda, f, g$ are given).
These variables are the discrete analogues
of the values of the convex potential and its gradient, respectively, at the points $x_1,\ldots,x_N$,
while, $H_{i}$ is the discrete analogue
of the Jacobian, on the simplex $S_i$, of the map $x_i\mapsto \eta_i$. 
One can think of $\det H_{i}$ as a measure of the volume distortion
of simplex $S_{i}$ under the map. 

We will see later (see Proposition \ref{DMAOPProp} (i)) that the LDMAOP is feasible (i.e., that there exists a point in the variable space 
$(\psi_{1},\ldots,\psi_N, \eta_{1},\ldots,\eta_N)\in(\RR\times\RR^n)^N$
satisfying the constraints  \eqref{Constraint1Eq}--\eqref{Constraint3Eq}) for a fine enough triangulation. Then for such a triangulation, pick a feasible point
in $(\RR\times\RR^n)^N$, and let $c$ be the corresponding cost. Note that the objective function is continuous in the optimization variables $\psi_{1},\ldots,\psi_{N}$ and $\eta_{1},\ldots,\eta_{N}$. Furthermore, it is clear from the problem that the intersection of the optimization domain with the region on which the cost is at most $c$ is closed. It is also clear that the $\eta_i$ are bounded over the optimization domain (since $\Lambda$ is bounded). Thus, since the objective function depends only on the $\eta_i$ directly, the problem admits a minimizer (not necessarily unique). In the case that $\Lambda$ is convex and $g$ is log-concave, the LDMAOP is in fact a convex optimization problem (see Remark \ref{convexityRmk}), so standard results 
(see, for instance, \cite[Chapter 11]{Boyd})  guarantee that the LDMAOP can be solved efficiently (in particular, in polynomial time).

\subsection{A modified optimization scheme}
\label{mDMAOPSubSec}

A slight variation of Definition \ref{DMAOPDef} turns out to give stronger results.

\begin{defin}
\label{mDMAOPDef}
The {\it discrete Monge--Amp\`ere optimization problem (DMAOP)}
associated to the data $(\O,\L,f,g,\{x_j\}_{j=1}^N,\{S_i\}_{i=1}^M)$ is:
\begin{eqnarray}
 & \underset{\{\psi_{i}\in\RR,\eta_{i}\in\RR^n\}_{i=1}^N}{\mathrm{minimize}} & \sum_{i=1}^M V_i\cdot
\max\bigg\{0,-(\det H_i)^{1/n}
+ 
\Big(
f\Big({\textstyle
\frac{ \sum_{j=0}^{n}x_{i_{j}}} {n+1} }\Big)
/
g\Big({\textstyle
\frac{ \sum_{j=0}^n \eta_{i_j} }{n+1} }\Big)
\Big)^{1/n}
\bigg\} \nonumber \\
 & \mathrm{subject\ to} & 
\psi_{j}\geq\psi_{i}+\left\langle \eta_{i},x_{j}-x_{i}\right\rangle ,\ i,j=1,\ldots,N,
\lb{Constraint1Eq2}\\
 &  & \eta_{i}\in\overline{\Lambda},\ i=1,\ldots,N,
\lb{Constraint2Eq2}\\
 &  & H_i
\geq 0,\ i=1,\ldots,M.
\lb{Constraint3Eq2}
\end{eqnarray}

\end{defin}

Most of our results and proofs will apply equally well to the LDMAOP
as to the DMAOP. However, the DMAOP improves upon the LDMAOP in two ways. First, the DMAOP is a convex problem for a wider class of $g$ (see Remark \ref{convexityRmk}), namely the set of $g$ such that $g^{-1/n}$ is convex. Second, the fact that $(\cdot)^{1/n}$ is bounded near zero (which is not true of $\log$) will allow us to require less regularity of the Brenier potential $\vp$. In specific, we will not need regularity up to the boundary. For notational clarity, the bulk of our convergence proof will treat the LDMAOP directly, though we remark that the same results hold for the DMAOP via obvious modifications. See Section \ref{upgrade} for details about how to strengthen the results obtained for the LDMAOP in the DMAOP case.

\subsection{Convex functions associated to the solution of the (L)DMAOP}
\label{ConvFnDMAOPSubSec}
Next, we construct a piecewise linear convex function $\phi$
associated with a solution 
$\{\psi_{j},\eta_{j}\}_{j=1}^{N}$
of the (L)DMAOP.
Define,
\beq
\label{ajEq}
a_{j}(x):=\psi_{j}+\langle\eta_{j},x-x_{j}\rangle,
\qq j=1,\ldots, N,
\eeq 
so 
$a_{j}$ is the (unique) affine function with 
$a_{j}(x_{j})=\psi_{j}$
and
$\nabla a_{j}(x_{j})=\eta_{j}$. 
Define the {\it optimization potential}  by
\beq\label{phiEq}
\phi(x):=b+\max_{j=1,\ldots N}a_{j}(x),
\eeq
where $b\in \RR$ is chosen such that $\phi(0)=0$ 
(and we have assumed, without loss of generality, that $0\in\Omega$). Notice that
we have defined $\phi$ on all of $\mathbb{R}^{n}$.
As the supremum of affine functions,  $\phi$ is convex. The point is that $\phi$ still
encodes the solution of the (L)DMAOP. Indeed, by the constraints of the (L)DMAOP
(specifically, \eqref{Constraint1Eq}),
\beq\label{ApproxPhiGridEq}
\phi(x_{j})=\psi_{j} + b,
\eeq
and 
\beq\label{ApproxGradPhiGridEq}
\eta_{j}\in\partial\phi(x_{j})
\eeq
(as $a_j$ is an affine function with slope
$\eta_j$ lying below $\phi$ but touching
it at $x_j$).

\subsection{Relationship of the (L)DMAOP with discrete optimal transport}
\label{discConnect}

We briefly describe a connection between the (L)DMAOP and classical
discrete optimal transport problems (DOTP). 
In fact, the solution of the (L)DMAOP gives the solution to a corresponding DOTP.

Let $\left\{ \psi_{j},\eta_{j}\right\} _{j=1}^{N}\in (\RR\times\RR^n)^N$ be a solution
of the (L)DMAOP (with notation as above). The construction
of \S\ref{ConvFnDMAOPSubSec}
 yields, by \eqref{ApproxPhiGridEq}-\eqref{ApproxGradPhiGridEq} a piecewise-linear convex function $\phi:\mathbb{R}^{n}\rightarrow\mathbb{R}$
such that $\phi(x_{j})=\psi_{j}$ and $\partial\phi(x_{j})\ni\eta_{j}$.
By Rockafellar's Theorem \cite[Theorem 24.8]{Rock}, the set 
$\bigcup_{j=1}^{N}(x_{j},\eta_{j})\subset\RR^{2n}$ is
cyclically monotone (see \cite[Section 24]{Rock} for definitions) since
it is a subset of the graph of the subdifferential of the convex function $\phi$.
Thus, $x_{j}\mapsto\eta_{j}$ solves the
optimal transport problem from $\mu_{D}=\sum_{j=1}^{N}\delta_{x_{j}}$
to $\nu_{D}=\sum_{j=1}^{N}\delta_{\eta_{j}}$ \cite{Villani},
where $\delta_p$ denotes the Dirac delta measure concentrated at $p$. 
We state this result
as a proposition.
\begin{prop}
Let $\left\{ \psi_{j},\eta_{j}\right\} _{j=1}^{N}$ be a solution of the (L)DMAOP. Then $T:\{x_{j}\}\rightarrow\{\eta_{j}\}$
given by $x_{j}\mapsto\eta_{j}$ solves the Monge problem
with source $\mu_{D}=\sum_{j=1}^{N}\delta_{x_{j}}$ and target $\nu_{D}=\sum_{j=1}^{N}\delta_{\eta_{j}}$.
\end{prop}

Of course, the target points $\eta_{j}$ are not fixed before the
optimization problem is solved. Indeed we expect the results of solving the
(L)DMAOP to be significantly better than the results obtained by picking $N$ target
points in advance and then solving the resulting DOTP. This expectation
is based on the fact that the (L)DMAOP chooses target points $\eta_{j}$
in a way that attempts to achieve correct volume distortion.\\

\subsection{Convergence of the discrete solutions}
\label{ConvSubSec}
Now, we take a sequence of almost-triangulations $\{x_i\}_{i=1}^{N(k)}, \{S_i^{(k)}\}_{i=1}^{M(k)}$ indexed by $k$ 
(so now both $N$ and $M$ are functions of $k$, although we will usually omit
that dependence from our notation) satisfying the following assumptions.
Denote by $\h{\rm diam}\, A$ the diameter of a set $A$ in Euclidean space.

\begin{definition}
\label{admissibleATDef} 
We say that the sequence of almost-triangulations 
$\{\{S_i^{(k)}\}_{i=1}^{M(k)}\}_{k\in\NN}$
is admissible if 
$$
\lim_{k\ra\infty}\;
\max_{i\in\{1,\ldots, M(k)\}}\h{\rm diam}\, S_{i}^{(k)}
=0,
$$
and there are open sets $\Omega_{\varepsilon}\subset \O$ indexed
by $\varepsilon>0$, with
$$
\Omega_{\varepsilon}\subset\Omega_{\varepsilon'}, \q
\h{\ for $\varepsilon'\le\varepsilon$},
$$
and 
$$
\bigcup_{\varepsilon>0}\Omega_{\varepsilon}=\Omega,
$$
and such that
for any $\varepsilon>0$ sufficiently small, we have that an $\varepsilon$-neighborhood
of $\Omega_{\varepsilon}$ is contained 
within the $k$-th almost-triangulation
for all $k$ sufficiently large, i.e.,
$$
\Omega_{\varepsilon}+B_\eps(0)
\subset \bigcup_{i=1}^{M(k)}
S_i^{(k)}, \q \forall k\gg 1.
$$
\hfill\break 

\end{definition}

Given an admissible sequence of almost-triangulations
we construct the optimization potentials  
\begin{equation}
\begin{aligned}
\label{phikEq}
\phi^{(k)}\in C^{0,1}(\RR^n), \q k\in\NN,
\end{aligned}
\end{equation}
associated with the solution 
$
\big\{\psi_{j}^{(k)},\eta_{j}^{(k)}\big\}_{j=1}^{N(k)}
$
of the $k$-th (L)DMAOP,
 by the prescription of the previous subsection (specifically, by equation \eqref{phiEq}). 
Our main theorem concerns the convergence of the optimization potentials to the Brenier potential, i.e., to the solution of the PDE  \eqref{MAEq}.

\begin{theorem}\label{MainThm}
Let $f\in C^{0}(\overline{\Omega}), \; g\in C^{0}(\overline{\Lambda}) $ and suppose $\Lambda$ is convex.
Let $\{\{S_i^{(k)}\}_{i=1}^{M(k)}\}_{k\in\NN}$ be an admissible sequence of almost-triangulations
(recall Definitions \ref{admissibleATDef}) .
Suppose that the optimal cost of the $k$-th (L)DMAOP tends to zero as $k\rightarrow\infty$.
Then, as $k\rightarrow \infty$, the optimization potentials $\phi^{(k)}$   \eqref{phikEq}  
converge uniformly on $\overline{\Omega}$ 
to the unique Brenier solution $\vp$ of the Monge--Amp\`ere equation \eqref{MAEq} with $\varphi(0)=0$.
\end{theorem}

\begin{remark}
For the proof of Theorem \ref{MainThm}, we do not actually need $\phi^{(k)}$ to be the optimization potential retrieved from the optimal solution of the $k$-th (L)DMAOP. We only need that the $\phi^{(k)}$ are defined via \eqref{ajEq} and  \eqref{phiEq} by some $\{\psi_{j}^{(k)},\eta_{j}^{(k)}\}_{j=1}^{N}$ which satisfy the constraints of Definition \ref{DMAOPDef} and for which the cost of Definition \ref{DMAOPDef} tends to zero as $k\ra\infty$. See \S\ref{otherMethods} for implications.
\end{remark}

The assumption on the optimal cost holds in many interesting cases
by assuming a mild regularity condition on the almost-triangulations.

\begin{defin}
\lb{regularATDef}
We say that a sequence
$
\big\{\{S_i\}_{i=1}^{M(k)}\big\}_{k\in\NN}
$
of almost-triangulations of $\Omega$ is regular if there exists $R>0$ such that 
$$
\det \left(u_{i,1}^{(k)} \,\,\, u_{i,2}^{(k)} \,\,\, \cdots \,\,\, u_{i,n}^{(k)} \right) \geq R,
\q
\forall i\in\{1,\ldots, M(k)\}, \;\; \forall k\in\NN,
$$
where
$$
u_{i,j}^{(k)} := \Vert x_{i_{j}}-x_{i_{0}} \Vert^{-1} \left( x_{i_{j}}-x_{i_{0}} \right).
$$
\end{defin}
For example, in dimension $n=2$, a sequence of almost-triangulations is regular if
the angles of the triangles are bounded below uniformly in $k$.

Based on Theorem \ref{MainThm} we prove the following general convergence result
that does not make any assumptions on the optimal cost of the discretized problems.
The convergence we obtain on the level of subdifferentials can be viewed as optimal since $\phi^{(k)}$ are Lipschitz but no better, i.e., $\phi^{(k)}\not\in C^1$.

\begin{theorem}
\label{SecondMainThm}
Let $f\in C^{0,\alpha}(\overline{\Omega}), \; g\in C^{0,\alpha}(\overline{\Lambda}) $ and suppose $\Lambda$ is convex.
Let $\big\{\{S_i^{(k)}\}_{i=1}^{M(k)}\big\}_{k\in\NN}$ be an admissible and regular sequence of almost-triangulations
(recall Definitions \ref{admissibleATDef} and \ref{regularATDef}).
Let  $\vp$ be the the unique Brenier solution
of the Monge--Amp\`ere equation \eqref{MAEq} with $\varphi(0)=0$,
and let $\phi^{(k)}$ be the optimization potentials \eqref{phikEq} obtained from the (L)DMAOP.
If the $\phi^{(k)}$ are obtained from the LDMAOP, additionally suppose that $\vp\in C^{2,\alpha}(\overline{\O})$. Then,
$$
\phi^{(k)}\ra \vp \h{\ uniformly on $\overline{\O}$},
$$
and $\del\phi^{(k)}\ra \nabla\vp$ pointwise on $\overline{\O}$.
In particular, $\nabla\phi^{(k)}$ converges pointwise almost everywhere to the optimal transport map pushing forward $\mu=f\,dx$ to $\nu=g\,dx$.
\end{theorem}

\begin{remark}
In the LDMAOP case, the theorem applies whenever $\Omega$ and $\Lambda$ are uniformly convex and of class $C^{2}$
since it is known (recalling our initial assumption that $f$ and $g$ are in $C^{0,\alpha}(\overline{\Omega})$ and $C^{0,\alpha}(\overline{\Lambda})$, respectively) that in that case $\vp\in C^{2,\alpha}(\overline{\O})$. In the DMAOP case, the only needed regularity is $\vp\in C^{2}(\O)$, which is already guaranteed by the convexity of $\Lambda$ and our assumptions on $f$ and $g$. Thus we do \textit{not} need to assume the convexity (or even connectedness) of $\O$ to get convergence for the DMAOP. For a review of the relevant regularity theory, see \cite[Chapter 4]{Villani}.

\end{remark}

\begin{proof}
Under the assumption that the 
sequence of almost-triangulations is admissible and regular and that
 $\vp\in C^{2}({\O})$, Corollary \ref{upgradedcostCor} states that
$c_k$, the optimal cost of the $k$-th DMAOP, tends to zero as $k\rightarrow\infty$, while
if $\vp\in C^{2,\alpha}(\overline{\O})$, Corollary \ref{costCor} states that
$c_k$, the optimal cost of the $k$-th LDMAOP, tends to zero as $k\rightarrow\infty$.
Thus, Theorem \ref{MainThm} implies $\phi^{(k)}$ converges uniformly to
the Brenier solution $\vp$. We will see from the proof of Theorem \ref{MainThm} that we can in fact assume that the $\phi^{(k)}$ are uniformly convergent on a closed ball $D$ containing $\overline{\O}$ in its interior. The convergence of subgradients then follows from Theorem \ref{BWThm}
since if a sequence of lower semicontinuous finite convex functions converges
uniformly on bounded sets to some convex function, then the sequence
epi-converges to this function \cite[Theorem 7.17]{RW}.
\end{proof}

\begin{remark}
The conditions of admissibility and regularity on the sequence of triangulations are necessary for technical reasons, but they are fulfilled easily in practice.
\end{remark}
\begin{remark}
We see from the statement of Theorem \ref{MainThm} that, in practice, even in situations in which we cannot guarantee convergence, we can acquire good heuristic evidence in favor of convergence if the optimal cost of the $k$-th DMAOP becomes small as $k\rightarrow \infty$. However, it remains an open problem to prove the error estimates that would make this insight rigorous.
\end{remark}

\begin{remark}
Although, via Corollary \ref{upgradedcostCor}, we do not require regularity up to the boundary to obtain convergence in the DMAOP case, it can be seen that the proof of Corollary \ref{upgradedcostCor} does not yield a rate for the convergence $c_k \ra0$. To prove such a rate, we need
 $f\in C^{0,\alpha}(\overline{\Omega}), \; g\in C^{0,\alpha}(\overline{\Lambda}) $ and
$\vp \in C^{2,\alpha}(\overline{\Omega})$, since in this case it follows from the proof of Corollary \ref{costCor} that (in both the DMAOP and LDMAOP cases) $c_k = O(h^{\alpha})$, where $h=h(k)$ is the maximal simplex diameter in the $k$-th triangulation. Note that a bound on $c_k$ does \textit{not} imply an error bound for the optimization potentials, though it suggests a candidate.

\end{remark}

\subsection{Possible application to convergence of other methods}
\label{otherMethods}

We remark that Theorem \ref{MainThm} could in principle be used to guarantee the convergence of other numerical methods for optimal transport. Indeed, suppose that a numerical method furnishes values $\psi_j^{(k)}$ and subgradients $\eta_j^{(k)}$ of a convex function at discretization points $x_j^{(k)}$ for a sequence of increasingly fine discretizations indexed by $k$. Further, suppose that we can associate a triangulation to the $k$-th set of discretization points and so produce an admissible and regular sequence of almost-triangulations of $\Omega$. (Note that such triangulation is possible in particular if the discretization points are taken from an increasingly fine rectangular grid. More generally, for sensible discretizations an admissible sequence of triangulations could likely be achieved by Delaunay triangulation.) Lastly, suppose that the (L)DMAOP cost evaluated at $\{\psi_1^{(k)},\ldots,\psi_{N(k)}^{(k)},\eta_1^{(k)},\ldots,\eta_{N(k)}^{(k)}\}$ converges to zero as $k\ra\infty$. This is possible whenever the numerical method achieves vanishing violation (on average over the discretization points and even in the one-sided sense of the cost function of the (L)DMAOP) of a finite-difference approximation to the Monge--Amp\`ere equation.

We suspect that other numerical methods for optimal transport based on solution of the Monge--Amp\`ere equation could be covered by this convergence result under sufficient regularity conditions on the Brenier potential $\varphi$. Indeed, it is perhaps most appropriate to think of Theorem $\ref{MainThm}$ as a convergence result for convex functions. We interpret the result as stating that a sort of one-sided convergence (over an increasingly fine set of discretization points) of the discrete Hessian determinants of a sequence of convex functions to the `right-hand side' of a Monge--Amp\`ere equation implies convergence of the convex functions themselves to the solution of the Monge--Amp\`ere equation. We can interpret this as a coercivity-type result allowing us to pass from vanishing violation (even one-sided) of the Monge--Amp\`ere equation to vanishing deviation from the solution.

\subsection{Comparison to existing methods}

Our work relies on a (convex) optimization approach. In this subsection we review
other, different, approaches to discretizing the OT problem and, sometimes, more generally,
the \MA equation or even more general fully nonlinear second order elliptic PDEs.
The literature on numerical methods for \MA equations in general,
and for optimal transport maps in particular, has grown considerably recently. Therefore, we do not attempt a comprehensive review of existing methods 
in the literature, but concentrate on briefly mentioning those approaches for which 
both a numerical algorithm has been implemented {\it and} a convergence result has been proven. For a thorough survey of existing numerical 
methods we refer the interested reader to the article 
of Feng--Glowinski--Neilan 
\cite{FengGN} and references therein.

Oliker--Prussner \cite{OlikerPrussner} and Baldes--Wohlrab \cite{BaldesWohlrab} 
initiated the study of discretizations of the 2-dimensional Monge--Amp\`ere equation, and obtained a convergence theorem for the Dirichlet problem for the equation
$u_{xx}u_{yy}-u_{xy}^2=f$ on a bounded domain in $\RR^2$.
This used, among other things, 
classical constructions of Minkowski \cite{Minkowski}
and Pogorelov \cite{Pogorelov}.

Benamou--Brenier \cite{BenamouBrenier} introduced, on the other hand, a discretization 
scheme for the dynamic formulation of the optimal transport problem, that
does not involve the \MA equation. This involves solving the system of equations
for $\rho:[0,T]\times\RR^{n}\ra\RR_+, v:[0,T]\times\RR^{n}\ra\RR^{n}$
\beq
\baeq
\del_t \rho + \div(\rho v) &=0, \cr
\del_t \phi + |\nabla_x\phi|^2/2 &=0,\cr
\eaeq
\eeq
with the constraint
$v=\nabla_x\phi$ and the boundary conditions
 $\rho(0,\,\cdot\,) =f,  \rho(T,\,\cdot\,)=g$.
The authors use discretization in space-time that falls under the framework of problems in numerical fluid mechanics. See also the work of 
Angenent--Haker--Tannenbaum \cite{AngenentHT} and
Haber--Rehman--Tannenbaum  \cite{HaberRT}.
More recently, Guittet proved that the Benamou--Brenier scheme converges 
when the target is convex and the densities are smooth \cite{Guittet}.

Recently, Benamou--Froese--Oberman developed a convergence proof via a direct discretization of the Monge--Amp\`ere equation \cite{BenamouFO,BenamouFO2}.
Their approach gives the convergence result for viscosity solutions of the Monge--Amp\`ere equation for convex target with $g\in C^{0,1}(\L)$ and $f\in L^\infty(\O)$.
Their discretization scheme and convergence proof rely on earlier work of 
Barles--Souganidis \cite{BarlesS} and Froese and Oberman \cite{FroeseO}.
Other recent work includes, e.g., 
Loeper--Rapetti \cite{LoeperR},
Sulman--Williams--Russell \cite{SulmanWR}, Kitagawa \cite{Kitagawa},
and Papadakis--Peyr\'e--Oudet \cite{PapadakisPO}.

Another approach one could pursue is to approximate the measures
by empirical measures (sums of Dirac measures). In the simplest
case when the number of Dirac measures is the same for the source
and the target, the solution is given by solving the
assignment problem that has efficient numerical implementations.
We refer to M\'erigot--Oudet \cite{QuentinOudet}
and \cite[p. 213]{PapadakisPO} for relevant references
(cf. \cite{AMSAbstract} for an implementation in some simple cases). 
It is interesting to note that the method
presented in this article a forteriori solves an assignment problem,
but for target Dirac measures whose location is not a priori known 
(as explained in \S\ref{discConnect} above).

\section{Monge--Amp\`ere optimization problems}
\label{InfDimConvSec}

We recall some of the notation from \S\ref{YetSubSec}.
Let $\Omega$ and $\Lambda$ be bounded open sets in $\mathbb{R}^{n}$
with $0\in\O$ and $\Lambda$ convex. Let $f\in C^{0,\alpha}(\overline{\Omega})$ and 
$g\in C^{0,\alpha}(\overline{\Lambda})$ be positive functions
bounded away from
zero and infinity
satisfying  \eqref{massEq}. 
Let $\varphi\in C^{2,\alpha}(\O)$ be the unique 
convex solution of  \eqref{MAEq} with $\vp(0)=0$. 

We only give the proof of Proposition \ref{InfDimProp} for $\calG_1$ since the proof for 
$\calG_2$ is similar. Proposition \ref{InfDimProp}  is a special case of the following result.

\begin{lemma}
\label{MAOPLemma}
With notation and hypotheses as in the above paragraph, $\varphi$
is the unique  solution of the following optimization problem: 
\begin{eqnarray*}
 & \underset{\psi\in \Cvx(\O)\cap C^{2}(\Omega)}{\mathrm{minimize}} & \calF(\psi):=\int_{\Omega}h\circ \calG(\psi,x)\cdot\rho(x)dx\\
\\
 & \mathrm{subject\ to} & \nabla\psi(\Omega)\subseteq\Lambda,
\end{eqnarray*}
 where 
$$
\calG(\psi,x):=
\max\left\{0,
-\log\det\left(\nabla^{2}\psi(x)\right)-\log g\left(\nabla\psi(x)\right)+\log f(x)\right\},
$$
and
$h:[0,\infty)\rightarrow\mathbb{R}$ is convex and increasing with
$h(0)=0$, and $\rho$ is a positive function on $\Omega$, bounded
away from zero and infinity.
\end{lemma}

Before giving the proof we make several remarks.

\begin{remark}
\label{convexityRmk}
Notice that if $g$ is log-concave, this optimization problem can
be thought of as an `infinite-dimensional convex optimization problem'
(where the value of $\psi$ at each point $x$ is an optimization
variable). To see that the problem
can indeed be thought of as `convex,' notice/recall that\end{remark}
\begin{itemize}
\item $\nabla\psi(x)$ and $\nabla^{2}\psi(x)$ are linear in $\psi$
\item $\log\circ\det$ is concave on the set of positive semidefinite (symmetric)
matrices
\item the pointwise maximum of two convex functions is convex
\item the composition of a convex increasing function with a convex function
is convex
\item the set of convex functions is a convex cone
\item the specification that $\nabla\psi(\Omega)\subseteq\Lambda$ is a
convex constraint since $\Lambda$ is convex.\end{itemize}
These points also demonstrate that the discretized version of the problem (the LDMAOP) outlined above is a convex problem in the usual sense.

By similar reasoning, notice that the infinite-dimensional problem of the $\calG_2$ case of Proposition \ref{InfDimProp} and its discretization (namely, the DMAOP) are convex whenever $g^{-1/n}$ is convex. Notice that this holds in particular whenever $g$ is log-concave. Indeed, this can be seen by writing $g^{-1/n} = \exp\left(-\frac{1}{n} \log g\right)$ and recalling that $\exp$ preserves convexity. Thus the DMAOP is a convex problem for a strictly larger class of target measures.

\begin{remark}
Nevertheless, for the proof of the main theorems
we do not require that $g^{-1/n}$ is convex (nor that $g$ is log-concave), though these assumptions ensure that the (L)DMAOP is convex and, thus, feasibly solvable.
\end{remark}

\begin{remark}
(Intuitive explanation of Lemma  \ref{MAOPLemma}.) We can think of the objective function
in the statement of the lemma as penalizing `excessive contraction'
of volume by the map $\nabla\psi$ (relative to the `desired' distortion
given by the ratio of $f$ and $g$) while ignoring `excessive expansion.'
However, since we constrain $\nabla\psi$ to map $\Omega$ into $\Lambda$,
we expect that excessive expansion at any point will result in excessive
contraction at another, causing the value of the objective function
to be positive. Thus we expect that the optimal $\psi$ must in fact
be $\varphi$.\end{remark}
\begin{proof}
Note that $F(\varphi)=0$ since $\varphi$ solves the Monge--Amp\`ere
equation and that  $F(\psi)\geq0$ always. Thus
letting $\psi$ be such that $F(\psi)=0$, it only remains to show
that $\psi=\varphi$. For a contradiction, suppose that $\psi\neq\varphi$.
Since $\varphi$ is the unique solution to the Monge--Amp\`ere equation
above, there exists some $x_{0}\in\Omega$ such that 
\[
\det\left(\nabla^{2}\psi(x_{0})\right)\neq\frac{f(x_{0})}{g\left(\nabla\psi(x_{0})\right)}.
\]
If we have that the left-hand side is less than the right-hand side
in the above, then $G(\psi,x_{0})>0$, so by continuity $G(\psi,x)>0$
for $x$ in a neighborhood of $x_{0}$, and $F(\psi)>0$. Thus we
can assume that in fact 
\[
\det\left(\nabla^{2}\psi(x)\right)\geq\frac{f(x)}{g\left(\nabla\psi(x)\right)}
\]
 for all $x$, with strict inequality at a point $x_{0}$. By continuity,
we must also have strict inequality on en entire neighborhood of $x_{0}$.
In addition, we have that $\det\left(\nabla^{2}\psi(x)\right)$ is
bounded away from zero, so $\psi$ is strongly convex. Thus $\nabla\psi$
is injective, and we obtain by a change of variables 
$$
\baeq
\int_{\nabla\psi(\Omega)}g(y)dy & =  \int_{\Omega}g\left(\nabla\psi(x)\right)\det\left(\nabla^{2}\psi(x)\right)dx\\
 & >  \int_{\Omega}f(x)dx.
\eaeq
$$
Of course, since $\nabla\psi(\Omega)\subseteq\Lambda$, we have in
addition that $\int_{\Lambda}g(y)dy\geq\int_{\nabla\psi(\Omega)}g(y)dy$.
We have arrived at a contradiction because $\int_{\Lambda}g=\int_{\Omega}f$
by \eqref{massEq}. 
\end{proof}

\section{Convergence of solutions of the LDMAOP}
\lb{ConvSec}

In the following we will often consider sequences of LDMAOPs indexed
by $k$. We will maintain the notation from \S\ref{ConvSubSec}, 
adding ``$(k)$'' in superscripts as necessary.

\subsection{The objective function} First, we would like to understand
the behavior of the objective function of the LDMAOP. 
The following proposition gives a criterion guaranteeing the optimal cost 
(i.e., the minimum of the objective function) of the LDMAOP
converges to zero. In particular, it implies that Theorem \ref{SecondMainThm} 
follows from Theorem \ref{MainThm}.
The idea is to study the cost associated to the restriction of the solution $\vp$
of the \MA equation to the $k$-th almost-triangulations, i.e., to estimate the cost
\begin{equation}
\begin{aligned}
\label{dkDef}
d_k:=F\Big( \big\{\vp(x_j^{(k)}),\nabla\vp(x_j^{(k)})\big\}_{j=1}^{N(k)} \Big)
\end{aligned}
\end{equation}
associated to 
\begin{equation}
\begin{aligned}
\label{vpkEq}
\big\{\vp(x_j^{(k)}),\nabla\vp(x_j^{(k)})\big\}_{j=1}^{N(k)}\in(\RR\times\RR^n)^{N(k)}.
\end{aligned}
\end{equation}
A small caveat, of course, 
is to show first that this data actually 
satisfies the constraints of the discrete \MA optimization problem (LDMAOP),
and, subsequently, 
that the $k$-th LDMAOP is feasible. This is the content of part (i) of the next
proposition.

\begin{prop}
\label{DMAOPProp}
Let 
$
\big \{ \{S_{i}^{(k)}\}_{i=1}^{M(k)}\big\}_{k\in\NN}
$
be a sequence of admissible and regular almost-triangulations of $\Omega$ 
(recall Definitions \ref{admissibleATDef} and \ref{regularATDef}).
Let $\vp$ be the unique Brenier solution
of the Monge--Amp\`ere equation \eqref{MAEq} with $\varphi(0)=0$,
and suppose that $\vp\in C^{2,\alpha}(\overline{\O})$.
Then:
\hfill\break
(i)
The data   \eqref{vpkEq} satisfies
the constraints  \eqref{Constraint1Eq}--\eqref{Constraint3Eq} for all $k$ sufficiently large.
\hfill\break
(ii)
$\lim_k d_{k}=0$.
\end{prop}

Denote by 
\beq\label{DMAOPkEq}
\left\{\psi_{j}^{(k)},\eta_{j}^{(k)}\right\}_{j=1}^N\in(\RR\times\RR^n)^{N(k)}
\eeq
the solution to the $k$-th LDMAOP.
The optimal (minimal) cost 
of the LDMAOP associated with the $k$-th almost-triangulation is then
\begin{equation}
\begin{aligned}
\label{ckEq}
c_k:=F\Big( \big\{\psi_{j}^{(k)},\eta_{j}^{(k)}\big\}_{i=1}^{N(k)} \Big).
\end{aligned}
\end{equation}
Since $c_k\le d_k$, an immediate consequence of Proposition \ref{DMAOPProp} is:

\begin{cor}
\label{costCor}
Under the assumptions of Proposition \ref{DMAOPProp}, $\lim_k c_{k}=0$.
\end{cor}

\begin{proof}[Proof of Proposition \ref{DMAOPProp}]
(i) We claim that the feasibility conditions
 \eqref{Constraint1Eq}--\eqref{Constraint3Eq} 
are satisfied for 
$\big\{\vp(x_{j}^{(k)}), \nabla\vp(x_{j}^{(k)})\big\}_{j=1}^{N(k)}$ for all $k$ sufficiently large.
First, the convexity of $\vp$ implies \eqref{Constraint1Eq}.
Second, \eqref{Constraint2Eq} follows from \eqref{MAEq}. 
It remains to check  \eqref{Constraint3Eq}. 
This follows immediately from the strong convexity of $\vp\in C^{2,\alpha}(\overline{\O})$ 
(recall
\eqref{MAEq} and the fact that $f,g$ are positive),
 together with the following lemma.
Given a matrix $C=[c_{ij}]$, denote by 
$$
||C||=\max_{i,j}|c_{ij}|.
$$

\begin{lemma} 
\label{LindseyLemma}
Let 
$
\big \{ \{S_{i}^{(k)}\}_{i=1}^{M(k)}\big\}_{k\in\NN}
$
be a sequence of admissible and regular almost-triangulations of $\Omega$.
Then (recall  (\ref{HiEq})),
$$
\lim_k\max_{i\in\{1,\ldots,M(k)\}}
\Big\Vert 
H\Big(S^{(k)}_i,\Big\{
\nabla\vp(x_{i_0}^{(k)}),\ldots,
\nabla\vp(x_{i_n}^{(k)})
\Big\}\Big)
-\nabla^{2}\vp(x_{i_{0}}^{(k)})\Big\Vert
=0.$$
\end{lemma}

\bpf
 First, let 
$$
h=h(k):=\max_{i\in\{1,\ldots,M(k)\}}\mathrm{diam}\,S_{i}^{(k)}.
$$ 
By Definition \ref{admissibleATDef}, 
\begin{equation}
\label{hkEq}
\lim_kh=0.
\end{equation} 

Fix some $i\in\{1,\ldots,M(k)\}$. Then with $A_{i}$ and $B_{i}$ defined as in Definition \ref{DMAOPDef} (though
now dependent on $k$ although we omit that from the notation), notice that the $(j,l)$-th entry of $A_{i}\nabla^{2}\varphi\left(x_{i_{0}}\right)$
is $\left(x_{i_{j}}-x_{i_{0}}\right)^{T}\left(\nabla\partial_{l}\varphi\left(x_{i_{0}}\right)\right)$,
which is of course equal to $D_{v_{j}}\left(\partial_{l}\varphi\right)\left(x_{i_{0}}\right),$
where $D_{v}$ denotes the directional derivative in the direction
$v$ and where 
$$
v_{j}:=x_{i_{j}}-x_{i_{0}}, \q j=1,\ldots,n.
$$ 
Now $\eta_{i_{j}}=\nabla\varphi\left(x_{i_{j}}\right)$,
so $\eta_{i_{j}}-\eta_{i_{0}}=\nabla\varphi\left(x_{i_{j}}\right)-\nabla\varphi\left(x_{i_{0}}\right)$,
i.e., the $(j,l)$-th entry of $B_{i}$ is $\partial_{l}\varphi\left(x_{i_{j}}\right)-\partial_{l}\varphi\left(x_{i_{0}}\right)$.\\

Next, set 
$$
\zeta:=\partial_{l}\varphi, \;\; x:=x_{i_{0}},\;\; y:=x_{i_{j}},\;\;
\tau_{j}:=\Vert v_j \Vert=\Vert x-y\Vert,\;\;
u_{j}:=\tau_{j}^{-1}v_{j}.
$$
Note $u_j$
is of unit length and that $D_{v_{j}}\left(\partial_{l}\varphi\right)\left(x_{i_{0}}\right)=\tau_{j}D_{u_{j}}\zeta(x)$. We have,
\beq
\baeq
\left|\left[B_{i}\right]_{jl}-\left[A_{i}\nabla^{2}\varphi\left(x_{i_{0}}\right)\right]_{jl}\right| 
& =  \left|\zeta(y)-\zeta(x)-\tau_{j}D_{u_{j}}\zeta(x)\right|\nonumber\\
 & =  \left|\int_{0}^{\tau_{j}}
D_{u_{j}}\zeta\left(\frac{(\tau_{j}-t)x+ty}{\tau_j}\right)\, dt-\tau_{j}D_{u_{j}}\zeta(x)\right|\nonumber\\
 & =  \left|\int_{0}^{\tau_{j}}\left[
D_{u_{j}}\zeta\left(\frac{(\tau_{j}-t)x+ty}{\tau_j}\right)-D_{u_{j}}\zeta(x)\right]\, dt\right|\nonumber\\
 & \leq \int_{0}^{\tau_{j}}\left|
D_{u_{j}}\zeta\left(\frac{(\tau_{j}-t)x+ty}{\tau_j}\right)-D_{u_{j}}\zeta(x)\right|\, dt\nonumber\\ 
 & \leq \int_{0}^{\tau_{j}}C_1\left\Vert \left(
\frac{(\tau_{j}-t)x+ty}{\tau_j}\right)-x\right\Vert ^{\alpha}\, dt\nonumber\\
 & =  
 C_1\int_{0}^{\tau_{j}}t^{\alpha}\, dt
 =
 C\tau_j^{\alpha+1},
\lb{halphaEstEq}
\eaeq
\eeq
where $C_1=||\vp||_{C^{2,\alpha}(\overline{\O})}$ and $C={C_1}/(1+\alpha)$.
 Now, write 
$$
A_{i}=DU, 
$$
where 
$$
D:=\mathrm{diag}(\tau_{1},\ldots,\tau_{n}).
$$
Thus, the rows of $U $ have unit length.
By our last inequality,
\begin{eqnarray*}
\left|\left[D^{-1}B_{i}\right]_{jl}-\left[D^{-1}A_{i}\nabla^{2}\varphi\left(x_{i_{0}}\right)\right]_{jl}\right| =  \left|\tau_{j}^{-1}\left[B_{i}\right]_{jl}-\tau_{j}^{-1}\left[A_{i}\nabla^{2}\varphi\left(x_{i_{0}}\right)\right]_{jl}\right|
\leq C\tau_j^{\alpha} \le Ch^{\alpha},
\end{eqnarray*}
 where $C$ is independent of $k$, $i$, $j$, and $l$.
 
 Now $U^{-1}=\frac{1}{\det U}\left((-1)^{j+l}M_{jl}\right)^{T}$,
where $M_{jl}$ is the $(j,l)$-th minor of $U$. Since the 
rows of $U$ 
are unit vectors,  $\vert U_{jl}\vert\leq1$.
Since $M_{jl}$ is a polynomial of $(n-1)!$ terms in the $U_{jl}$,
we have that $\vert M_{jl}\vert\leq(n-1)!$ for all $j,l$, and hence
$\left|\left[U^{-1}\right]_{jl}\right|\leq\frac{(n-1)!}{\det U}$.
By Definition \ref{regularATDef}, $\det U$ is bounded below by a constant $R>0$ (independent
of $k$ and $i$), so we have that $\left|\left[U^{-1}\right]_{jl}\right|\leq R'$
for $R'=R^{-1}(n-1)!>0$ (independent of $k$ and $i$). Then it follows
that 
\begin{eqnarray*}
\left|\left[U^{-1}D^{-1}B_{i}\right]_{jl}-\left[U^{-1}D^{-1}A_{i}\nabla^{2}\varphi\left(x_{i_{0}}\right)\right]_{jl}\right| & = & \left|\left[U^{-1}\left(D^{-1}B_{i}-D^{-1}A_{i}\nabla^{2}\varphi\left(x_{i_{0}}\right)\right)\right]_{jl}\right|\\
 & \leq & nR'Ch^{\alpha},
\end{eqnarray*}
Of course, since $A_{i}=DU$, this means precisely
that 
$$
\max_{i=1,\ldots,M(k)}\left\Vert A_{i}^{-1}B_{i}-\nabla^{2}\varphi\left(x_{i_{0}}\right)\right\Vert \leq C'h^{\alpha},
$$
for some $C'>0$  independent of $k$ and $i$. 
Since $\nabla^{2}\varphi\left(x_{i_{0}}\right)$
is symmetric,  
\[
\max_{i=1,\ldots,M(k)}\left\Vert A_{i}^{-1}B_{i}-\nabla^{2}\varphi\left(x_{i_{0}}\right)\right\Vert=\max_{i=1,\ldots,M(k)}\left\Vert \left(A_{i}^{-1}B_{i}\right)^{T}-\nabla^{2}\varphi\left(x_{i_{0}}\right)\right\Vert.
\]
Thus,
\begin{equation*}
\baeq
\max_{i}\left\Vert \frac{1}{2}A_{i}^{-1}B_{i}+\frac{1}{2}\left(A_{i}^{-1}B_{i}\right)^{T}-\nabla^{2}\varphi\left(x_{i_{0}}\right)\right\Vert 
 & \leq  \max_{i}\left\Vert \frac{1}{2}A_{i}^{-1}B_{i}-\frac{1}{2}\nabla^{2}\varphi\left(x_{i_{0}}\right)\right\Vert
\cr
& \qq +\max_i\left\Vert \frac{1}{2}\left(A_{i}^{-1}B_{i}\right)^{T}-\frac{1}{2}\nabla^{2}\varphi\left(x_{i_{0}}\right)\right\Vert \\
 & =  \max_{i}\left\Vert A_{i}^{-1}B_{i}-\nabla^{2}\varphi\left(x_{i_{0}}\right)\right\Vert\le C'h^\alpha,
\eaeq
\end{equation*}
which, by \eqref{hkEq}, concludes the proof of Lemma \ref{LindseyLemma}.
\epf

\begin{remark} {\rm
\label{}
It is tempting to rephrase the regularity assumption (Definition \ref{regularATDef}) in terms of
eigenvalues instead of determinant, however the matrices $U$ and $A^{-1}B$ are not symmetric in general, and so the more involved argument we used seems to be necessary to prove Lemma \ref{LindseyLemma}.
} \end{remark}

\noindent (ii)
Given that the feasibility conditions \eqref{Constraint1Eq}--\eqref{Constraint3Eq}
hold, $d_k$ is well-defined. 
The rest of the proof is devoted to showing that
$d_k$ converges to zero.

Let
$$
y_{i}^{(k)}:=\frac1{n+1}\sum_{j=0}^{n}x_{i_{j}}^{(k)}
$$
denote the barycenter of $S_i$.
Since $f$ is uniformly continuous and bounded away from zero on $\overline{\Omega}$, 
\beq\label{sourceCenterEq}
\max_{i=1,\ldots,M(k)}\vert\log f(x_{i_{0}}^{(k)})
-\log f(y_{i}^{(k)})\vert\rightarrow0.
\eeq
Let
$z_{i}^{(k)}$ denote the barycenter of the simplex formed by 
the gradients at the vertices of the $i$-th simplex, i.e.,
$$
z_{i}^{(k)}:=\frac1{n+1}\sum_{j=0}^{n}\nabla\vp(x_{i_{j}}^{(k)}).
$$
Then similarly, since $g$ is uniformly continuous and bounded away from zero on $\overline{\Lambda}$ and $\nabla\vp$ is Lipschitz, 
\beq\label{loggLipEq}
\max_{i=1,\ldots,M(k)}\vert\log g(\nabla\vp(x_{i_{0}}^{(k)}))
-\log g(z_{i}^{(k)})\vert\rightarrow0,
\eeq
By \eqref{MAEq}, 
\beq\label{MAPointEq}
\det\nabla^{2}\vp(x_{i_{0}}^{(k)})=
\frac{f(x_{i_{0}}^{(k)})}{g(\nabla\vp(x_{i_{0}}^{(k)}))}.
\eeq
Then, by \eqref{dkDef}, \eqref{FiDef}, and \eqref{FDef}, we have
\begin{equation*}
\baeq
  & (\mathrm{vol}(\O))^{-1} d_k \\
&\leq\max_{i}
\Big\vert
\log
\det H\Big(S^{(k)}_i,\Big\{
\nabla\vp(x_{i_0}^{(k)}),\ldots,
\nabla\vp(x_{i_n}^{(k)})
\Big\}\Big)
-\log f (y_{i}^{(k)})+\log
g(z_{i}^{(k)})\Big\vert\\
 &\le  \max_{i}\Big\vert\log\det H\Big(S^{(k)}_i,\Big\{
\nabla\vp(x_{i_0}^{(k)}),\ldots,
\nabla\vp(x_{i_n}^{(k)})
\Big\}\Big)-\log f (x_{i_{0}}^{(k)})
                 +\log g(\nabla\vp(x_{i_{0}}^{(k)}))\Big\vert
\\
&\qq+\max_{i}\big\vert\log f (x_{i_{0}}^{(k)})-\log f (y_{i}^{(k)})\big\vert
+\max_{i}\big\vert\log g(\nabla\vp(x_{i_{0}}^{(k)}))-\log g(z_{i}^{(k)})\big\vert.\\
\eaeq
\end{equation*}
The last term tends to zero with $k$ by \eqref{loggLipEq}, while the second
does so by \eqref{sourceCenterEq}. Finally, the first term tends to zero with $k$
by Lemma \ref{LindseyLemma} and \eqref{MAPointEq} (note here that since
$\nabla^{2}\vp(\overline{\Omega})$ is compact
and entirely contained in the set of positive definite matrices, 
$$
\det H\Big(S^{(k)}_i,\Big\{
\nabla\vp(x_{i_0}^{(k)}),\ldots,
\nabla\vp(x_{i_n}^{(k)})
\Big\}\Big)
$$
is bounded away from zero for all $k\gg 1$ by Lemma \ref{LindseyLemma}).
\end{proof}

\section{Proof of the convergence theorem}

We now turn to the proof of Theorem \ref{MainThm}, stating that the 
potentials $\phi^{(k)}$   \eqref{phikEq} converge to 
the Brenier potential $\vp$. This section is organized as follows. In 
\S\ref{BaryCSubSec}--\S\ref{MotivSubSec}  we define the barycentric extension of the gradient of the optimization potentials, and show how this relates to the discrete Jacobian on each simplex (Lemma \ref{HikgkLemma}). This sets the stage for the remainder of the proof which occupies the rest of this lengthy section. In 
\S\ref{strategySubSec}  we describe the strategy for the proof. 
The proof itself occupies \S\ref{FirstOrderSubSec}--\S\ref{StabSubSec}. 

Let $D$ be a closed ball such that 
\beq
\label{Ddef}
\mathrm{int}\,D \supset \overline{\O},
\eeq
where $\h{\rm int}\, A$ denotes the interior of a set $A$.
By the Arzel\`a--Ascoli theorem, 
since $\{\phi^{(k)}\}_k$ is an equicontinuous, uniformly bounded family
(recall \eqref{ajEq}--\eqref{phiEq} and note that
  $\eta_{j}^{(k)}\in\overline{\Lambda}$ for all $k,j$, with $\L$ bounded,
and $\phi^{(k)}(0)=0$) it has a uniformly converging subsequence.
Thus, to prove Theorem \ref{MainThm} it suffices to show that every subsequence of
$\phi^{(k)}$ that converges uniformly on $D$ converges to $\vp$ 
on $\overline{\Omega}$.

Thus, assume that 
\begin{equation}
\begin{aligned}
\label{phiphikEq}
\phi^{(k)}\rightarrow\phi \h{\ uniformly\ on }D
\end{aligned}
\end{equation}
 for
some $\phi$, and we need only show that $\phi=\vp$ on $\Omega$. Notice
that $\phi$ is convex and continuous as a uniform limit of continuous
uniformly bounded 
convex functions.

\subsection{Barycentric extension
of the gradient of the optimization potentials}
\label{BaryCSubSec}

 The objective function of the (L)DMAOP provides
us with some sort of control over the `second-order properties' of
the $\phi^{(k)}$, but these properties are neither well-defined at
this stage nor readily accessible because the $\phi^{(k)}$ are piecewise
linear
and so only $C^{0,1}$ and no better. 
In order to get a handle on the `second-order convergence'
of the $\phi^{(k)}$, we will replace the
piecewise constant but discontinuous
 subdifferentials of $\phi^{(k)}$ with
continuous, piecewise-affine functions that interpolate rather than
jump, which we may then differentiate once again. 

For the remainder of the article, let
\begin{equation}
\begin{aligned}
\label{kDMAOPEq}
\psi^{(k)}_1,\ldots,\psi^{(k)}_{N(k)}\in\RR \h{\ \ and \ \ } \eta^{(k)}_1,\ldots,\eta^{(k)}_{N(k)}
\in\overline{\L}\subset\RR^n,
\end{aligned}
\end{equation}
denote the solution of the $k$-th (L)DMAOP (Definition \ref{DMAOPDef}) 
associated to the data 
$$
(\O,\L,f,g,\{x_j^{(k)}\}_{j=1}^{N(k)},\{S_i^{(k)}\}_{i=1}^{M(k)}).
$$
Thus, with \eqref{ApproxGradPhiGridEq} in mind, we define 
a vector-valued function $G^{(k)}$
by barycentrically interpolating the values $\{\eta_{i_{j}}^{(k)}\}_{j=1}^n$
over the $i$-th simplex $S_i^{(k)}$, for all $i=1,\ldots,M(k)$. 
Namely,
for each $x$ in 
$$
S^{(k)}_i=\h{\rm co}(x^{(k)}_{i_0},...x^{(k)}_{i_n}),
$$
write 
\begin{equation}
\begin{aligned}
\label{xEq}
x=\sum_{j=0}^n\sigma_j x^{(k)}_{i_j},
\end{aligned}
\end{equation}
with $\sigma_j\in[0,1]$. Then, 
\beq\label{GkEq}
G^{(k)}(x):=\sum_{j=0}^n
\sigma_j \eta_{i_{j}}^{(k)}, \qq \h{\ if \ } x\in S_i^{(k)}
\eeq
(note that this is well-defined also for $x$ lying in more than one
simplex).
Alternatively, 
$G^{(k)}$ is the unique vector-valued function that is affine  on
each simplex in the $k$-th almost-triangulation and satisfies $G^{(k)}(x^{(k)}_i) = \eta^{(k)}_i$ 
for all $i=1,\ldots,N(k)$.

\subsection{The motivation for defining the barycentric extension}
\label{MotivSubSec}
Next, we explain the main role the functions $G^{(k)}$ play.

Let 
\begin{equation}
\begin{aligned}
\label{ikEq}
i^{(k)}:\bigcup_{i=1}^{M(k)} \mathrm{int}\, S^{(k)}_{i} \ra \{1,\ldots,M(k)\},
\end{aligned}
\end{equation}
denote the map assigning to a point the index of the unique simplex 
in the $k$-th almost-triangulation containing it, i.e., 
$i^{(k)}\big(\mathrm{int}\, S^{(k)}_{j}\big)=j$.
Define a (locally constant) matrix-valued function
$$
\calH^{(k)}:\bigcup_{i=1}^{M(k)} \mathrm{int}\, S^{(k)}_{i} \ra\h{\rm Sym}^2(\RR^n),
$$
by
\begin{equation}
\begin{aligned}
\label{HkEq}
\calH^{(k)}(x):=H_{i^{(k)}(x)}^{(k)},
\end{aligned}
\end{equation}
where (recall  \eqref{HiEq})
$$
H_{j}^{(k)}:=
H\Big(S^{(k)}_j,\Big\{
\eta_{j_0}^{(k)},\ldots,
\eta_{j_n}^{(k)}
\Big\}\Big).
$$
Define also,
\begin{equation}
\begin{aligned}
\label{taukEq}
\tau^{(k)}(x):=\h{ the barycenter of the simplex $S^{(k)}_{i^{(k)}(x)}$}, 
\end{aligned}
\end{equation}
and 
\begin{equation}
\begin{aligned}
\label{gammakEq}
\gamma^{(k)}(x):=
\h{ the mean of the $\eta_{j}^{(k)}$ associated to the vertices of simplex
$S^{(k)}_{i^{(k)}(x)}$}.
\end{aligned}
\end{equation}
Finally, recalling \eqref{FiDef}, we define a (locally constant) per-simplex penalty function 
\beq
\mathcal{C}^{(k)}(x)
:=F_{i^{(k)}(x)}^{(k)}
\Big( \big\{\psi_{j}^{(k)},\eta_{j}^{(k)}\big\}_{i=1}^{N(k)} \Big).
\eeq 
By \eqref{FiDef},
\beq
\lb{CkHk}
\mathcal{C}^{(k)}(x) = \max\Big\{0,-\log\det \calH^{(k)}(x)-\log g(\gamma^{(k)}(x))+\log f (\tau^{(k)}(x)) \Big\},
\eeq
By the definition of the optimal cost \eqref{ckEq}, 
\begin{equation}
\begin{aligned}
c_k & = \int_{\bigcup_{i=1}^{M(k)} S_i} \mathcal{C}^{(k)}(x) \,dx \\
& = \int_{\bigcup_{i=1}^{M(k)} S_i} \max\Big\{0,-\log\det \calH^{(k)}(x)-\log g(\gamma^{(k)}(x))+\log f (\tau^{(k)}(x)) \Big\}\,dx.
\label{ckInt}
\end{aligned}
\end{equation}

The following result is the motivation for introducing the functions
$G^{(k)}$. 
When combined with \eqref{ckInt}, it relates second-order information that we can extract from $\phi^{(k)}$ (via $G^{(k)}$) with the cost $c_k$, over which we have control by the assumptions of Theorem \ref{MainThm}. In fact, we have $c_k\ra0$, so we can hope that in some sense, as $k$ becomes large, $\phi^{(k)}$ approaches a subsolution of the Monge-Amp\`ere equation.

\begin{lemma}
\label{HikgkLemma}
For  $x\in \bigcup_{i} \mathrm{int}\, S^{(k)}_{i}$,
$
\calH^{(k)}(x)
=\nabla G^{(k)}(x)+\left(\nabla G^{(k)}(x)\right)^T
$.
\end{lemma}

\bpf
We fix some $k$ and then omit $k$ from our notation
in the remainder of the proof. We also fix $i\in\{1,\ldots,M(k)\}$ 
and work within the simplex $S_{i}$, i.e., assume that
$x\in S_{i}$, i.e., $i^{(k)}(x)=i$. Now let $v_{j}=x_{i_{j}}-x_{i_{0}}$. 
We claim that
$$
D_{v_{j}}G(x)=\eta_{i_{j}}-\eta_{i_{0}}, \q x\in\h{\rm int}\, S_{i}.
$$
Intuitively, this is because $g$ is affine on $S_{i}$ with 
$G\left(x_{i_{j}}\right)=\eta_{i_{j}}$. 
For the proof, recall the definition of the functions $\sigma_j$
from \eqref{xEq}. Then, letting $\delta_{st}=1$ if $s=t$ and zero otherwise,
$$
\baeq
D_{v_{j}}G(x)
&=
\frac{d}{dt}\Big|_{t=0} G\big(x+t(x_{i_{j}}-x_{i_{0}})\big)
\cr
&=
\sum_{s=0}^n\frac{d}{dt}\Big|_{t=0} \sigma_s\big(x+t(x_{i_{j}}-x_{i_{0}})\big)\eta_{i_{s}}
\cr
&=
\sum_{s=0}^n\frac{d}{dt}\Big|_{t=0} 
\big(\sigma_s(x)+t(\delta_{js}-\delta_{0s})\big)\eta_{i_{s}}
\cr
&=
\eta_{i_{j}}-\eta_{i_{0}},
\eaeq
$$
as claimed.

Now $D_{v_{j}}G=v_{j}\cdot\nabla G$, so $D_{v_{j}}G$
is the $j$-th row of $A_{i}\nabla G$, where $\nabla G$
denotes the matrix with $j$-th row $\frac{\partial}{\partial x_{j}}G$
and where $A_{i}$ is as in Definition \ref{DMAOPDef}.
Since $D_{v_{j}}G=\eta_{i_{j}}-\eta_{i_{0}}$ is also the $j$-th
row of $B_{i}$, we have that $B_{i}=A_{i}\nabla G$, i.e.,
$\nabla G=A_{i}^{-1}B_{i}$. The statement now follows from the definition of $H_i$ \eqref{HiEq}.
\epf

\subsection{Strategy for the proof}
\label{strategySubSec}

In this subsection we outline the strategy for the proof of Theorem \ref{MainThm}. 

The results of the previous subsection indicate that the optimization potentials should be approximate subsolutions of the Monge--Amp\`ere equation. 
Since the optimization potentials converge to $\phi$, this gives some hope that $\phi$ itself might be such a subsolution.
To make this rigorous we regularize.
Let $\xi_{\varepsilon}$ be a standard set of mollifiers (supported
on $B_{\varepsilon}(0)$). 
Notice that $G^{(k)}$ and $\calH^{(k)}$ are only defined
on the almost-triangulation of $\Omega$, so we run into trouble near the
boundary when convolving with $\xi_{\varepsilon}$. Thus,
we will work with the regions $\O_\eps$ given by
Definition \ref{admissibleATDef}.

\begin{lemma}
\label{symmetrizedLemma}
Fix $\varepsilon>0$.
As $k\ra \infty$,
$
\calH^{(k)}\star\xi_{\varepsilon}(x)
$
converges uniformly to $\nabla^{2}(\phi\star\xi_{\varepsilon})$ on  $\O_\varepsilon$.
\end{lemma}

The proof of Lemma \ref{symmetrizedLemma} takes place in \S\ref{SecondOrderSubSec}.
Lemma \ref{symmetrizedLemma} gives us control on the second-order behavior of
$\phi\star\xi_{\varepsilon}$. 
The proof uses an auxillary result established in \S\ref{FirstOrderSubSec}
that gives control over the first-order behavior of $\phi\star\xi_{\varepsilon}$.

The next step of the proof involves 
taking the limits in $k $ both in the previous lemma and in 
 \eqref{ckInt}. Thanks to the fact that $c_k\ra 0$ this yields 
the following statement roughly saying that $\phi\star\xi_{\varepsilon}$ is an approximate subsolution
to the \MA equation, i.e., that 
$\nabla(\phi\star\xi_{\varepsilon})$ cannot `excessively' shrink volume.

\begin{lemma}
\label{kLimLemma}
Fix $\varepsilon>0$.
For $x\in\Omega_{\varepsilon}$,
\[
\det\nabla^{2}(\phi\star\xi_{\varepsilon})(x)\geq\frac{\inf\{ f (y):y\in B_{\varepsilon}(x)\}}{\sup\{g\left(\nabla\phi(y)\right):y\in B_{\varepsilon}(x),\nabla\phi(y) \,\mathrm{exists}\}}.
\]
\end{lemma}

The proof of Lemma  \ref{kLimLemma} is presented in \S\ref{DensitySubSec}.

The next step in the proof is to take the limit  $\varepsilon\ra0$ and show 
that $\phi$ must be a weak solution in the sense that $\nabla\phi $ pushes forward $\mu$ to $\nu$. 
The proof of this fact breaks up into several steps.
It spreads over \S \ref{EpsilonSubSec}--\S \ref{PropSubSec}.  First, we define the measures
\begin{equation}
\begin{aligned}
\label{nuepsEq}
\nu_{\varepsilon}:=(\nabla \phi\star\xi_\ve)_{\#}\mu|_{\Omega_{\varepsilon}}
\end{aligned}
\end{equation}
obtained by pushing forward the restriction of $\mu$ to $\Omega_{\varepsilon}$ by 
$\nabla \phi\star\xi_\ve$. Denote the density of these measures by
\begin{equation}
\begin{aligned}
\label{gepsEq}
g_\ve \, dx:=\nu_{\varepsilon}.
\end{aligned}
\end{equation} 
Using Lemma \ref{kLimLemma}, we show that a subsequence of these measures 
(roughly speaking) converges weakly to the target measure $\nu$. Intuitively speaking, Lemma  \ref{kLimLemma} says that $\nabla \phi\star\xi_\ve$ does not shrink volume `excessively' at any point. 
Combining this with the fact that the image of $\nabla \phi\star\xi_\ve$ must lie within $\overline{\Lambda}$ 
motivates the convergence. The precise result we prove is the following.

\begin{prop}
\label{measProp}
For any sequence $\varepsilon\rightarrow0$, $\mu(\Omega_{\varepsilon})^{-1}\nu_{\varepsilon}$ is a sequence of probability measures converging weakly to $\nu(\Lambda)^{-1}\nu$.
\end{prop}

The proof of Proposition \ref{measProp} is given in
\S \ref{PropSubSec} based on some auxiliary results proven in  \S \ref{EpsilonSubSec}.

The last step of the proof of Theorem \ref{MainThm} 
is to show that any uniform limit of the optimization potentials coincides with the Brenier potential $\vp$. 

\begin{lemma}
\label{FinalLemma}
Let $\phi^{(k)}$ be defined by   \eqref{phikEq} and suppose that $\phi^{(k)}$ converges uniformly to some $\phi$. Then,  $\phi=\vp$.
\end{lemma}

The proof of Lemma  \ref{FinalLemma} is presented in \S\ref{FinalSubSec}. 
It hinges on Proposition \ref{measProp}, stability results for optimal transport maps
(proved in \S\ref{StabSubSec}),
and all of the previous steps in the proof.

\begin{remark} {\rm
\label{}
Though it seems natural that the stability of optimal transport plays a role in this proof, it is perhaps unexpected that we have employed the stability of optimal transport to obtain convergence in $\varepsilon$ (rather than in $k$). As mentioned earlier, we could not take the seemingly more direct route and needed to use mollifiers to obtain regularity.
} \end{remark}

\subsection{First order control on $\phi$}
 \label{FirstOrderSubSec} 

We want to show that $G^{(k)}$ approaches $\partial\phi$ in some
sense. We make use of the following semi-continuity result of Bagh--Wets  \cite[Theorem 8.3]{BaghW}
(cf. \cite[Theorem 24.5]{Rock}).
Recall that $f_k$ epi-converges to $f$ (roughly) if the epigraphs of $f_k$ converge to the epigraph of $f$;
we refer to \cite [p. 240]{RW}  for more precise details.

\begin{theorem}
\label{BWThm}
Let $f$ and $\{f_k\}_{k\in\NN}$ be lower semicontinuous convex functions with $f_k $ 
epi-converging to $f $. Fix $x\in\h{\rm int}\,\h{\rm dom}\, f$ and $\eps>0$. 
Then, there exists $\delta>0$ and $K\in\NN$ such that
$$
\del f_k(y)\subset \del f(x)+B_\eps(0), \q
\forall y\in B_\delta(x), \, \forall k\ge K.
$$
Moreover, if $f$ is differentiable at $x$ then
\begin{equation}
\begin{aligned}
\label{BWEq}
\lim_{k\ra\infty} \del f_k(x)=\{\nabla f(x)\}.
\end{aligned}
\end{equation}
\end{theorem}

\begin{lemma}
 \label{GkLimitLemma} 
As $k$ tends to infinity, $G^{(k)}$ converges to $\nabla\phi$ almost everywhere on $\O$.
\end{lemma}

\bpf 
If a sequence of lower semicontinuous finite convex functions converges
uniformly on bounded sets to some convex function, then the sequence
epi-converges to this function \cite[Theorem 7.17]{RW}.
Thus, we may apply Theorem \ref{BWThm} to $\phi^{(k)}$. 
Fix $x\in\O$ and $\epsilon>0$. There exists a $\delta>0$ and $K\in\NN$ such that
$$
\partial\phi^{(k)}(y)\subset\partial\phi(x)+B_{\epsilon}(0),  \q
\forall y\in B_\delta(x), \, \forall k\ge K.
$$
Fix $\epsilon>0$ and a point $x\in\O$ where $\phi$ is differentiable. 
Additionally, take $\delta>0$
and $K$ according to the aforementioned result. 
If necessary, take $K$ even larger, so that for all $k\ge K$
the maximal distance of $x$ to the vertices of the simplices containing
it is at most $\delta$. We assume from now on that $k\ge K$. 
Thus for all vertices $x_{j}^{(k)}$ of any
simplex containing $x$, we have that 
$$
\partial\phi^{(k)}(x_{j}^{(k)})\subset\nabla\phi(x)+B_{\varepsilon}(0).
$$
By   \eqref{ApproxGradPhiGridEq},
$\eta_{i_{j}}^{(k)}\in\partial\phi^{(k)}(x_{i_{j}}^{(k)})$.
On the other hand, by  \eqref{GkEq}, $G^{(k)}(x)$
is a convex combination of the $\eta_{i_{j}}^{(k)}$.
Thus, $G^{(k)}(x)\in\nabla\phi(x)+B_{\varepsilon}(0)$.
This proves that $G^{(k)}\rightarrow\nabla\phi$
almost everywhere since $\phi$ is differentiable almost everywhere.
\epf

\subsection{Second order control on $\phi$ and a proof of Lemma \ref{symmetrizedLemma}}
 \label{SecondOrderSubSec} 
Unfortunately, we do not have enough regularity to maintain that $\nabla G^{(k)}$
approaches $\nabla^{2}\phi$ almost everywhere. We can obtain this
regularity by convolving everything with a sequence of mollifiers.

The motivation for doing so is fairly intuitive. Strictly speaking, the second-order behavior of the $\phi^{(k)}$ is completely trivial. The second-derivatives of the $\phi^{(k)}$ are everywhere either zero or undefined. However, by virtue of solving the (L)DMAOP, the $\phi^{(k)}$ do actually contain second-order information in some sense. Indeed, we may think of the graphs of the $\phi^{(k)}$ as having some sort of curvature that becomes apparent when we `blur' $\phi^{(k)}$ on a small scale and then take $k$ large enough so that the scale of the discretization is much smaller than the scale of the blurring. This blurring is achieved by convolving with smooth mollifiers.

Let $\xi_{\varepsilon}$ be a standard set of mollifiers (supported
on $B_{\varepsilon}(0)$). Notice that $G^{(k)}$ is only defined
on the almost-triangulation of $\Omega$, so we run into trouble near the
boundary when convolving with $\xi_{\varepsilon}$. Thus,
we will work with the regions $\O_\eps$ given by
Definition \ref{admissibleATDef}.

The main result of this subsection is:

\begin{lemma}\label{FirstGkLemma}
Fix $\varepsilon>0$. On $\O_\varepsilon$,
$\nabla G^{(k)} \star\xi_{\varepsilon}$
converges uniformly to $\nabla^{2}(\phi\star\xi_{\varepsilon})$
(in each of the $n^2$ components).
\end{lemma}

Lemma  \ref{FirstGkLemma} immediately implies Lemma \ref{symmetrizedLemma}
thanks to Lemma \ref{HikgkLemma} and symmetrization
(noting the symmetry of $\nabla^{2}(\phi\star\xi_{\varepsilon})$).

We start with three auxiliary results.
The first states that 
differentiation and convolution commute when the functions involved are uniformly Lipschitz.
We leave the standard proof to the reader.
Note
that $\nabla\phi\star\xi_{\varepsilon}:=
(\nabla\phi)\star\xi_{\varepsilon}$ is everywhere defined because
$\nabla\phi$ exists almost everywhere.

\begin{claim}
\label{phiConvClaim}
For all $\ve>0$,
 $\nabla\phi\star\xi_{\varepsilon}=\nabla(\phi\star\xi_{\varepsilon})$ and  
$\nabla G^{(k)}\star \xi_{\varepsilon}=\nabla (G^{(k)} \star \xi_{\varepsilon})$.
\end{claim}

\noindent
The second is a mollified version of Lemma \ref{GkLimitLemma}.

\begin{claim}\label{SecondGkClaim}
Fix $\varepsilon>0$. For $x\in\O_\varepsilon$,
$G^{(k)}\star\xi_{\varepsilon}$ converges uniformly to $\nabla\phi\star\xi_{\varepsilon}$
 (in each of the $n$ components).
\end{claim}

\bpf First, we claim pointwise convergence, i.e., that 
\begin{equation}
\begin{aligned}
\label{GkptwiseEq}
\lim_{k\ra\infty}G^{(k)}\star\xi_{\varepsilon}(x)
=\nabla\phi\star\xi_{\varepsilon}(x) \q
\h{ \rm for each $x\in\Omega_{\varepsilon}$}.
\end{aligned}
\end{equation}
  To check that this is true,
note that for $x\in\Omega_{\varepsilon}$, 
$$
\baeq
\Big| 
G^{(k)}\star\xi_{\varepsilon}(x)-\nabla\phi\star\xi_{\varepsilon}(x)
\Big|
&=
\Big|
\int\big(G^{(k)}(y)-\nabla\phi(y)\big)\xi_{\varepsilon}(y-x)\, dy
\Big|
\cr
&\leq
\int\big| G^{(k)}(y)-\nabla\phi(y)\big|\xi_{\varepsilon}(y-x)\, dy
\eaeq
$$
(once again note that these integrals make sense since
$\nabla\phi$ exists almost everywhere).
 Since 
$\vert G^{(k)}(y)-\nabla\phi(y)\vert\rightarrow0$
a.e. by Lemma \ref{GkLimitLemma}, while $|\xi_{\varepsilon}|\le C$ 
and the $G^{(k)}$ are uniformly bounded, equation   \eqref{GkptwiseEq} follows from
bounded convergence (note that $\varepsilon$ is constant in this
limit). 

Next, notice that the $G^{(k)}$ are uniformly bounded independently of $k$,
in fact (since we may assume, without loss of generality, that $0\in\L$),
\begin{equation}
\begin{aligned}
\label{GkBoundEq}
|G^{(k)}|\le \h{\rm diam}\,\L
\end{aligned}
\end{equation}
by \eqref{GkEq}.
Write
\begin{equation}
\begin{aligned}
\label{GkiEq}
G^{(k)}=(G_1^{(k)},\ldots,G_n^{(k)}).
\end{aligned}
\end{equation}
Thus, $G_i^{(k)}\star\xi_{\varepsilon}, i=1,\ldots,n,$ are uniformly
bounded independently of $k$. 
$$
\baeq
|\nabla(G_i^{(k)}\star\xi_{\varepsilon})|
&=
\vert G_{i}^{(k)}\star\nabla\xi_{\varepsilon}(x)
\vert
\cr
&=\Big|\int G_{i}^{(k)}(y)\nabla\xi_{\varepsilon}(x-y)\, dy\Big|
\cr
&\leq \h{\rm diam}\,\L\Big|\int\nabla\xi_{\varepsilon}(x-y)\, dy\Big|
\le C,
\eaeq
$$
with $C=C(\L,\varepsilon)$. Thus, the $G^{(k)}\star\xi_{\varepsilon}$ have uniformly
bounded derivatives (in each component).
The statement of the claim now follows from 
Remark \ref{uniformRemark} below.
\epf

\begin{remark} {\rm
\label{uniformRemark}
We will use the following fact more than once.
If a uniformly bounded sequence of differentiable functions 
$\{f_{k}\}_{k\in\NN}$
with uniformly bounded (first) derivatives on compact sets satisfies $f_{k}\rightarrow f$
pointwise, then $f_{k}\rightarrow f$ uniformly  on compact sets. (This can be established easily using the Arzel\`a--Ascoli theorem.)
} \end{remark}

The third auxiliary result is a one-variable interpolation-type result. 
\begin{claim}
\label{KolmogClaim}
Let $I\subset\RR$ be a closed interval, and let 
  $\{f_{k}\}_{k\in\NN},f:I\rightarrow\mathbb{R}$ be
smooth functions such that (i) $f_{k}\rightarrow f$
uniformly, (ii) the $f_{k}''$ are uniformly bounded independently of  $k$, and (iii) $f''$ is bounded.
Then, $f_{k}'\rightarrow f'$ uniformly. 

\end{claim}

\begin{proof}
We make use of the Landau--Kolmogorov inequality 
\begin{equation}
\begin{aligned}
\label{LKInEq}
\Vert g'\Vert_{\infty}\leq 
C\Vert g\Vert_{\infty}^{1/2}\Vert g''\Vert_{\infty}^{1/2}
\end{aligned}
\end{equation}
for
$g\in C^2(I)$
(see, e.g., \cite{chuiLandau}).
Apply   \eqref{LKInEq} to $g=f_{k}-f$ to obtain 
$\Vert f_{k}'-f'\Vert_{\infty}
\leq C\Vert f_{k}-f\Vert_{\infty}^{1/2}\Vert f_{k}''-f''\Vert_{\infty}^{1/2}$.
Since $f''$ is bounded and the $f_{k}''$ are uniformly bounded,
we have that $\Vert f_{k}''-f''\Vert_{\infty}^{\frac{1}{2}}$ is uniformly
bounded in $k$. Of course, $\Vert f_{k}-f\Vert_{\infty}\rightarrow0$
by uniform convergence.
Therefore, $\Vert f_{k}'-f'\Vert_{\infty}\rightarrow0$,
as desired.
\end{proof}

\bpf[Proof of Lemma \ref{FirstGkLemma}]
The functions (recall  \eqref{GkiEq}) 
$$
\partial_{j}^{2}(G_{i}^{(k)}\star\xi_{\varepsilon})=G_{i}^{(k)}\star\partial_{j}^{2}\xi_{\varepsilon},
\q
i,j\in\{1,\ldots,n\}, 
$$
are uniformly bounded independently of $k$ by a constant depending on $\varepsilon$
(by the uniform boundedness of $G_{i}^{(k)}$---recall  \eqref{GkBoundEq}). 
By Claim \ref{phiConvClaim},
 $\nabla\phi\star\xi_{\varepsilon}=\nabla(\phi\star\xi_{\varepsilon})$
so also $\nabla(\nabla\phi\star\xi_{\varepsilon})=\nabla^{2}(\phi\star\xi_{\varepsilon})$.
Thus, since $\phi\star\xi_{\varepsilon}$ is smooth,
$\nabla(\nabla\phi\star\xi_{\varepsilon})$ is bounded in all of its $n^2$
components.

Let $x\in\Omega_{\varepsilon}$. Then fix $i,j$ and let $\delta>0$ be small enough such that
$I:=\{ x+t e_{j}\,:\,t\in[-\delta,\delta]\} \subset\Omega_{\varepsilon}$.
By Claim \ref{SecondGkClaim}, $G_{i}^{(k)}\star\xi_{\varepsilon}\rightarrow\partial_{i}\phi\star\xi_{\varepsilon}$
uniformly. Restricting to the $j$-th variable and applying Claim \ref{KolmogClaim},
$$
\partial_{j}(G_{i}^{(k)}\star\xi_{\varepsilon})\rightarrow\partial_{j}(\partial_{i}\phi\star\xi_{\varepsilon})
$$ 
uniformly on $I\ni x$. Since $x$, $i$, and $j$ were arbitrary, we see that 
$$
\nabla(G^{(k)}\star\xi_{\varepsilon})
\rightarrow\nabla(\nabla\phi\star\xi_{\varepsilon})
$$
pointwise (though we cannot yet say that this convergence is uniform). 
The uniformity of the convergence now follows from Remark \ref{uniformRemark}.
Finally, invoking Claim \ref{phiConvClaim} implies the statement of Lemma \ref{FirstGkLemma}.
\epf

\subsection{Obtaining a density inequality for $\phi\star\xi_{\varepsilon}$
and a proof of Lemma \ref{kLimLemma}}
 \label{DensitySubSec} 

In this subsection we prove Lemma \ref{kLimLemma}.

First, we prove a mollified version of \eqref{ckInt}.

 \begin{lemma}
 \label{kDensIneqLemma}
Fix $\varepsilon>0$. For $k$ sufficiently large,
$$
\int_{\O_\ve} \max\Big\{0,-\log\det \big(\calH^{(k)}\star \xi_\ve\big)
+\log\frac{f \circ \tau^{(k)}}{g\circ\gamma^{(k)}}\star\xi_\ve \Big\} \leq c_k,
$$
where $c_k$ is the optimal cost defined in \eqref{ckEq},
and $\tau^{(k)}$ and $\gamma^{(k)}$ are  defined in
 \eqref{taukEq}--\eqref{gammakEq}.
\end{lemma}

\bpf

Note that $\max(0,\cdot)$ is convex, so applying Jensen's inequality to \eqref{CkHk} yields
$$
\baeq
\mathcal{C}^{(k)}\star \xi_\ve
& = \max\Big\{0,-\log\det \calH^{(k)}-\log (g\circ\gamma^{(k)})+\log (f \circ \tau^{(k)}) \Big\} \star \xi_\ve
\cr
& \geq
\max\Big\{0,\Big[-\log\det \calH^{(k)}-\log (g\circ\gamma^{(k)})+\log (f \circ \tau^{(k)}) \Big] \star \xi_\ve\Big\}
\cr
& = \max\Big\{0,\big[-\log\det \calH^{(k)}\big]\star \xi_\ve +\log\frac{f \circ \tau^{(k)}}{g\circ\gamma^{(k)}}\star\xi_\ve \Big\}.
\eaeq
$$
Now by the convexity of $-\log\det(\cdot)$ on the set of positive semidefinite matrices and Jensen's inequality once again, we have 
$$
\big[-\log\det \calH^{(k)}\big]\star \xi_\ve \geq -\log\det \big(\calH^{(k)}\star \xi_\ve\big),
$$
and combining the last two inequalities yields 
\beq
\lb{CkConvIneq}
\mathcal{C}^{(k)}\star \xi_\ve \geq 
\max\Big\{0,-\log\det \big(\calH^{(k)}\star \xi_\ve\big)
+\log\frac{f \circ \tau^{(k)}}{g\circ\gamma^{(k)}}\star\xi_\ve \Big\}.
\eeq
Recall from Definition \ref{admissibleATDef} that
$$
\Omega_{\varepsilon}+B_\eps(0)
\subset \bigcup_{i=1}^{M(k)}
S_i^{(k)}, \q \forall k\gg 1.
$$
Thus for such $k$ sufficiently large, noting that $\mathcal{C}^{(k)} \geq 0$, we have 
$$
\baeq
\int_{\O_\ve} \mathcal{C}^{(k)}\star \xi_\ve
& = \int_{\O_\ve} \int_{B_\ve (0)} \mathcal{C}^{(k)}(x-y)\xi_\ve(y)\,dy\,dx
\cr
& = \int_{B_\ve (0)} \xi_\ve(y) \int_{\O_\ve}  \mathcal{C}^{(k)}(x-y)\,dx\,dy
\cr
& \leq \int_{B_\ve (0)} \xi_\ve(y) \int_{\bigcup_{i=1}^{M(k)}S_i^{(k)}}  \mathcal{C}^{(k)}(x)\,dx\,dy
\cr
& = c_k,
\eaeq
$$
where the last line follows from \eqref{ckInt}. Combining with \eqref{CkConvIneq} completes the proof.
\epf

At least intuitively, in order to prove Lemma \ref{kLimLemma} we need to `take the limit as $k\rightarrow\infty$' in  Lemma \ref{kDensIneqLemma}
so that we can employ Lemma \ref{symmetrizedLemma}. 
The main technical obstacle in taking the limit in $k$ is controlling 
the behavior of $\tau^{(k)}$ and $\gamma^{(k)}$; the proof relies on
tools from convex analysis.

\begin{remark}
 \label{fguniformRemark} 
Notice that in the case that $f$ and $g$ are uniform densities on $\Omega$ and $\Lambda$, respectively, the proof of Lemma \ref{kLimLemma} is almost trivial. 
Even in the case that only $g$ is uniform, the proof is 
still considerably easier. This is true because the most difficult part of the proof is controlling the behavior of $\gamma^{(k)}$, which requires results from convex analysis, most crucially a result on the `locally uniform' convergence of the subdifferentials of a sequence of convergent convex functions.
\end{remark}

\begin{proof}[Proof of Lemma \ref{kLimLemma}]

Let $\alpha,\beta>0$ and fix $x\in\O_\ve$. Using Theorem \ref{BWThm},
for every $z\in\overline{B}_{\alpha}(x)$,
there exists $\delta(z)>0$ and $N_{\beta}(z)$ such that
\begin{equation}
\begin{aligned}
\label{}
 \partial\phi^{(k)}(y)\subset\partial\phi(z)+B_{\beta}(0),
\q
\forall y\in B_{\delta(z)}(z),\; \forall k\geq N_{\beta}(z).
\end{aligned}
\end{equation}
By compactness,
there exist $z_{1},\ldots,z_{p}\in\overline{B}_{\alpha}(x)$ such that the $B_{\delta(z_{i})}$
cover $\overline{B}_{\alpha}(x)$.
Setting
$$
N'_{x,\alpha,\beta}:=\max_{i\in\{1,\ldots,p\}}N_{\beta}(z_{i}),
$$
we thus have that 
\beq
\label{nbhdEq}
\partial\phi^{(k)}(y)\subset
\bigcup_{z\in\overline{B}_{\alpha}(x)}\partial\phi(z) +B_{\beta}(0),
\q
\forall k\geq N'_{x,\alpha,\beta},\;
\forall y\in\overline{B}_{\alpha}(x).
\eeq

For $k$ sufficiently large, i.e., 
$$
k\geq N_{\alpha}
$$
 for some $N_{\alpha}$ depending only on $\alpha$, we have that 
\begin{equation}
\begin{aligned}
\label{simplexinEq}
\h{the
simplex $i^{(k)}(x)$ 
(recall  \eqref{ikEq})
containing $x$ is contained in $\overline{B}_{\alpha}(x)$}
\end{aligned}
\end{equation}
by the admissibility   of our sequence of almost-triangulations
(Definition \ref{admissibleATDef}); in particular 
$$
\tau^{(k)}(x)\in \overline{B}_{\alpha}(x).
$$ 
Statement \eqref{simplexinEq} also implies, by  \eqref{ApproxGradPhiGridEq}, 
that $\gamma^{(k)}(x)$ is a convex combination of $n+1$ elements
of $\bigcup_{y\in\overline{B}_{\alpha}(x)}\partial\phi^{(k)}(y)$,
so by \eqref{nbhdEq},
$$
\gamma^{(k)}(x)\in\mathrm{co}
\Big(
\bigcup_{z\in\overline{B}_{\alpha}(x)}\partial\phi(z)+B_{\beta}(0)\Big),
\q
\forall k\geq N_{x,\alpha,\beta},
$$
where
$$
N_{x,\alpha,\beta}:=\max\{N_{\alpha},N'_{x,\alpha,\beta}\}.
$$
Thus for $k\geq N_{x,\alpha,\beta}$,
\beq
\label{minMaxIneq}
\log\frac{ f (\tau^{(k)}(x))}{g(\gamma^{(k)}(x))}\geq\log
\frac
{\min\{ f (y):y\in\overline{B}_{\alpha}(x)\}}
{
\max\left\{ g(z):z\in\mathrm{co}\left(
\bigcup_{z\in\overline{B}_{\alpha}(x)}\partial\phi(z)+B_{\beta}(0)\right)\right\}
}.
\eeq

For almost every $x\in\O_\ve$
we have that for any $\gamma>0$ there exists 
$$
C(x,\gamma)>0
$$ 
such that 
  \cite[Corollary 24.5.1]{Rock} 
\[
\partial\phi(z)\subset\nabla\phi(x)+B_{\gamma}(0)
=B_{\gamma}(\nabla\phi(x)),
\q
\forall z\in\overline{B}_{\alpha}(x), \;
\forall \alpha\in (0,C).
\]
Hence, for $\alpha \in (0,C)$, 
$$
\bigcup_{z\in\overline{B}_{\alpha}(x)}\partial\phi(z)+B_{\beta}(0)
\subset B_{\gamma+\beta}(\nabla\phi(x)),
$$
implying that 
\beq
\label{hullContain}
\mathrm{co}\Big(
{\textstyle\bigcup_{z\in\overline{B}_{\alpha}(x)}}\partial\phi(z)
+B_{\beta}(0)\Big)\subset B_{\gamma+\beta}(\nabla\phi(x)).
\eeq

Therefore for a.e. $x$ and any $\alpha,\beta,\gamma>0$ with 
$\alpha\in(0,C(x,\gamma))$
we have by \eqref{minMaxIneq} and \eqref{hullContain} that 
\[
\log\frac{ f (\tau^{(k)}(x))}{g(\gamma^{(k)}(x))}\geq\log\frac{\min\{ f (y):y\in\overline{B}_{\alpha}(x)\}}{\max\left\{ g(z):z\in \overline{B}_{\gamma+\beta}(\nabla\phi(x))\right\} },
\h{ for $k\geq N_{x,\alpha,\beta}$}.
\]
It follows that 
\beq
\label{liminf1Ineq}
\liminf_{k\rightarrow\infty}\log\frac{ f (\tau^{(k)}(x))}{g(\gamma^{(k)}(x))}\geq\log\frac{\min\{ f (y):y\in\overline{B}_{\alpha}(x)\}}{\max\left\{ g(z):z\in 
\overline{B}_{\gamma+\beta}(\nabla\phi(x))\right\} }.
\eeq
 for a.e. $x$ and any $\alpha,\beta,\gamma>0$ with $\alpha\in(0,C(x,\gamma))$.
By the continuity of $ f $ and $g$, 
\beq
\label{numLimit}
\lim_{\alpha\rightarrow0}\min\{ f (y):y\in\overline{B}_{\alpha}(x)\}= f (x)
\eeq
 and for a.e. $x$ also,
\beq
\label{denomLimit}
\lim_{\left(\beta,\gamma\right)\rightarrow0}\max\left\{ g(z):z\in B_{\gamma+\beta}(\nabla\phi(x))\right\} =g(\nabla\phi(x)).
\eeq
Taking limits in \eqref{liminf1Ineq} (first $\alpha \rightarrow 0$, followed by $\left(\beta,\gamma\right)\rightarrow0$) and applying \eqref{numLimit} and \eqref{denomLimit},
\beq
\label{liminfIneq2}
\liminf_{k\rightarrow\infty}\log\frac{ f (\tau^{(k)}(x))}{g(\gamma^{(k)}(x))}\geq\log\frac{ f (x)}{g(\nabla\phi(x))} \q
\h{ for a.e. $x$.}
\eeq
Next we observe that mollification preserves this inequality in the following sense:
\begin{eqnarray}
\liminf_{k\rightarrow\infty}\left(\log\frac{ f \circ\tau^{(k)}}{g\circ\gamma^{(k)}}\star\xi_{\varepsilon}\right)(x)\label{bigLiminfIneqStart}
 & = & \liminf_{k\rightarrow\infty}\int\xi_{\varepsilon}(x-y)\cdot\log\frac{ f (\tau^{(k)}(y))}{g(\gamma^{(k)}(y))}\, dy \nonumber\\
 & \geq & \int\xi_{\varepsilon}(x-y)\cdot\liminf_{k\rightarrow\infty}\log\frac{ f (\tau^{(k)}(y))}{g(\gamma^{(k)}(y))}\, dy \nonumber\\
 & \geq & \int\xi_{\varepsilon}(x-y)\cdot\log\frac{ f (y)}{g(\nabla\phi(y))}\, dy \nonumber\\
 & = & \left(\log\frac{ f }{g\circ\nabla\phi}\star\xi_{\varepsilon}\right)(x), \label{bigLiminfIneqEnd}
\end{eqnarray}
 where we have used the Fatou-Lebesgue theorem (applicable since the sequence of integrands is dominated by an integrable function; indeed, the domain is bounded and, as $f$ and $g$ are bounded away from zero, the integrands are uniformly bounded) to pass the $\liminf$
within the integral and \eqref{liminfIneq2} in the penultimate step. Note that   \eqref{bigLiminfIneqEnd} makes sense since
$\nabla \phi$ exists a.e.

Now, take a $\liminf$ in Lemma \ref{kDensIneqLemma} 
and use the fact that $c_{k}\rightarrow0$, along with \eqref{bigLiminfIneqEnd}, to see that
\begin{eqnarray}
0=\liminf_{k\ra\infty} c_k & \geq & \liminf_{k\ra\infty}
\int_{\O_\ve} \max\Big\{0,-\log\det \big(\calH^{(k)}\star \xi_\ve\big)
+\log\frac{f \circ \tau^{(k)}}{g\circ\gamma^{(k)}}\star\xi_\ve \Big\} \nonumber\\
& \geq &
\int_{\O_\ve} \liminf_{k\ra\infty} \max\Big\{0,-\log\det \big(\calH^{(k)}\star \xi_\ve\big)
+\log\frac{f \circ \tau^{(k)}}{g\circ\gamma^{(k)}}\star\xi_\ve \Big\} \nonumber\\
& \geq & \int_{\O_\ve} \max\Big\{0,-\limsup_{k\ra\infty}\log\det \big(\calH^{(k)}\star \xi_\ve\big)
+\log\frac{ f }{g\circ\nabla\phi}\star\xi_{\varepsilon} \Big\}. \q  \label{bigMaxInt}
\end{eqnarray}
Note that we have passed the $\liminf$ inside of the integral using Fatou's lemma (applicable since the integrands are nonnegative), and we have used the fact that $$\liminf_k \max(0,a_k)=\max(0,\liminf_k a_k),$$ which follows from the monotonicity and continuity of $\max(0,\cdot)$.

Now it follows from \eqref{bigMaxInt} that 
\beq
\lb{limsupIneq}
\limsup_{k\ra\infty}\log\det \big(\calH^{(k)}\star \xi_\ve\big)
\geq \log\frac{ f }{g\circ\nabla\phi}\star\xi_{\varepsilon}
\eeq
a.e. on $\O_\ve$.
But Lemma \ref{symmetrizedLemma} 
implies that in fact 
$$
\det \big(\calH^{(k)}\star\xi_{\varepsilon}\big)
$$ 
is pointwise convergent on $\O_\ve$ as $k\ra\infty$, so we must have a.e. in $\O_\ve$ that
$$
\lim_{k\ra\infty} \det \big(\calH^{(k)}\star\xi_{\varepsilon}\big) > 0
$$
because otherwise \eqref{limsupIneq} is violated on a set of positive Lebesgue measure. Then we conclude that the sequence on the left-hand side of \eqref{limsupIneq} is actually convergent a.e. in $\O_\ve$ and 
$$
\log \lim_{k\ra\infty}\det \big(\calH^{(k)}\star \xi_\ve\big)
\geq \log\frac{ f }{g\circ\nabla\phi}\star\xi_{\varepsilon}
$$
a.e. in $\O_\ve$. Of course, by Lemma \ref{symmetrizedLemma} we then have 
\[
\log\det\nabla^{2}(\phi\star\xi_{\varepsilon})\geq\log\frac{ f }{g\circ\nabla\phi}\star\xi_{\varepsilon},
\]
a.e. in $\O_\ve$ and in fact, by the continuity of both sides of the inequality, everywhere in $\O_\ve$. 
This implies (since $\xi_{\varepsilon}$ is supported on $B_{\varepsilon}(0)$)
\[
\log\det\nabla^{2}(\phi\star\xi_{\varepsilon})(x)\geq\log\frac{\inf\{ f (y):y\in B_{\varepsilon}(x)\}}{\sup\{g\left(\nabla\phi(y)\right):y\in B_{\varepsilon}(x),\nabla\phi(y) \,\mathrm{exists}\}},
\]
for $x\in\Omega_{\varepsilon}$,
completing the proof of Lemma \ref{kLimLemma}.
\end{proof}

\subsection{Passing to the limit in $\varepsilon$---part I}
 \label{EpsilonSubSec} 
 In this subsection we prove 
Proposition \ref{measProp}. 

\def\phive{\phi_\varepsilon}

Let
\begin{equation}
\begin{aligned}
\label{LepsEq}
\Lambda_{\varepsilon}:=
\nabla\phi_{\varepsilon}(\Omega_{\varepsilon}).
\end{aligned}
\eeq
Because $\L$ is convex, $\Lambda_{\varepsilon} \subset\overline{\L}$.

To ease the notation in the following, we set 
$$
\phi_\ve := \phi\star\xi_\ve.
$$

\begin{claim}
\label{geClaim}
Fix $\varepsilon>0$. Let $g_\ve$ be defined as   \eqref{gepsEq}. For $y\in\L_\varepsilon$,
$$
g_{\varepsilon}(y) =  f \big(
(\nabla\phi_{\varepsilon})^{-1}(y)\big)
/
\det\nabla^{2}\phi_{\varepsilon}
\big((\nabla\phi_{\varepsilon})^{-1}(y)\big).
$$
\end{claim}

\begin{proof}
Since convolution with a non-negative kernel preserves convexity, $\phi_{\varepsilon}$ is a smooth convex function. Moreover, 
Lemma \ref{kLimLemma} in fact implies that $\phi_{\varepsilon}$ is uniformly convex
on any compact subset of $\Omega_{\varepsilon}$,
so $\nabla\phi_{\varepsilon}$ is invertible
on $\Omega_{\varepsilon}$. This is evident if $\O_\ve$ is convex, but is also true in general. Indeed, for any two points $x,y\in\Omega_{\varepsilon}$, consider the restriction of $\phi_{\varepsilon}$ to the line containing these two points. The second directional derivative of $\phi_{\varepsilon}$ in the direction of a unit vector parallel to this line must be non-negative along this line and strictly positive near both $x$ and $y$ (recall that for $\ve$ small, $\phi_\ve$ is defined, convex and finite on 
a ball containing $\O$, in particular on the convex hull of $\O$, and uniformly convex when restricted to any compact subset of $\O_\ve$). In conclusion, 
$\nabla\phive (x)$ and $\nabla\phive( y)$ cannot agree. Thus, 
\begin{equation}
\begin{aligned}
\label{nablaphiinvEq}
\h{$(\nabla\phi_{\varepsilon})^{-1}$ exists on $\Omega_{\varepsilon}$}
\end{aligned}
\end{equation}
 as claimed. The standard formula for the push-forward of a measure 
and the definitions \eqref{nuepsEq}--\eqref{gepsEq}
then imply the statement.
\end{proof}

 Let
\begin{equation}
\begin{aligned}
\label{gbarEpsDef}
&\overline{g}_{\varepsilon}(y):=
\cr
&\q\begin{cases}
\displaystyle
 f\big((\nabla\phi_{\varepsilon})^{-1}(y)\big)
\frac{
\sup
\left\{ 
g(\nabla\phi(x)):x\in 
B_{\varepsilon}\big((\nabla\phi_{\varepsilon})^{-1}(y)\big),
\nabla\phi(x) \,\mathrm{exists} 
\right\} 
}
{
\inf
\left\{  
f (x):x\in B_{\varepsilon}\big((\nabla\phi_{\varepsilon})^{-1}(y)\big)
\right\} 
}, 
& y\in \Lambda_{\varepsilon},\cr
g, & y\in \overline{\Lambda}\backslash\Lambda_{\varepsilon}.
\end{cases}
\end{aligned}
\end{equation}

\begin{claim}
\label{epsDensClaim}
On $\overline{\Lambda}$,
\beq
\label{epsDensIneq}
g_{\varepsilon}\leq\overline{g}_{\varepsilon}.
\eeq

\end{claim}

\begin{proof}
The inequality is trivial on $\overline{\Lambda}\backslash\Lambda_{\varepsilon}$
since by the definitions \eqref{nuepsEq}, \eqref{gepsEq}, and   \eqref{LepsEq},
$$
g_\ve = 0 \h{\ outside of $\L_\ve$}.
$$
On the other hand, Lemma \ref{kLimLemma} and Claim \ref{geClaim} precisely imply 
  \eqref{epsDensIneq} on $\L_\ve$. 
\epf

For any $\psi$ convex defined on a convex set $C$, let 
$$
\mathrm{dom}(\psi)
$$ 
denote the set of points in $C$ at which $\psi$ is finite, and let 
\begin{equation}
\begin{aligned}
\label{singEq}
\mathrm{sing}(\psi)
\end{aligned}
\end{equation}
 denote the set of points in $\mathrm{int}\,\mathrm{dom}(\psi)$ where $\psi$ is not differentiable.
Similarly, denote by
\begin{equation}
\begin{aligned}
\label{DiffEq}
\Delta(\psi)
\end{aligned}
\end{equation}
the complement of $\mathrm{sing}(\psi)$ in $\mathrm{int}\,\mathrm{dom}(\psi)$. Note that $\Delta(\psi)$ has full measure in $\mathrm{int}\,\mathrm{dom}(\psi)$ because convex functions are locally Lipschitz.

\begin{lemma}
\label{legendreLemma}
As $\varepsilon\rightarrow0$,
$\overline{g}_{\varepsilon}\rightarrow g$ a.e. on $\overline{\Lambda} \backslash \partial\phi\left(\mathrm{sing}(\phi)\right)$. 
\end{lemma}

\begin{remark}
As in Remark   \ref{fguniformRemark}, the proof of this lemma becomes considerably easier in the case that $g$ is a uniform density and trivial in the case that both $f$ and $g$ are uniform densities.
\end{remark}

Before proving Lemma \ref{legendreLemma} we make several technical remarks. Recall from the beginning of the proof (see \eqref{Ddef}) that $\phi$ is taken to be defined on a ball $D$ containing $\overline{\Omega}$ in its interior, and in fact $\phi$ is the uniform limit of the $\phi^{(k)}$ on $D$. Accordingly, $\phi$ is convex and continuous on $D$, and $\mathrm{dom}(\phi)=D$. Likewise $\mathrm{dom}(\phi_\ve)=D_\ve$, where $D_\ve$ is the closed ball (concentric with $D$) of radius $\ve$ less than that of $D$.

Thus far, we have only studied the behavior of $\phi$ inside of 
$\Omega$
as there has been no need to consider its behavior elsewhere. However, since $\O$ may not be convex, it is important to consider $\phi$ as being defined on a (larger) convex set in order to employ the language and results of convex analysis.

Then,  the convex conjugate of $\phi$
$$
\phi^{*}(y):=\sup_{x\in D}[\langle x,y\rangle - \phi(x)]
$$
is finite on all of $\mathbb{R}^n$, i.e., $\mathrm{dom}(\phi^*)=\mathbb{R}^n$.

The next claim collects basic properties concerning the Legendre dual that we need later.

\begin{claim}
\label{invConvClaim}
(i) $\nabla \phi_{\varepsilon}^{*}=\left(\nabla \phi_\varepsilon\right)^{-1}$ on $\Lambda_{\varepsilon}$,
\\ 
(ii) $\nabla \phi^{*}_{\varepsilon} \rightarrow\nabla\phi^{*}$ pointwise on $\Delta(\phi^*)$.
\end{claim}
\bpf
(i) This follows from  \eqref{nablaphiinvEq} and the standard formula for the gradient of the Legendre dual of a smooth strongly convex function \cite[Theorem 26.5]{Rock}, noting that, by Lemma \ref{kLimLemma}, $\phi_{\ve}$ is indeed strongly convex on compact subsets of $\O_\ve$.

(ii) Note that $\phi_{\varepsilon}\rightarrow \phi$ uniformly on compact subsets of $\mathrm{int}\,\mathrm{dom}(\phi) = \mathrm{int}\,D$. Then by \cite[Theorem 7.17]{RW}, the $\phi_{\varepsilon}$ epi-converge to $\phi$ 
(we again recall that $f_k$ epi-converges to $f$ (roughly) if the epigraphs of $f_k$ converge to the epigraph of $f$; see \cite [p. 240]{RW}  for more precise details). 
Then by \cite[Theorem 11.34]{RW}, we have that the $\phi_{\varepsilon}^{*}$ epi-converge to $\phi^{*}$. Again using Theorem \ref{BWThm}, we have that $\partial \phi_{\varepsilon}^{*}(x)\rightarrow \nabla \phi^{*}(x)$ for all $x\in \mathrm{int\,dom}\,\phi^{*}=\RR^n$ such that $\nabla \phi^{*}(x)$ exists.
\epf

\bpf[Proof of Lemma \ref{legendreLemma}]
 
It suffices to assume that 
 $y\in \Delta(\phi^*)\cap \big(\overline{\Lambda} \backslash \partial\phi\left(\mathrm{sing}(\phi)\right)\big)$
 since otherwise $y $ is contained in a measure zero set. Thus, by Claim
 \ref{invConvClaim} (ii),
 \beq
 \lim_{\ve\ra0}
\nabla\phi^{*}_{\varepsilon}(y)= \nabla\phi^{*}(y). \nonumber
\eeq

Let 
$$
E:= \{\varepsilon>0\,:\,y\in\Lambda_{\varepsilon}\}.
$$ 
Since $\overline{g}_{\varepsilon}(y) = g(y)$ whenever $\varepsilon\notin E$
(recall \eqref{gbarEpsDef}), it suffices to show that 
$\overline{g}_{\varepsilon_{j}} (y) \rightarrow g(y)$ for all sequences $\varepsilon_{j}\in E$ that tend to zero. Let $\varepsilon_{j}$ be such a sequence. 
Notice that since $y\in \Lambda_{\varepsilon_{j}}$ for all $j$, by 
Claim \ref{invConvClaim} (i), 
$\nabla \phi_{\varepsilon_{j}}^{*}(y)
=
(\nabla \phi_{\varepsilon_{j}})^{-1}(y)$ for all $j$, so 
\beq
\label{gradConv}
\lim_{j\ra\infty}\left(\nabla\phi_{\varepsilon_{j}}\right)^{-1}(y)= \nabla\phi^{*}(y). 
\eeq
 Let $\gamma>0$. Then by \eqref{gradConv}, there exists $N\in\NN$ so that 
 \beq
 \left| \left(\nabla\phi_{\varepsilon_{j}}\right)^{-1}(y) -  \nabla\phi^{*}(y) \right| < \gamma/2,
 \q \forall j\geq N.\nonumber
\eeq
Take $N$ large enough so that $\varepsilon_{j} < \gamma/2$ for all $j\geq N$, so then 
\beq
B_{\varepsilon_{j}}\left(\left(\nabla\phi_{\varepsilon_{j}}\right)^{-1}(y)\right) \subset B_{\gamma}\left(\nabla\phi^{*}(y)\right)\cap\Omega,\q \forall j\geq N. \nonumber
\eeq
Thus,
\beq
\inf\left\{  f (x):x\in B_{\varepsilon_{j}}\left(\left(\nabla\phi_{\varepsilon_{j}}\right)^{-1}(y)\right)\right\} \geq \inf\left\{  f (x):x\in B_{\gamma}\left(\nabla\phi^{*}(y)\right)\cap\Omega\right\},
\q \forall j\geq N.\nonumber
\eeq
Then taking the $\liminf$ as $j\rightarrow \infty$ followed by the limit as $\gamma \rightarrow 0$, and using the continuity of $f$, 
\beq
\liminf_{j\rightarrow \infty} \inf\left\{  f (x):x\in B_{\varepsilon_{j}}\left(\left(\nabla\phi_{\varepsilon_{j}}\right)^{-1}(y)\right)\right\} \geq f\left(\nabla\phi^{*}(y)\right) \nonumber.
\eeq
Also, 
\beq
\inf\left\{  f (x):x\in B_{\varepsilon_{j}}\left(\left(\nabla\phi_{\varepsilon_{j}}\right)^{-1}(y)\right)\right\} \leq f\left(\left(\nabla\phi_{\varepsilon_{j}}\right)^{-1}(y)\right) \underset{j\rightarrow \infty}{\longrightarrow} f\left(\nabla\phi^{*}(y)\right) \nonumber
\eeq
(where the limit follows by the continuity of $f$), so 
\beq
\limsup_{j\rightarrow \infty} \inf\left\{  f (x):x\in B_{\varepsilon_{j}}\left(\left(\nabla\phi_{\varepsilon_{j}}\right)^{-1}(y)\right)\right\} \leq f\left(\nabla\phi^{*}(y)\right). \nonumber
\eeq
Thus,
\beq
\lim_{j\rightarrow \infty} \inf\left\{  f (x):x\in B_{\varepsilon_{j}}\left(\left(\nabla\phi_{\varepsilon_{j}}\right)^{-1}(y)\right)\right\} = f\left(\nabla\phi^{*}(y)\right) \nonumber.
\eeq
It follows that
\beq
\lim_{j\rightarrow \infty} \frac{f\left( \left( \nabla\phi_{\varepsilon_{j}} \right) ^{-1} (y) \right)}{\inf\left\{  f (x):x\in B_{\varepsilon_{j}}\left(\left(\nabla\phi_{\varepsilon_{j}}\right)^{-1}(y)\right)\right\} } = 1 \nonumber.
\eeq

Thus, to conclude the proof of the lemma it remains only to show that
\beq
\lim_{j\rightarrow \infty} \sup\left\{ g(\nabla\phi(x)):x\in B_{\varepsilon_{j}}\left(\left(\nabla\phi_{\varepsilon_{j}}\right)^{-1}(y)\right),\nabla\phi(x) \,\mathrm{exists}\right\} = g(y). \label{remainderEq}
\eeq

By the same arguments as above we have that 
\begin{eqnarray}
& & \limsup_{j\rightarrow\infty} \sup\left\{ g(\nabla\phi(x)):x\in B_{\varepsilon_{j}}\left(\left(\nabla\phi_{\varepsilon_{j}}\right)^{-1}(y)\right),\nabla\phi(x) \,\mathrm{exists}\right\} \nonumber \\ & & \leq \ \sup\left\{ g(\nabla\phi(x)):x\in B_{\gamma}\left(\nabla\phi^{*}(y)\right)\cap\Omega,\nabla\phi(x) \,\mathrm{exists}\right\} \label{gSupRHS}
\end{eqnarray}
for any $\gamma>0$.

We claim that $\phi$ is differentiable at $\nabla\phi^{*}(y)$. 
Indeed, 
$z\in \partial \phi(\nabla \phi^{*}(y))$ if and only if $\nabla \phi^{*}(y)\in \partial \phi^{*}(z)$ \cite[Corollary 23.5.1]{Rock}. Thus plugging in $y$ for $z$, we see that $y\in\partial\phi(\nabla \phi^{*}(y))$. By assumption, $y\notin \partial\phi\left(\mathrm{sing}(\phi)\right)$, so it must be that $\nabla \phi^{*}(y) \notin \mathrm{sing}(\phi)$, as claimed, and,
moreover, 
\begin{equation}
\begin{aligned}
\label{nablayyEq}
\nabla \phi\left( \nabla \phi^{*} (y)\right)=y.
\end{aligned}
\end{equation}
Next, for any $\alpha >0$ there exists
 $\gamma>0$ such that \cite[Corollary 24.5.1]{Rock}
\beq
\label{subgradCont}
\partial\phi(\nabla \phi^{*} (y)+v)\subset\nabla\phi(\nabla \phi^{*} (y))+B_{\alpha}(0) = B_{\alpha}(y),
\q \forall v\in\overline{B}_{\gamma}(0).
\eeq
  Together with the continuity of $g$, this implies that the right hand side of \eqref{gSupRHS} converges to $g(y)$ as $\gamma\rightarrow 0$, so we have

\beq
\limsup_{j\rightarrow\infty} \sup\left\{ g(\nabla\phi(y)):y\in B_{\varepsilon_{j}}\left(\left(\nabla\phi_{\varepsilon_{j}}\right)^{-1}(y)\right),\nabla\phi(y) \,\mathrm{exists}\right\} \leq g(y) \nonumber.
\eeq
Of course, we also have 
\beq
\sup\left\{ g(\nabla\phi(x)):x\in B_{\varepsilon_{j}}\left(\left(\nabla\phi_{\varepsilon_{j}}\right)^{-1}(y)\right),\nabla\phi(x) \,\mathrm{exists}\right\} \geq g\left( \nabla\phi\left( \left(\nabla\phi_{\varepsilon_{j}}\right)^{-1} (y)+v_{j}\right)\right), \nonumber
\eeq
where $v_j$ is a vector with length less than $\varepsilon_{j}$ chosen so that $\phi$ is differentiable at $\left(\nabla\phi_{\varepsilon_{j}}\right)^{-1} (y)+v_{j}$. 
Then following from \eqref{subgradCont}, the continuity of $g$, \eqref{gradConv}, and 
  \eqref{nablayyEq}, we have that the right hand side tends to $g(y)$ as $j\rightarrow \infty$. 
  Thus, \eqref{remainderEq} holds, and the proof of Lemma \ref{legendreLemma} complete.
\epf

\subsection{Passing to the limit in $\varepsilon$---part II: proof of Proposition \ref{measProp}}
\label{PropSubSec}

Define a measure supported on $\L$,
\begin{equation}
\begin{aligned}
\label{tnuepsEq}
\tilde\nu:=(\nabla \phi)_{\#}\mu.
\end{aligned}
\end{equation}

\begin{claim}
\label{measClaim}
For any sequence $\varepsilon\rightarrow0$, 
$\nu_{\varepsilon}=(\nabla \phi_\ve)_{\#}\mu|_{\Omega_{\varepsilon}}$  converges weakly to $\tilde\nu$.

\end{claim}

\bpf
Let $\varepsilon\rightarrow0$, and 
let 
$\zeta$ be a bounded continuous function on $\mathbb{R}^{n}$. 
Recalling the definition of $\nu_\ve$  \eqref{nuepsEq}, the change of variables formula for the push-forward measure gives 
$$
\baeq
\mu(\Omega_{\varepsilon})^{-1} \int \zeta\,d\nu_{\varepsilon} 
& =  
\mu(\Omega_{\varepsilon})^{-1} \int \zeta\circ\nabla\phi_{\varepsilon}\,d\mu\vert_{\Omega_{\varepsilon}} \\
& =  \mu(\Omega_{\varepsilon})^{-1} 
\int_{\Omega} 
(\zeta\circ\nabla\phi_{\varepsilon})\cdot f\cdot\chi_{\Omega_{\varepsilon}}\,dx.
\eaeq
$$
Now $\nabla \phi_{\varepsilon} \rightarrow \nabla\phi$ pointwise almost everywhere, and $\chi_{\Omega_{\varepsilon}}\rightarrow\chi_{\Omega}$ pointwise, so (recalling that $\zeta$ is bounded and continuous), we have by bounded convergence and the fact that $\mu(\Omega_{\varepsilon})\rightarrow \mu(\Omega)$ that
$$
\baeq
\lim_{\ve\rightarrow0} 
\mu(\Omega_{\varepsilon})^{-1} \int \zeta\,d\nu_{\varepsilon} & =  \mu(\Omega)^{-1} \int_\Omega \zeta\circ\nabla\phi \, d\mu 
\cr
& = \mu(\Omega)^{-1} \int_\Omega \zeta \, d\left( (\nabla\phi)_{\#}\mu\right).
\eaeq
$$
This proves that $\mu(\Omega)\mu(\Omega_{\varepsilon})^{-1}\nu_{\varepsilon}$ converges weakly to $\tilde{\nu}:=(\nabla\phi)_{\#}\mu$. 
Since 
$$
\lim_{\ve\rightarrow0} \mu(\Omega_{\varepsilon})=\mu(\Omega),
$$ 
we are done.
\epf

Observe that $\tilde{\nu}$ must be absolutely continuous because
the densities $g_\ve$ of $\nu_{\varepsilon}$
are bounded above uniformly in $\ve$, see 
 \eqref{gbarEpsDef},
\begin{equation}
\begin{aligned}
\label{gvemaxEq}
g_\ve\le \sup f \sup g/\inf f<C
\end{aligned}
\end{equation}
 and supported on the compact set $\overline{\Lambda}$.
Hence, $\tilde{\nu}$ has a density that we denote by
\begin{equation}
\begin{aligned}
\label{tgEq}
\tilde g\,dx:=\tilde\nu.
\end{aligned}
\end{equation} 
 Proposition \ref{measProp} follows from Claim \ref{measClaim} and the next
two results. 

\begin{lemma}
\label{ggtildeLemma}
$\tilde g \le g$ a.e.
\end{lemma}

\begin{cor}
\label{ggtildeCor}
$\tilde g =g$ a.e., i.e.,
$\tilde \nu =\nu$.
\end{cor}

\begin{proof}[Proof of Corollary \ref{ggtildeCor}] 
By Lemma \ref{ggtildeLemma} $g\geq\tilde{g}$ a.e.; thus,
\beq
\label{measureEq}
\baeq
\int_{\Lambda}\vert g -\tilde{g}\vert 
& = \int_{\Lambda}\left(g-\tilde{g}\right)\cr
 & =  \int_{\Lambda}g-\int_{\Lambda}\tilde{g}  =  \nu(\Lambda)-\tilde{\nu}(\Lambda). 
\eaeq
\eeq
As noted in the previous paragraph $\tilde{\nu}$
is absolutely continuous (with respect to the Lebesgue measure).
Hence, as $\partial\Lambda$ is a Lebesgue null set,
$\tilde{\nu}(\partial\Lambda)=0$, i.e., $\Lambda$
is a continuity set of $\tilde{\nu}$. 
Therefore, by Claim \ref{measClaim},
\[
\tilde{\nu}(\Lambda)=\tilde{\nu}(\overline{\Lambda})=\lim_{\ve\ra 0}\nu_{\varepsilon}(\overline{\Lambda})=\lim_{\ve\ra 0}\mu\vert_{\Omega_{\varepsilon}}\big(\left(\nabla\phi_{\varepsilon}\right)^{-1}(\overline{\Lambda})\big).
\]
 Now $\left(\nabla\phi_{\varepsilon_{n}}\right)^{-1}(\overline{\Lambda})\supset\Omega_{\varepsilon_{n}}$,
so $\mu\vert_{\Omega_{\varepsilon_{n}}}
\big(\left(\nabla\phi_{\varepsilon_{n}}\right)^{-1}(\overline{\Lambda})\big)=\mu\left(\Omega_{\varepsilon_{n}}\right)$,
and $\tilde{\nu}(\Lambda)=\mu(\Omega)$. Of course, by   \eqref{massEq}, $\mu(\Omega)=\nu(\Lambda)$,
so by \eqref{measureEq} we have that $\int_{\Lambda}\vert g-\tilde{g}\vert=0$.
This implies that $g=\tilde{g}$ a.e., so $\nu=\tilde{\nu}$.
\end{proof}

\bpf[Proof of Lemma \ref{ggtildeLemma}] 
Define (recall \eqref{singEq})
\begin{equation}
\begin{aligned}
\label{}
S:=\partial \phi \left(\mathrm{sing}(\phi)\right).
\end{aligned}
\end{equation}

\begin{claim}
 \label{nullsetClaim} 
$\Delta(\phi^{*})\cap S$ is a $\tilde{\mathrm{\nu}}$-null set.
\end{claim}
\bpf
First we claim that on the set where $\phi^* $ is differentiable, $S$ can be written as 
\begin{equation}
\begin{aligned}
\label{SDeltaphiEq}
\Delta(\phi^{*}) \cap S 
=
\Delta(\phi^{*}) \cap
\left(\nabla\phi^{*}\right)^{-1}\left(\mathrm{sing}(\phi)\right).
\end{aligned}
\end{equation}
Indeed, suppose that $y\in \Delta(\phi^{*}) \cap
\left(\nabla\phi^{*}\right)^{-1}\left(\mathrm{sing}(\phi)\right)$. Then $\nabla \phi^{*}(y)\in\mathrm{sing}(\phi)$. Now by duality \cite[Corollary 23.5.1] {Rock}, $y\in\partial\phi \left(\nabla \phi^{*}(y)\right)$, so $y\in S$.

Then suppose that $y\in \Delta(\phi^{*}) \cap S$. Then $y\in \partial\phi(x)$ for some $x\in \mathrm{sing}(\phi)$, implying, again by duality \cite[Corollary 23.5.1] {Rock}, that $x\in \partial \phi^{*}(y)$, i.e., $x=\nabla\phi^{*}(y)$, and $y\in \left(\nabla\phi^{*}\right)^{-1}\left(\mathrm{sing}(\phi)\right)$. This gives the claimed set equality.

Next, we claim that 
\begin{equation}
\begin{aligned}
\label{BorelEq}
\h{$\Delta(\phi^{*})\cap S$ is Borel}.
\end{aligned}
\end{equation}
For the proof, define an auxiliary vector-valued function $\Phi: \mathbb{R}^n \ra \overline{\O}$ by
\begin{equation}
\begin{aligned}
\label{PhiEq}
\Phi(y):=
\begin{cases}
\nabla \phi^{*}(y) & \h{ on $\Delta(\phi^{*})$},\cr
z_0& \h{ on $\mathrm{sing}(\phi^{*})$},
\end{cases}
\end{aligned}
\end{equation}
where $z_0\in \Delta(\phi)\cap\O$ (any such (fixed) $z_0$ will do).

\begin{claim}
\label{}
$\Phi$ is a Borel-measurable function. 
\end{claim}

\begin{proof}

Recall that the set of points of differentiability of a continuous function (e.g., $\phi$ and $\phi^{*}$) is Borel 
(this is an elementary fact, though see, e.g., \cite{Zahorski}), 
and hence $\mathrm{sing}(\phi)$, $\mathrm{sing}(\phi^{*})$ are also Borel.

Let $U$ be open. 
Recall that $\nabla \phi^{*}$ is continuous on $\Delta(\phi^{*})$
 \cite[Corollary 24.5.1]{Rock}. 
Thus, if $z_0\notin U$, then $\Phi^{-1}(U)$ is open in $\Delta(\phi^{*})$, i.e., $\Phi^{-1}(U)=O\cap\Delta(\phi^{*})$ for some open $O$. If $z_0\in U$, then $\Phi^{-1}(U)$ is the union of $\mathrm{sing}(\phi^{*})$ with some set open in $\Delta(\phi^{*})$. In either case, $\Phi^{-1}(U)$ is Borel, i.e.,
$\Phi$ is a Borel-measurable function.
\end{proof}

\noindent
 Therefore,  $\Phi^{-1}(\mathrm{sing}(\phi))$ is Borel. 
 Notice, since by construction $z_0 \notin \mathrm{sing}(\phi)$, that 
$$
\Phi^{-1}(\mathrm{sing}(\phi)) 
= \left(\nabla\phi^{*}\right)^{-1}\left(\mathrm{sing}(\phi)\right)\cap\Delta(\phi^{*}),
$$
which together with   \eqref{SDeltaphiEq} proves \eqref{BorelEq}. 

Thus, $\Delta(\phi^{*})\cap S$ is $\tilde{\nu}$-measurable
since $\tilde{\nu}$ is absolutely continuous. Compute,
$$
\baeq
\tilde{\nu}(\Delta(\phi^{*})\cap S) & =  \int \chi_{\Delta(\phi^{*})\cap S} \, d\left( (\nabla \phi)_{\#} \mu \right) \\
& =  
\int_{\Omega} \chi_{\Delta(\phi^{*})\cap S} \circ \nabla\phi \, d\mu.
\eaeq
$$
This integral vanishes since
the integrand is only nonzero on the 
Lebesgue-null set $\mathrm{sing}(\phi)$, 
while $\mu$ is absolutely continuous. The proof of Claim \ref{nullsetClaim} is complete.
\epf

Define a set $E\subset\mathrm{sing}(\phi^{*})$ by,
\begin{equation}
\begin{aligned}
\label{EDefEq}
\left(\Delta(\phi^{*})\cap S\right)^{c}=
S^{c} \cup \mathrm{sing}(\phi^{*})=:S^{c} \cup E,
\end{aligned}
\end{equation}
and by requiring that the union in the last expression be disjoint.
 Since $E$ is contained within a set of Lebesgue-measure zero, $E$ is Lebesgue-measurable with measure zero, 
and hence $S^{c}$ is Lebesgue-measurable as well.
Claim \ref{nullsetClaim} implies that
\begin{equation}
\begin{aligned}
\label{tnurewriteEq}
\tilde{\nu} = 
\tilde{\nu}\vert_{\left(\Delta(\phi^{*})\cap S\right)^{c}}
=
\tilde{\nu}\vert_{S^{c} \cup \mathrm{sing}(\phi^{*})}
=
\tilde{\nu}\vert_{S^{c} \cup E}.
\end{aligned}
\end{equation}
By the absolute continuity of $\tilde{\nu}$ (recall  \eqref{tgEq}), 
\beq
\label{rnDeriv}
\tilde{g}(x)=\lim_{r\rightarrow0}\frac{\tilde{\nu}(B_{r}(x))}{\mathrm{vol}(B_{r}(x))},
\q \h{for a.e. $x\in\Lambda$}. 
\eeq
Let $\alpha>0$ and let $U$ be an open set containing $S^{c} \cup E$ with 

\begin{equation}
\begin{aligned}
\label{malphaEq}
m\big(U\backslash (S^{c} \cup E)\big)<\alpha
\end{aligned}
\end{equation}
(where $m$ denotes the Lebesgue measure). 
This is possible because $S^{c} \cup E$ is Borel
thanks to   \eqref{BorelEq} and \eqref{EDefEq}. Then, using   \eqref{tnurewriteEq}, 
\beq
\tilde{\nu}(B_{r}(x)) = \tilde{\nu}\big(B_{r}(x)\cap (S^{c} \cup E)\big) \leq \tilde{\nu}(B_{r}(x)\cap U). \label{openIneq}
\eeq
Now since $B_{r}(x)\cap U$ is open, by Claim \ref{measClaim},  
\beq
\label{weakConvIneq}
\tilde{\nu}(B_{r}(x)\cap U)\leq\liminf_{\ve\ra 0}\nu_{\ve}(B_{r}(x)\cap U),
\q \forall r>0. 
\eeq
Observe that 
$$
\baeq
B_{r}(x) \cap U &= \big(B_{r}(x) \cap (S^{c} \cup E)\big) \cup \big(B_{r}(x) \cap (U\backslash (S^{c} \cup E))\big) 
\cr
&\subset \big(B_{r}(x) \cap (S^{c} \cup E)\big) \cup \big(U\backslash (S^{c} \cup E)\big),
\eaeq
$$
so then  by   \eqref{gepsEq}, \eqref{gvemaxEq}, \eqref{malphaEq}, and Claim \ref{epsDensClaim},
$$
\baeq
\nu_{\ve}(B_{r}(x)\cap U) & \leq  
\nu_{\ve}\big(B_{r}(x) \cap (S^{c} \cup E)\big) 
+ 
\nu_{\ve}\big(U\backslash (S^{c} \cup E)\big) \\ 
& \leq  \int_{B_{r}(x)\cap (S^{c} \cup E)} g_{\ve} + C m\big(U\backslash (S^{c} \cup E)\big) \\
& \leq  \int_{B_{r}(x)\cap S^{c}} \overline{g}_{\ve} + C\alpha,
\eaeq
$$
where in the last step we used the fact that $E$ has Lebesgue measure zero. 
Now by Lemma \ref{legendreLemma}, $\overline{g}_{\ve}\rightarrow g$ a.e. on $S^{c}$, so by bounded convergence (since the $\overline{g}_{\varepsilon}$ are uniformly bounded
by definition \eqref{gbarEpsDef} by the same bound as in  \eqref{gvemaxEq}) the last expression is convergent and
\beq
\liminf_{\ve\ra 0} \nu_{\ve}(B_{r}(x)\cap U) 
\leq  \int_{B_{r}(x)\cap S^{c}} g + C\alpha \label{liminfIneq}.
\eeq
Then by \eqref{openIneq}, \eqref{weakConvIneq}, and \eqref{liminfIneq}, we have that 
\beq
\tilde{\nu}(B_{r}(x)) \leq \int_{B_{r}(x)\cap S^{c}} g + C\alpha \leq \int_{B_{r}(x)} g + C\alpha \nonumber
\eeq
for all $\alpha >0$, i.e.,
\beq
\tilde{\nu}(B_{r}(x)) \leq \int_{B_{r}(x)} g. \label{ballIneq}
\eeq
Then by \eqref{rnDeriv}, \eqref{ballIneq} and continuity of $g $,
\beq
\tilde{g}(x)  \leq \liminf_{r\rightarrow0}\frac{1}{\mathrm{vol}(B_{r}(x))} \int_{B_{r}(x)} g = g(x), \nonumber
\eeq
for a.e. $x$, concluding the proof of Lemma \ref{ggtildeLemma}. 
\epf

\subsection{Concluding the proof via the stability of optimal transport}
 \label{FinalSubSec}
 
 We are at last in a position to complete the proof of the main theorem. As explained in \S \ref{strategySubSec} it remains only to establish Lemma  \ref{FinalLemma},
whose proof hinges on two claims. The proof of these claims will follow the proof of the lemma.

Recall that $\nabla\vp$ is the unique optimal transport map from $\mu$ to $\nu$
and $\vp(0)=0$.

\begin{claim}
\label{probConvClaim}
 Fix $\delta>0$. As $\ve$ tends to zero,
$\nabla\phi_{\varepsilon}$ converges to $\nabla\vp$ in probability with respect to 
$\mu\vert_{\Omega_{\delta}}/\mu(\Omega_{\delta})$.
\end{claim}

\begin{claim}
\label{aeConvClaim}
As $\ve$ tends to zero,
$\nabla\phi_{\varepsilon}$ converges to $\nabla\phi$ a.e. on $\Omega$.
\end{claim}

\begin{proof}[Proof of Lemma  \ref{FinalLemma}]
Let $\beta>0$.
A consequence of Claim  \ref{probConvClaim} is that
there exists a sequence $\ve_{j}\ra0$ such that 
$\nabla\phi_{\varepsilon_j}
\rightarrow \nabla\vp$
$\mu$-almost everywhere on $\Omega_{\beta}$, 
hence a.e. on $\Omega_{\beta}$ (because $f$ is bounded away from zero). 
But, by  Claim  \ref{aeConvClaim}, 
$\nabla\phi_{\varepsilon_j}\rightarrow\nabla\phi$
a.e. on $\Omega_{\beta}$. Hence, $\nabla\phi=\nabla\vp$
on $\Omega_{\beta}$. 
Since $\bigcup_{\beta>0}\Omega_{\beta}=\Omega$,
we have that $T=\nabla\phi$ on $\Omega$, i.e., $\nabla\vp=\nabla\phi$
a.e. Since $\vp,\phi\in C^{0,1}(\overline{\Omega})$ 
both are absolutely continuous and since $\phi(0)=0=\vp(0)$, we have that $\phi=\vp$ on $\overline{\O}$. 
\end{proof}

\bpf[Proof of Claim \ref{probConvClaim}] 
The stability theorem for optimal transport maps  
states that whenever the push-forward of a given probability measure 
$\alpha$ under a sequence of optimal transport maps $\{T_j\}_{j\in\NN}$
converges weakly to $\beta$, i.e.,
$$
(T_j)_{\#}\alpha\ra \beta \h{\ \ weakly as\  $j\ra\infty$},
$$
then $T_j$ converges in probability to the unique optimal transport
map pushing-forward $\alpha$ to $\beta$, assuming such a unique map exists
\cite[Corollary 5.23]{VillaniOldNew}.   
We need a slight extension of this result where instead of a fixed measure $\alpha $ we have a sequence of measures $\alpha_j$
converging weakly to $\alpha $, and
$$
(T_j)_{\#}\alpha_j\ra \beta \h{\ \ weakly as\  $j\ra\infty$}.
$$
The result we need is stated in Proposition \ref{StabilityThm} below. Its proof is given  in \S\ref{StabilityThmSubSec}.

Now, Proposition \ref{StabilityThm} may be 
applied to
$$
\alpha_j:=\mu(\Omega_{\varepsilon(j)})^{-1}\mu_{\varepsilon(j)},
\q
T_j:=\nabla\phi_{\varepsilon(j)},
\q
\alpha:=\mu(\Omega)^{-1}\mu,
\q
\beta:=\mu(\Omega)^{-1}\nu,
$$
where $\{\epsilon (j)\}_{j\in\NN}$ is any sequence 
of positive numbers converging to $0$. 
Indeed, Brenier's theorem \cite{Brenier} gives that $T_j$ is an optimal
transport map pushing-forward 
$\alpha_j$
to $\mu(\Omega_{\varepsilon(j)})^{-1}\nu_{\varepsilon(j)}$
and these latter measures converge weakly to $\beta $ by Proposition \ref{measProp},
while $\alpha_j$ evidently weakly converges to $\alpha $.
Thus,
\beq
\label{quasiProbConv}
\lim_{\ve\ra0}\mu\vert_{\Omega_{\varepsilon}}
\Big(
\big\{ 
x\in\Omega\,:\, 
d\big(
\nabla\phi_{\varepsilon}(x),\nabla\vp(x))\geq\gamma
\big)
\big\} 
\Big)\rightarrow0,
\q
\forall \gamma>0. 
\eeq
Let $\delta,\gamma>0$.
Then for all $\ve$ sufficiently close to zero, 
$\Omega_{\delta}\subset\Omega_{\varepsilon}\subset\Omega$,
so 
$$
\baeq
\mu\vert_{\Omega_{\delta}}\big(\!\left\{ x\in\Omega_{\delta}:d\left(\nabla\phi_{\varepsilon}(x),\nabla\vp(x)\right)\geq\gamma\right\}\!\big) 
& =  
\mu\big(\Omega_{\delta}\cap\left\{ x\in\Omega_{\delta}:d\left(\nabla\phi_{\varepsilon}(x),\nabla\vp(x)\right)\geq\gamma\right\}\!\big)
\\
 & \leq  
\mu\big(\Omega_{\varepsilon}\cap\left\{ x\in\Omega:d\left(\nabla\phi_{\varepsilon}(x),\nabla\vp(x)\right)\geq\gamma
\right\}\!\big)
\\
 & =  
\mu\vert_{\Omega_{\varepsilon}}\big(\!\left\{ x\in\Omega:d\left(\nabla\phi_{\varepsilon}(x),\nabla\vp(x)\right)\geq\gamma
\right\}\!\big),
\eaeq
$$
 and the last expression approaches zero as $\ve\ra 0$ 
by \eqref{quasiProbConv}, concluding the proof of Claim \ref{probConvClaim}.
 \epf

\bpf[Proof of Claim \ref{aeConvClaim}] 
Because $\phi$ is convex and continuous, 
$\phi_{\varepsilon}\rightarrow\phi$ pointwise. 
Semi-continuity of the subdifferential map
 \cite[Theorem 24.5]{Rock} gives that for any $x\in\Omega$ and $\alpha>0$,
$$
\partial\phi_{\varepsilon}(x)\subset\partial\phi(x)+B_{\alpha}(0),
\q \h{for all $\ve$ sufficiently small}
$$
(this is a low-brow version of Theorem \ref{BWThm}). 
Thus at a point $x$ such that $\partial\phi$ is a singleton, 
$$
\vert
\nabla\phi(x)-\nabla\phi_{\varepsilon}(x)
\vert<\alpha
$$
for all $\ve$ sufficiently small, 
i.e., $\nabla\phi_{\varepsilon}(x)\rightarrow\nabla\phi(x)$.
Since $\partial\phi$ is a singleton almost everywhere,
the claim follows.
\epf

\subsection{A stability result}
\label{StabSubSec}
 \label{StabilityThmSubSec} 

Let $\Pi(\alpha,\beta) $  denote the set of probability measures  
on $X\times Y$ whose marginals are
$\alpha$ on $X$ and $\beta $ on $Y$, i.e.,
for every $\mu\in \Pi(\alpha,\beta)$,
$(\pi_1)_{\#}\mu=\alpha, (\pi_2)_{\#}\mu=\beta$, where $\pi_1:X\times Y\ra X,
\pi_2:X\times Y\ra Y$ are the natural projections. 
Elements of $\Pi(\alpha,\beta) $ are called transference plans.
 Given a function $c:X\times Y\rightarrow\mathbb{R}$, define the cost associated to $\mu\in \Pi(\alpha,\beta)$ by 
$$
\int_ {X\times Y}c\,d\mu.
$$ 
A transference plan is called optimal if it realizes the infimum of the cost over $\Pi(\alpha,\beta) $.
Optimal transference plans
satisfy the following standard stability result  \cite[Theorem 5.20]{VillaniOldNew}.

\begin{theorem}
\label{StabilityThm0}
Let $X$ and $Y$ be open subsets
of $\mathbb{R}^{n}$, and let $c:X\times Y\rightarrow\mathbb{R}$
be a continuous cost function with $\inf c>-\infty$.
Let $\alpha_{j}$
and $\beta_{j}$ be sequences of probability measures on $X$ and $Y$,
respectively, such that 
$\alpha_j$ converges weakly to $\alpha$ and
$\beta_{j}$ converges weakly to $\beta$. 
For each $j$, let $\pi_{j}$
be an optimal transference plan between $\alpha_{j}$ and $\beta_{j}$.
Assume that
\[
\int c\, d\pi_{j}<\infty, \;\forall j,\q \liminf_{j}\int c\, d\pi_{j}<\infty.
\]
 Then, there exists a subsequence $\{j_l\}_{l\in\NN}$ such that 
 $\pi_{j_l}$ converges weakly to an  optimal transference plan.
\end{theorem}

The following result, and its proof, are a slight modification of  \cite[Corollary 5.23]{VillaniOldNew}.

\begin{prop}
\label{StabilityThm}
Let $X$ and $Y$ be open subsets
of $\mathbb{R}^{n}$, and let $c:X\times Y\rightarrow\mathbb{R}$
be a continuous cost function with $\inf c>-\infty$. 
Let $\alpha_{j}$
and $\beta_{j}$ be sequences of probability measures on $X$ and $Y$,
respectively, such that $\alpha_{j} \leq C \alpha$ for all $j$, and
$\alpha_j$ converges weakly to $\alpha$ and
$\beta_{j}$ converges weakly to $\beta$. For each $j$, let $\pi_{j}$
be an optimal transference plan between $\alpha_{j}$ and $\beta_{j}$.
Assume that
\[
\int c\, d\pi_{j}<\infty, \;\forall j,\q \liminf_{j}\int c\, d\pi_{j}<\infty.
\]
Suppose that there exist measurable
maps $T_{j},T:X\rightarrow Y$ such that 
$\pi_{j}=\left(\id\otimes T_{j}\right)_{\#}\alpha_{j}$
and $\pi=\left(\id\otimes T\right)_{\#}\alpha$. Assume
additionally that $\pi$ is the unique 
 optimal transference plan in $\Pi(\alpha,\beta)$. 
 Then,
\[
\lim_{j\rightarrow\infty}\alpha_{j}\left[\left\{ x\in X\,:\,\vert T_{j}(x)-T(x)\vert>\varepsilon\right\} \right]=0,
\q
\h{for all $\varepsilon>0$}.
\]
\end{prop}

\begin{proof}
First, note that by Theorem \ref{StabilityThm0} and the uniqueness of $\pi$, we have that $\pi_{j}\rightarrow \pi$ weakly (and there is no need to take a subsequence).
Now, let $\varepsilon>0$ and $\delta>0$. By Lusin's theorem, there exists a compact set $K\subset X$ with $\alpha(X\backslash K)< C^{-1}\delta$ (so $\alpha_{j}(X\backslash K)< \delta$) such that the restriction of $T$ to $K$ is continuous. Then let 
\[
A_{\varepsilon} = \left\{ (x,y) \in K\times Y\,:\,\vert T(x)-y\vert \geq \varepsilon \right\}.
\]
By the continuity of $T$ on $K$, $A_{\varepsilon}$ is closed in $K\times Y$, hence also in $X\times Y$. Since $\pi = \left(\id\otimes T\right)_{\#}\alpha$, meaning in particular that $\pi$ is concentrated on the graph of $T$, we have that $\pi (A_\varepsilon)=0$. Then by weak convergence and the fact that $A_{\varepsilon}$ is closed,
$$
\baeq
0 = \pi(A_\varepsilon) & \geq  \limsup_{j\rightarrow\infty} \pi_{j} (A_{\varepsilon}) \\
& = 
\limsup_{j\rightarrow\infty} \pi_{j} \left( \left\{ (x,y) \in K\times Y\,:\,\vert T(x)-y\vert \geq \varepsilon \right\} \right) \\
& =  \limsup_{j\rightarrow\infty} \alpha_{j}  \left( \left\{ x \in K \,:\,\vert T(x)-T_{j}(x)\vert \geq \varepsilon \right\} \right) \\
& \geq   \limsup_{j\rightarrow\infty} \alpha_{j}  \left( \left\{ x \in X \,:\,\vert T(x)-T_{j}(x)\vert \geq \varepsilon \right\} \right) - \alpha_{j}(X\backslash K) \\
& \geq   \limsup_{j\rightarrow\infty} \alpha_{j}  \left( \left\{ x \in X \,:\,\vert T(x)-T_{j}(x)\vert \geq \varepsilon \right\} \right) - \delta, \\
\eaeq
$$
and the desired result follows by letting $\delta$ tend to zero.
\end{proof}

\section{Upgrading the convergence proof for the DMAOP}
\label{upgrade}

It is readily seen that the proof of Theorem \ref{MainThm} in the LDMAOP case holds with only slight modifications in the DMAOP case. Thus, to establish the DMAOP case of Theorem \ref{SecondMainThm}, it only remains to prove a stronger version of Proposition \ref{DMAOPProp}, i.e., a version that does not require regularity of the Brenier potential $\vp$ up to the boundary. In the following we only assume $\vp\in C^2(\O)$.

First, we make some updated definitions. By analogy with \eqref{FiDef} we redefine $F_i$ as follows:  
\beq
\label{FiDef2}
F_i\Big( \big\{\psi_j^{(k)},\eta_j^{(k)}\big\}_{j=1}^{N(k)} \Big)
:=
\max\bigg\{0,-(\det H_i)^{1/n}
+ 
\Big(
f\Big({\textstyle
\frac{ \sum_{j=0}^{n}x_{i_{j}}} {n+1} }\Big)
/
g\Big({\textstyle
\frac{ \sum_{j=0}^n \eta_{i_j} }{n+1} }\Big)
\Big)^{1/n}
\bigg\}.
\eeq
We redefine $F$ with respect to $F_i$ as in \eqref{FDef}:
\beq
\label{FDef2}
F\Big( \big\{\psi_j^{(k)},\eta_j^{(k)}\big\}_{j=1}^{N(k)} \Big)
:=
\sum_{i=1}^{M(k)} V_i \cdot F_i\Big( \big\{\psi_j^{(k)},\eta_j^{(k)}\big\}_{j=1}^{N(k)} \Big),
\eeq
so again $F_i$ is a per-simplex penalty, and $F$ is the objective function of the DMAOP.

To prove the strengthened version of Proposition \ref{DMAOPProp}, it will not suffice as it did before to simply plug the discrete data \eqref{vpkEq} associated to the Brenier potential $\vp$ into the DMAOP and hope that the corresponding cost goes to zero as $k\ra \infty$. The reason is that we no longer have that the Hessian of $\vp$ is bounded away from zero on $\overline{\O}$. Instead, we will define functions that are strongly convex on $\mathbb{R}^n$ and that agree with $\vp$ on subsets that exhaust $\overline{\O}$. Since the functions that we construct may not be differentiable, we will also need to mollify slightly before plugging the associated data into the DMAOP.

Let $U$ be open and compactly contained in $\O$. In turn, let $V$ be open such that $U \subset \subset V$ and $V\subset \subset \O$. Then there exists $\alpha>0$ such that $\nabla^2 \vp (x) \geq \alpha I$ for all $x\in \overline{V}$, and $\delta := \mathrm{dist}(\overline{U},V^c)$ is strictly positive.

Also, $\mathrm{dist}(\nabla\vp(x),\L^c)$ is continuous in $x$ over the compact set $\overline{U}$ and (since $\nabla\vp(\overline{U})\subset \L$) strictly positive. Hence the function attains a positive minimum $\gamma$ over $\overline{U}$, i.e., $\mathrm{dist}(\nabla\vp(x),\L^c) \geq \gamma > 0$ for all $x\in \overline{U}$.

Let $R = \mathrm{diam} (\O)$, and define 
\beq
\label{betaDefin}
\beta := \min \left\{\alpha,\frac{\alpha\delta^2}{R^2},\frac{\gamma}{2R} \right\}.
\eeq
For every point $y\in\overline{U}$, define a quadratic polynomial $Q_y$ on all of $\mathbb{R}^n$ by 
\beq
\label{Qdef}
Q_y (x) := \vp(y) + \langle\nabla\vp(y),x-y\rangle + \half\beta |x-y|^2.
\eeq

\begin{lemma}
\label{QLemma}
For any fixed $y\in \overline{U}$, $Q_y \leq \vp$ on $\overline{U}$ and $\nabla Q_y(\O_+) \subset \L$ for some open set $\O_+$ such that $\O\subset\subset\O_+$.
\end{lemma}
\begin{proof}
To simplify the proof of this claim, we fix $y\in\overline{U}$ and consider 
\beq
\label{Qineq}
\psi (x) := \vp(x) - \vp(y) - \langle\nabla\vp(y),x-y\rangle.
\eeq
For the first statement it suffices to show that 
\beq
\lb{sufficesEq}
\psi(x) \geq \half\beta|x-y|^2, \q \h{for all $x\in\overline{U}$.}
\eeq   
Note that $\psi$ is convex on $\mathbb{R}^n$ (since the Brenier potential may be taken to be defined on $\mathbb{R}^n$) with $\psi(0)=0$, $\nabla\psi(0)=0$, and $\nabla^2 \psi (x) \geq \alpha I$ for all $x\in \overline{V}$. Evidently $\psi \geq 0$ everywhere.
By integration along rays, for $x\in \overline{B_\delta (y)} \subset \overline{V}$, 
\beq
\label{Qineq}
\psi(x) \geq \half\alpha |x-y|^2
\eeq
In particular, since $\alpha \geq \beta$,  the inequality \eqref{sufficesEq} follows 
but only for $x\in\overline{B_\delta (y)}$.

Now let $x\in \overline{U} \setminus \overline{B_\delta(y)}$. 
In order to deal with the possible non-convexity of the domain $U$,
let $z\in \partial B_\delta(y)$ (so $|x-z|=\delta$) such that $x,y,z$ are collinear. From \eqref{Qineq} we have that $\psi(z) \geq \half\alpha\delta^2$. From this fact, together with \eqref{betaDefin}, we obtain
$$
\half\beta|x-y|^2 \leq \half\beta R^2 \leq \half\alpha \delta^2 \leq \psi(z).
$$
Now by our choice of $z$, we can write $z=tx+(1-t)y$ for some $t\in[0,1]$. Then by convexity (and the non-negativity of $\psi$), 
$$
\psi(z) \leq t\psi(x) + (1-t)\psi(y) = t\psi(x) \leq \psi(x).
$$
Thus we have established that $\half\beta|x-y|^2 \leq \psi(x)$, and the first statement of the lemma is proved.

Take $\O_+$ to be an open set with $\mathrm{diam}(\O_+) \leq \frac{3}{2} R$ such that $\O \subset\subset \O_+$.
It remains to show that $\nabla Q_y(\O_+) \subset \L$. Let $x\in \O_+$, and note that 
$$
\nabla Q_y(x) = \nabla \vp(y) + \beta(x-y),
$$
so (recalling \eqref{betaDefin})
$$
\vert \nabla Q_y(x) - \nabla \vp(y) \vert \leq \frac{3}{2}\beta R \leq \frac{3}{4} \gamma.
$$
But since $\mathrm{dist}(\nabla\vp(x),\L^c) \geq \gamma$, it follows that $\nabla Q_y(x)\in\L$. This completes the proof.
\end{proof}

In summary, we have shown that for any $U$ open and compactly contained in $\O$, there exists $\beta>0$ such that 
the quadratic polynomial $Q_y$ as defined in \eqref{Qdef} satisfies $Q_y \leq \vp$ on $\overline{U}$ and $\nabla Q_y(\O_+) \subset \L$ for all $y$, for some open $\O_+$ such that $\O\subset\subset\O_+$. (Note that $\beta$ does not depend on $y$.) Then define $\vp_U$ on all of $\mathbb{R}^n$ via 
\beq
\vp_U (x) := \sup_{y\in \overline{U}} Q_y (x).
\eeq
Evidently $\vp_U = \vp$ on $\overline{U}$ and $\nabla \vp_U (\O_+) \subset \L$
(see, e.g., \cite[Proposition 2.7] {RT}). Since the pointwise supremum of $\beta$-strongly convex functions is $\beta$-strongly convex, $\vp_U$ is $\beta$-strongly convex.

Although $\vp_U$ is not necessarily differentiable, we can substitute $\vp_U$ with a smooth approximation via the following lemma.

\begin{lemma}
Let $\ve>0$, and consider an open set $U\subset\subset\O$. There exists a smooth convex function $\tilde{\vp}:\mathbb{R}^n \ra \mathbb{R}$ with $\nabla^2 \tilde{\vp} \geq \beta I$, $\nabla \tilde{\vp} (\overline{\O}) \subset \L$, and $\Vert \tilde{\vp} - \vp \Vert_{C^{2}(\overline{U})} <\ve$.
\end{lemma}
\begin{proof}
Let $V$ be open such that $U\subset\subset V \subset\subset \O$. By the preceding arguments, we can take $\vp_V$ to be $\beta$-strongly convex and agreeing with $\vp$ on $V$ such that $\nabla \vp_V (\O_+) \subset \L$. Let $\xi_\delta$ denote, as before, a standard mollifier supported on $B_\delta (0)$. For $\delta$ sufficiently small, a $\delta$-neighborhood of $U$ is contained in $V$, so in fact, for small $\delta$, we have $\vp_V \star \xi_\delta = \vp \star \xi_\delta$ on $\overline{U}$. It follows that $\Vert \vp_V \star \xi_\delta - \vp \Vert_{C^{2}(\overline{U})} < \ve$ for small enough $\delta$. Furthermore, for $\delta$ sufficiently small, a $\delta$-neighborhood of $\overline{\Omega}$ is contained in $\O_+$, so by the convexity of $\L$, we have that $\nabla (\vp_V\star\xi_\delta) (\overline{\O}) \subset \L$. Noticing that mollification preserves $\beta$-strong convexity, the proof is completed by taking $\tilde{\vp} = \vp_V \star \xi_\delta$ for some $\delta$ small enough.
\end{proof}
\begin{remark}
One can avoid using the convexity of $\L$ in the proof of the preceding lemma by a more complicated argument. However, since we have assumed this fact elsewhere in this article, we make use of it here to keep the proof as simple as possible.
\end{remark}

As before, let $\Omega_\ve$ be as in Definition \ref{admissibleATDef}, and define $U_\ve := \O_\ve + B_{\ve /2}(0) \subset \O$. By the preceding lemma, we can let $\vp_\ve$ be smooth and convex such that $\nabla^2 \vp_\ve \geq \beta I$, $\nabla \vp_\ve (\overline{\O}) \subset \L$, and $\Vert \vp_\ve - \vp \Vert_{C^{2}(\overline{U_\ve})} <\ve$.
(Note that here $\vp_\ve$ is \textit{not} the same as $\vp \star \xi_\ve$.)

By analogy with \eqref{dkDef}, we consider the cost 
\begin{equation}
\begin{aligned}
\label{dkDef2}
d_\ve ^{(k)}:=F\Big( \big\{\vp_\ve (x_j^{(k)}),\nabla\vp_\ve (x_j^{(k)})\big\}_{j=1}^{N(k)} \Big)
\end{aligned}
\end{equation}
associated to the data
\begin{equation}
\begin{aligned}
\label{vpkEq2}
\big\{\vp_\ve (x_j^{(k)}),\nabla\vp_\ve (x_j^{(k)})\big\}_{j=1}^{N(k)}\in(\RR\times\RR^n)^{N(k)}.
\end{aligned}
\end{equation}
extracted from our modified Brenier potential $\vp_\ve$.

We now state and prove our improvement of Proposition \ref{DMAOPProp}.

\begin{prop}
\label{DMAOPProp2}
Let 
$
\big \{ \{S_{i}^{(k)}\}_{i=1}^{M(k)}\big\}_{k\in\NN}
$
be a sequence of admissible and regular almost-triangulations of $\Omega$ 
(recall Definitions \ref{admissibleATDef} and \ref{regularATDef}).
Let $\vp$ be the unique Brenier solution
of the Monge--Amp\`ere equation \eqref{MAEq} with $\varphi(0)=0$,
and suppose that $\vp\in C^{2}(\O)$.
Then:
\hfill\break
(i)
The data  \eqref{vpkEq2} satisfies
the constraints  \eqref{Constraint1Eq2}--\eqref{Constraint3Eq2} for all $k$ sufficiently large.
\hfill\break
(ii)
$\limsup_k d_{\ve}^{(k)} = o(\ve)$, where $d_\ve ^{(k)}$ is defined as in \eqref{dkDef2}.
\end{prop}

Let $c_k$ be the optimal cost of the $k$-th DMAOP. Let $\ve>0$. If the data \eqref{vpkEq2} associated with $\vp_\ve$ is feasible (which is true by part (i) of Proposition \ref{DMAOPProp2} for $k$ sufficiently large), then $c_k \leq d_\ve ^{(k)}$. Thus by part (ii) of Proposition \ref{DMAOPProp2}, $\limsup_k c_k = o(\ve)$. This yields the following analogue of Corollary \ref{costCor}.

\begin{cor}
\label{upgradedcostCor}
Under the assumptions of Proposition \ref{DMAOPProp2}, $\lim_k c_{k}=0$.
\end{cor}

\begin{proof}[Proof of Proposition \ref{DMAOPProp2}.]

First let 
\beq
I_k = \left\{i=1,\ldots,M(k)\,:\, S_i^{(k)} \subset \overline{U_\ve} \right\},
\eeq
and let 
\beq
J_k = \{1,\ldots,M(k)\} \setminus I_k.
\eeq
Also, recall that given a matrix $A=[a_{ij}]$, we define
$$
||A||=\max_{i,j}|a_{ij}|.
$$

Since $\vp_\ve$ is smooth (and so in particular in $C^{2,\alpha}(\overline{\Omega})$) and strongly convex on $\mathbb{R}^n$, we have by the same reasoning as in the proof of Lemma \ref{LindseyLemma}, that (recalling  (\ref{HiEq})) 
$$
\lim_k\max_{i\in\{1,\ldots,M(k)\}}
\Big\Vert 
H\Big(S^{(k)}_i,\Big\{
\nabla\vp_\ve (x_{i_0}^{(k)}),\ldots,
\nabla\vp_\ve (x_{i_n}^{(k)})
\Big\}\Big)
-\nabla^{2}\vp_\ve (x_{i_{0}}^{(k)})\Big\Vert
=0.$$
From this it follows (as in the proof of Proposition \ref{DMAOPProp}) that the data \eqref{vpkEq2} satisfies the constraint \eqref{Constraint3Eq2}. That the data satisfies the constraints \eqref{Constraint1Eq2} and \eqref{Constraint2Eq2} is evident from the construction of $\vp_\ve$.

Furthermore, since $\Vert \vp_\ve - \vp \Vert_{C^{2}(\overline{U_\ve})} <\ve$ by construction, we have that 
$$
\limsup_k\max_{i\in I_k}
\Big\Vert 
H\Big(S^{(k)}_i,\Big\{
\nabla\vp_\ve (x_{i_0}^{(k)}),\ldots,
\nabla\vp_\ve (x_{i_n}^{(k)})
\Big\}\Big)
-\nabla^{2}\vp (x_{i_{0}}^{(k)})\Big\Vert
\leq \ve.$$
From this it follows (as in the proof of Proposition \ref{DMAOPProp}) that 
$$
\limsup_k \max_{i\in I_k} F_i \left( \big\{\vp_\ve (x_j^{(k)}),\nabla\vp_\ve (x_j^{(k)})\big\}_{j=1}^{N(k)} \right) = o(\ve).
$$
Henceforth we will abbreviate $F_i := F_i \left( \big\{\vp_\ve (x_j^{(k)}),\nabla\vp_\ve (x_j^{(k)})\big\}_{j=1}^{N(k)} \right)$. Then the preceding implies that 
$$
\limsup_k \sum_{i\in I_k} V_i\cdot F_i = o(\ve).
$$
Thus to establish that $\limsup_k d_\ve ^{(k)} = o(\ve)$, it will suffice to show that 
$$
\limsup_k \sum_{i\in J_k} V_i\cdot F_i = o(\ve).
$$
However, since $f$ is bounded above, $g$ is bounded away from zero, and $(\cdot)^{1/n}$ is bounded below by zero over the nonnegative numbers, it follows from $\eqref{FiDef2}$ that $F_i \leq C$ for some constant $C$ (depending only on $f,g$). Thus 
$$
\limsup_k \sum_{i\in J_k} V_i\cdot F_i \leq C \limsup_k \sum_{i\in J_k} V_i = C \limsup_k \mathrm{vol}\left(\bigcup_{i\in J_k} S_i \right).
$$
Now suppose $i\in J_k$, and assume $k$ is large enough such that the maximal simplex diameter is at most $\ve/4$. Then $S_i$ contains a point $x$ that is not in $\overline{U_\ve}$. Recall that $U_\ve = \O_\ve + B_{\ve/2}(0)$, so it follows that $\mathrm{dist}(x,\O_\ve) \geq \ve/2$. Since $\mathrm{diam}(S_i)\leq \ve/4$, we see that $S_i \subset \overline{\O} \setminus \Omega_\ve$. Thus 
$$
\mathrm{vol}\left(\bigcup_{i\in J_k} S_i \right)
\leq \mathrm{vol}\left(\overline{\O} \setminus \Omega_\ve \right) = o(\ve),
$$
and this completes the proof.

\end{proof}

\section{Numerical experiments}

\subsection{Implementation details}
We provide only experiments for the DMAOP. Only three details of the implementation bear mentioning. First, we used DistMesh for the triangulation of $\Omega$ \cite{distMesh}. 
Second, we solved each convex optimization problem using MOSEK \cite{mosek}, called via the modeling language YALMIP \cite{YALMIP}. 
Third, our MATLAB program for solving the DMAOP 
allows the user to hand-draw the support of the source measure, and several of the following examples have source measures with hand-drawn support.

There is significant room for improvement in the efficiency of the implementation. The most
computationally 
expensive inefficiency is that we do not call MOSEK directly. Nonetheless, we are still able to solve the DMAOP over fine triangulations in an acceptable amount of time. It seems that problem size for the current implementation is actually limited not by run time, but by roundoff errors. We leave a detailed study of numerical aspects of this method to future work, though we provide basic observations about run time and convergence in  \S\ref{runtime} below.

\subsection{Examples}
\label{numEx}

We will consider only examples in the plane. Furthermore, we will always take the target measure $\nu$ to be the measure whose support is the unit ball and having uniform density on its support. It is not difficult to consider other convex target domains or to consider non-uniform log-concave densities (the most prominent examples being Gaussian densities). However, the visualizations that follow are more intuitive in the case that the target measure has uniform density on its support.

For our first example, we consider a source measure (see Figure \ref{convex_source}) supported on a convex polygon $\Omega$ with an oscillatory density $f$ bounded away from zero.
\begin{figure}
    \includegraphics[scale=0.5]{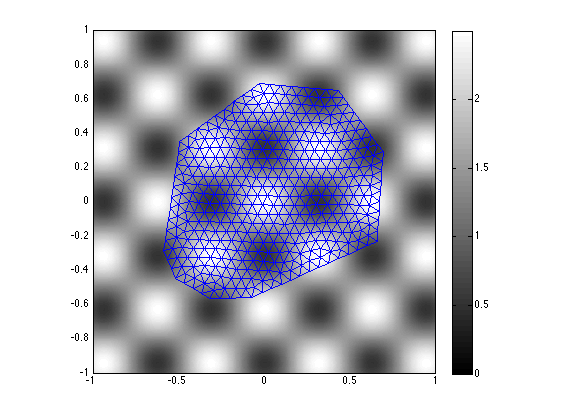}
    \caption{Source measure supported on a convex polygon $\Omega$ (seen triangulated in the above). The shading in the background represents the density of $f$, though we understand that $f\equiv 0$ outside of $\Omega$. There are 405 points in this triangulation.} 
    \label{convex_source}
\end{figure}

For the triangulation (consisting of 405 vertices) pictured in Figure \ref{convex_source}, the DMAOP took 49.3 seconds to solve on a 2011 MacBook Pro with a 2.2 GHz Intel Core i7 processor. (All numerical computations for this article were performed on this machine.)

For every point $x$ in the triangulation we can consider the interpolation $T_{t}(x):=(1-t)x+tT(x)$ for $t\in[0,1]$. We visualize this interpolation at times $t=0,\frac{1}{3},\frac{2}{3},1$. This interpolation can be understood as the solution of a dynamical optimal transport problem, though we will not discuss this fact further. See Figure \ref{convex_movie}.

\begin{figure}
    \noindent \begin{center}
    \includegraphics[scale=0.35, trim = 0mm 5mm 0mm 0mm, clip]{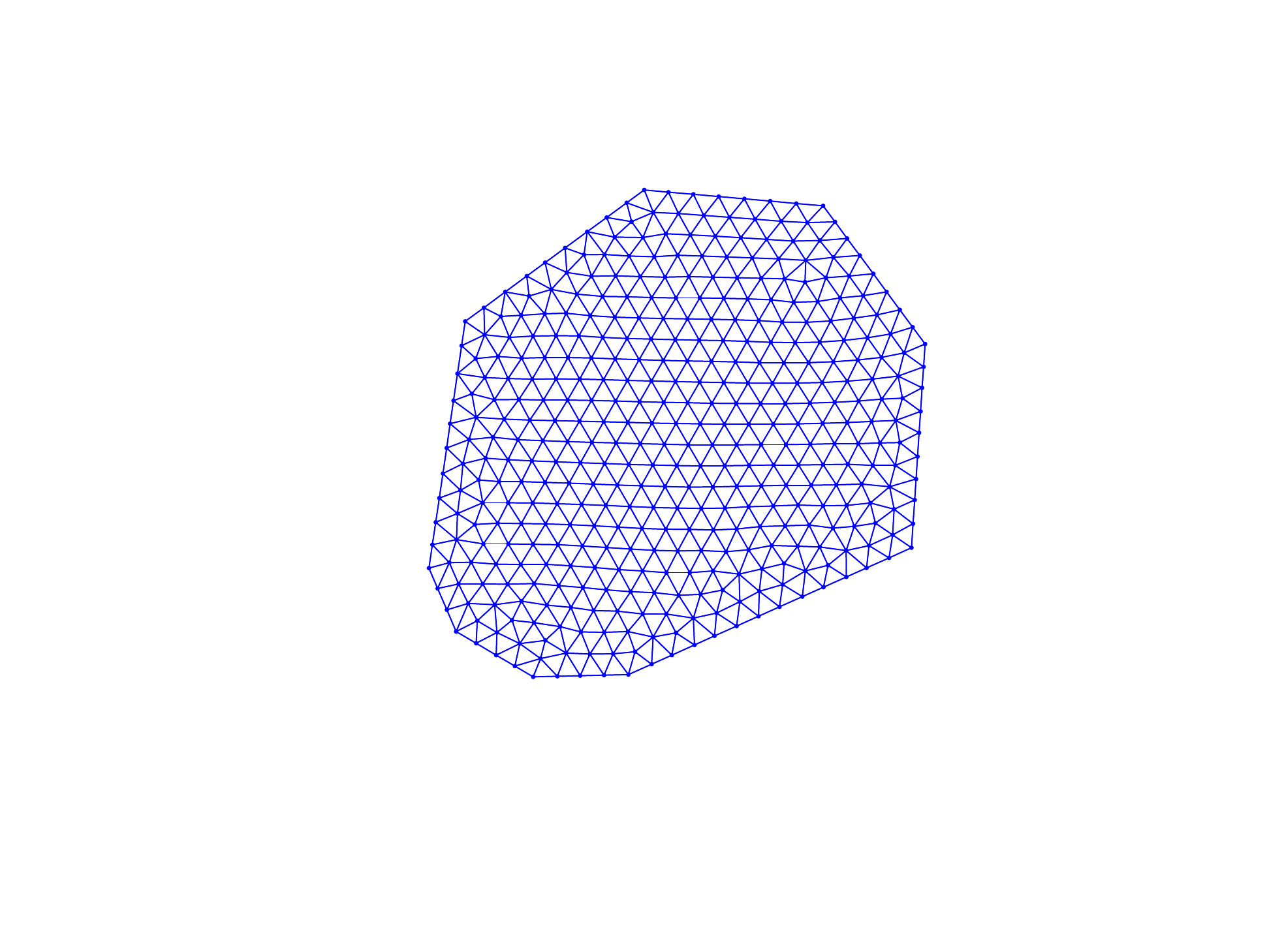}\includegraphics[scale=0.35, trim = 0mm 5mm 0mm 0mm, clip]{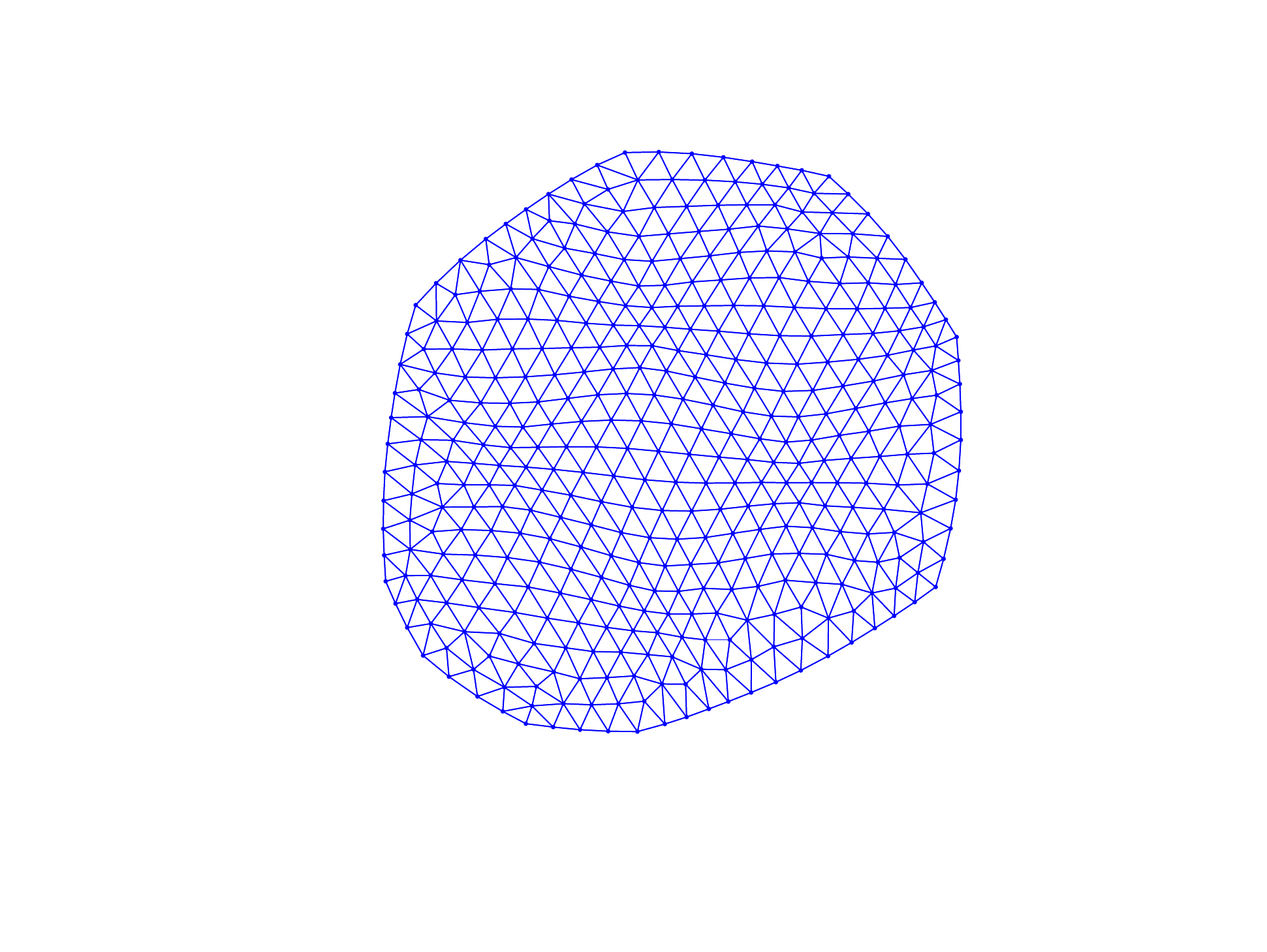}
    \par\end{center}
    
    \noindent \begin{center}
    \includegraphics[scale=0.35, trim = 0mm 7mm 0mm 0mm, clip]{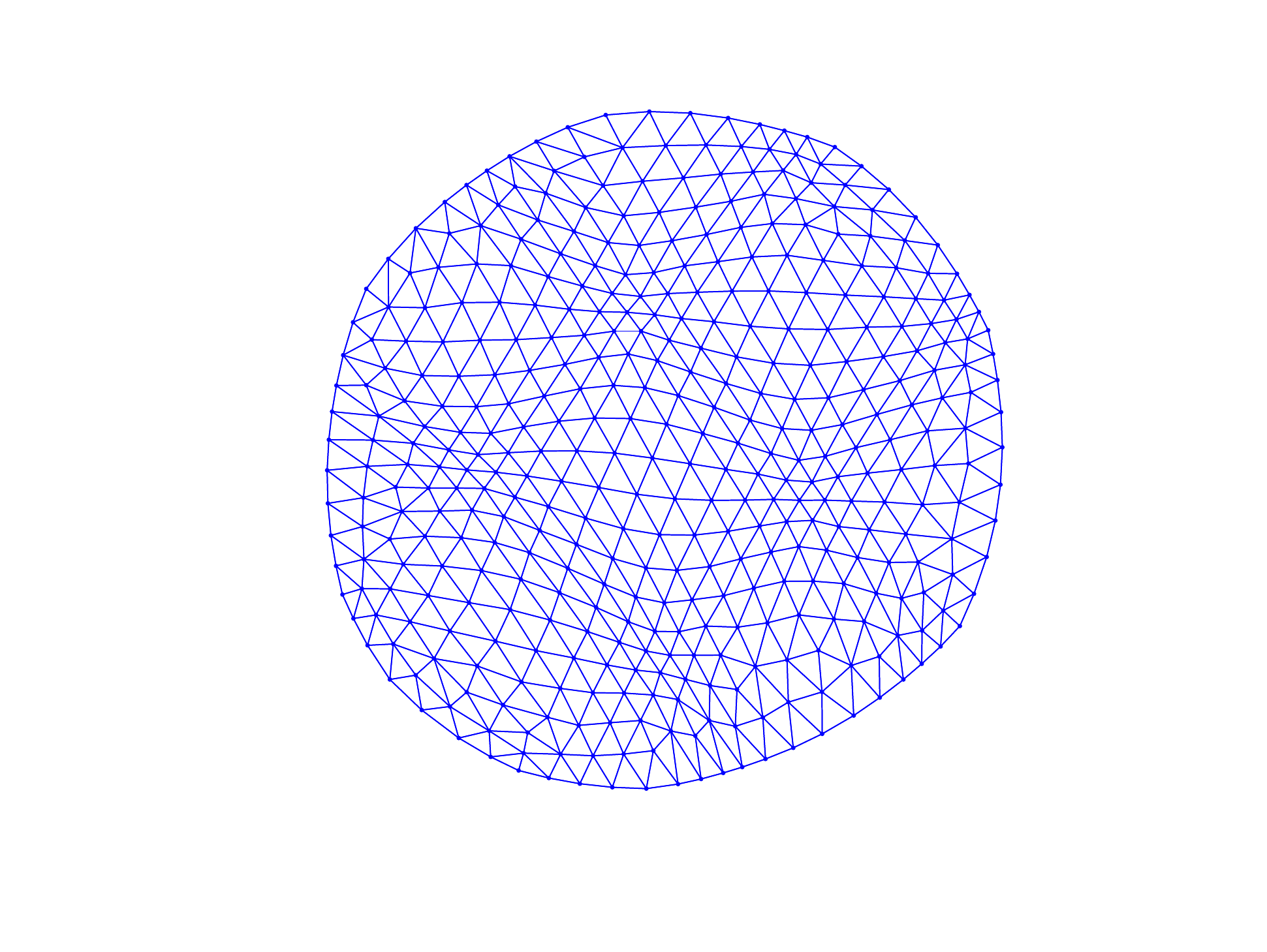}\includegraphics[scale=0.35, trim = 0mm 7mm 0mm 0mm, clip]{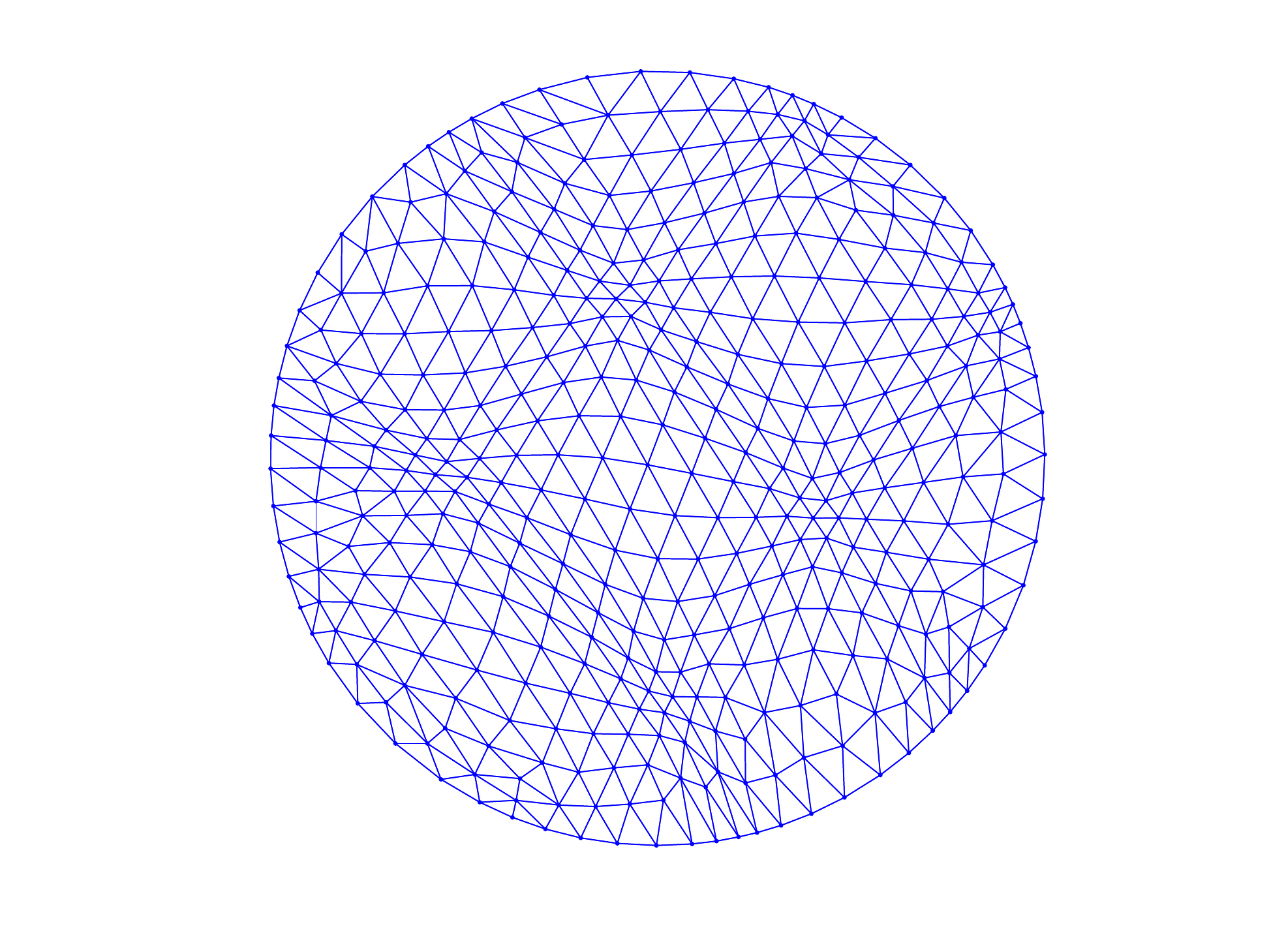}
    \par\end{center}
    
    \caption{Visualization of numerical solution of the dynamical optimal transport problem associated with the example of Figures \ref{convex_source} and \ref{convex_potential}. Times $t=0,\frac{1}{3},\frac{2}{3},1$ are depicted at upper left, upper right, lower left, and lower right respectively. (The target measure is uniform on the unit disc.)}
    
    \label{convex_movie}
\end{figure}

Next we consider a source measure with uniform density on a non-convex support. See Figure \ref{caf_movie} for a visualization of the domain, its triangulation, the numerical solution to the dynamical optimal transport problem. Our triangulation uses 340 points, and solving the DMAOP took 51.2 seconds. Note that a detailed theoretical study of a similar example is given in \cite{CJLPR}.
\begin{figure}
    \noindent \begin{center}
    \includegraphics[scale=0.35, trim = 0mm 0mm 0mm 2mm, clip]{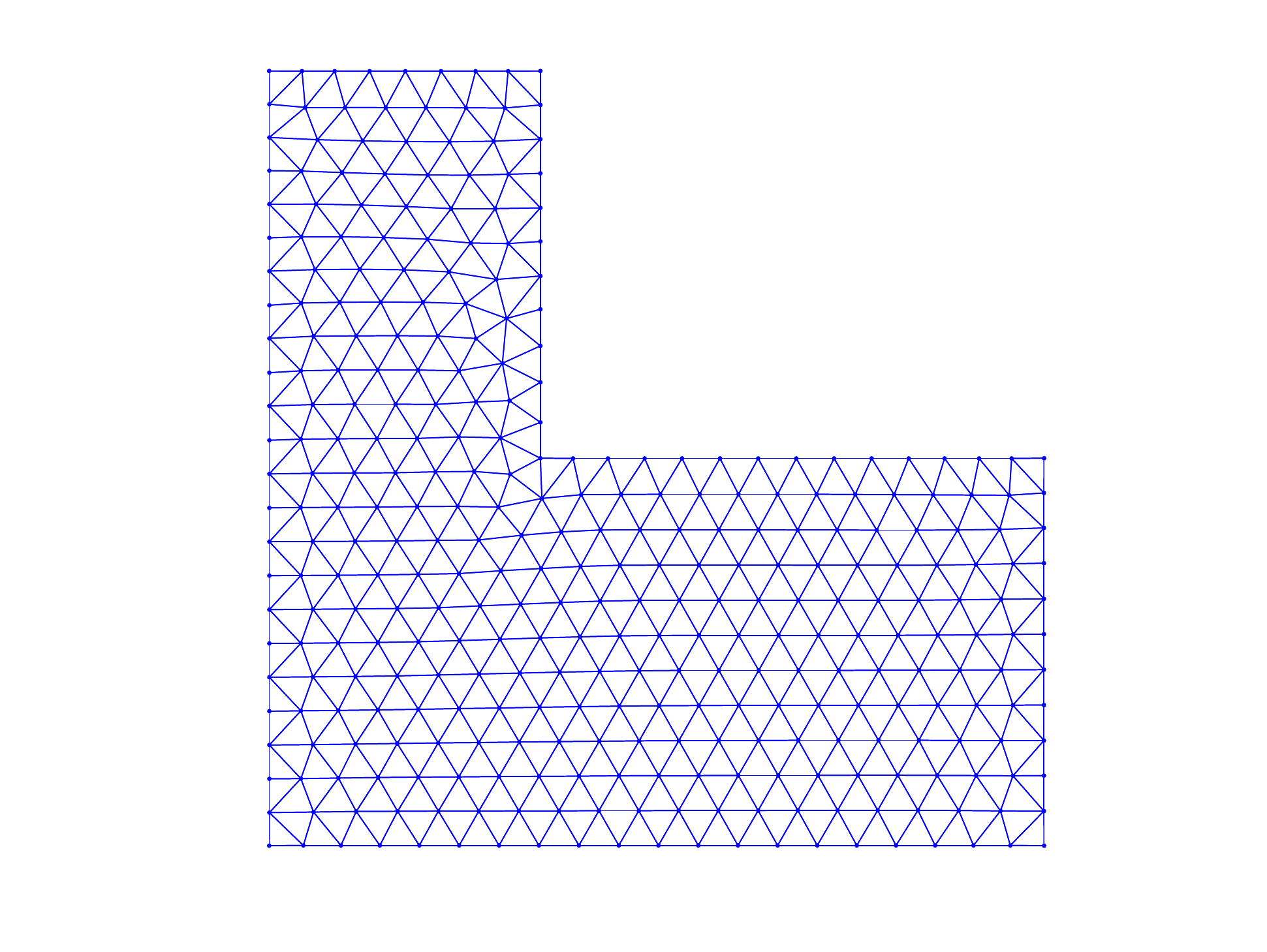}\includegraphics[scale=0.35, trim = 0mm 0mm 0mm 2mm, clip]{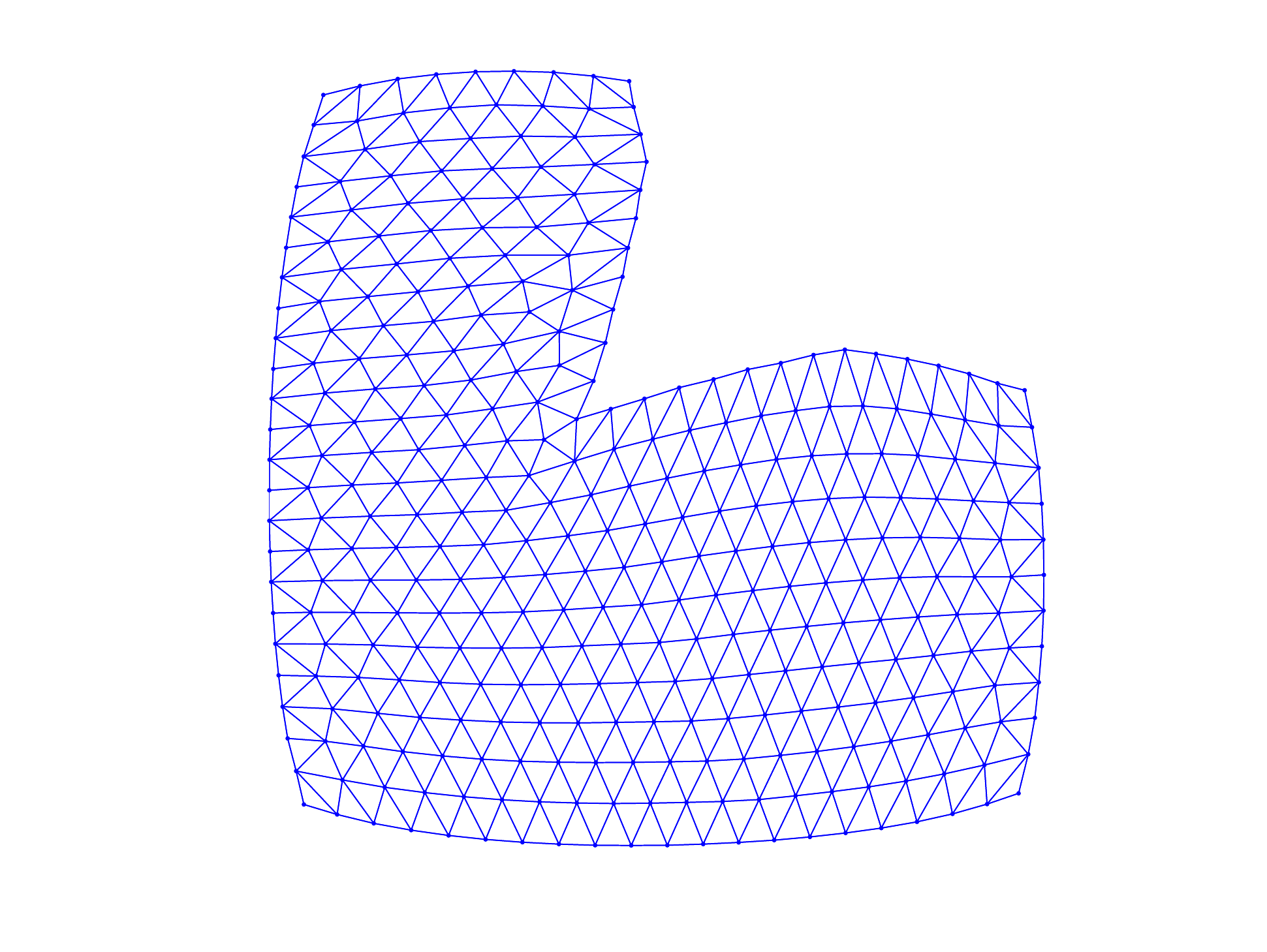}
    \par\end{center}
    
    \noindent \begin{center}
    \includegraphics[scale=0.35]{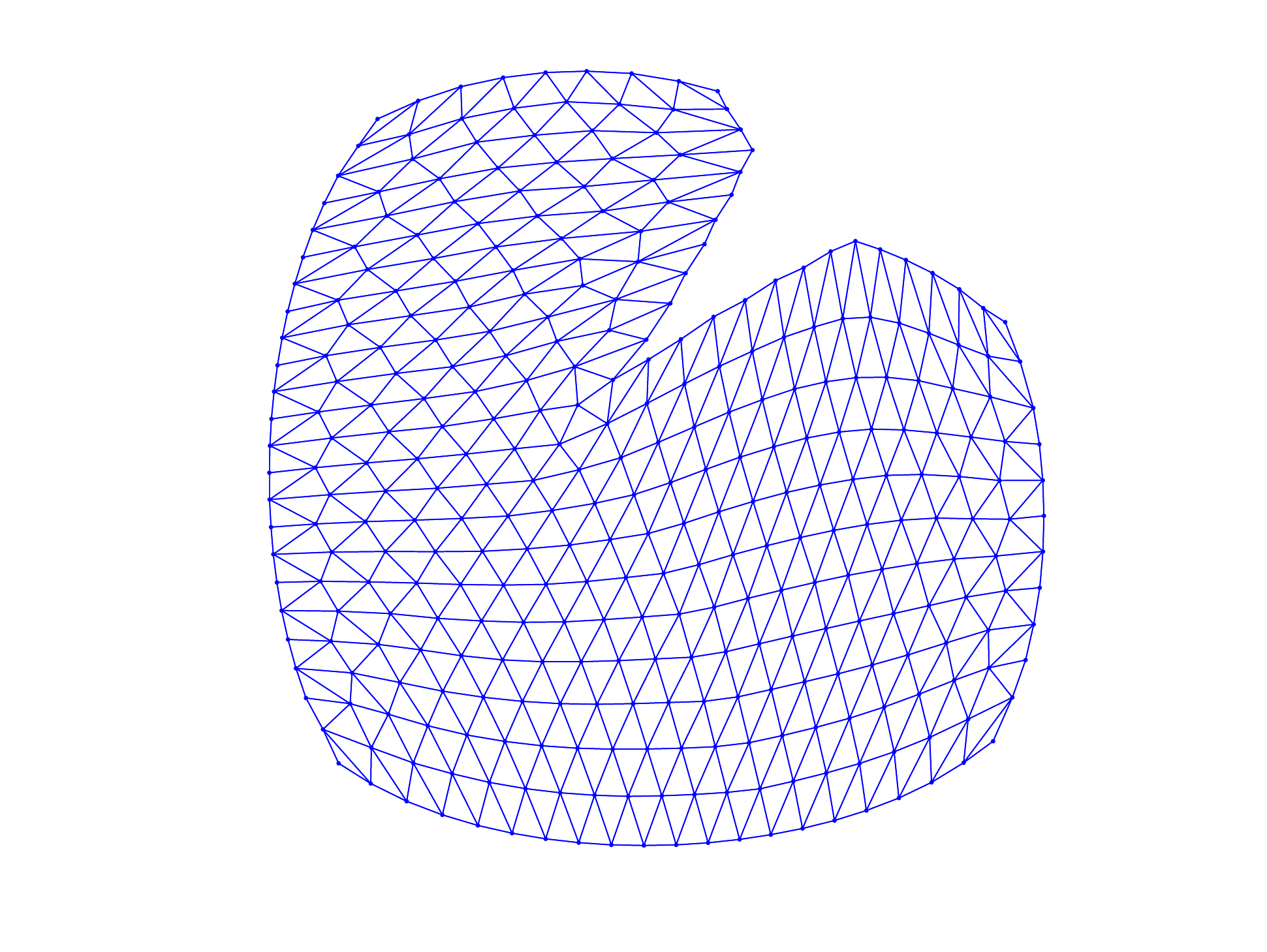}\includegraphics[scale=0.35]{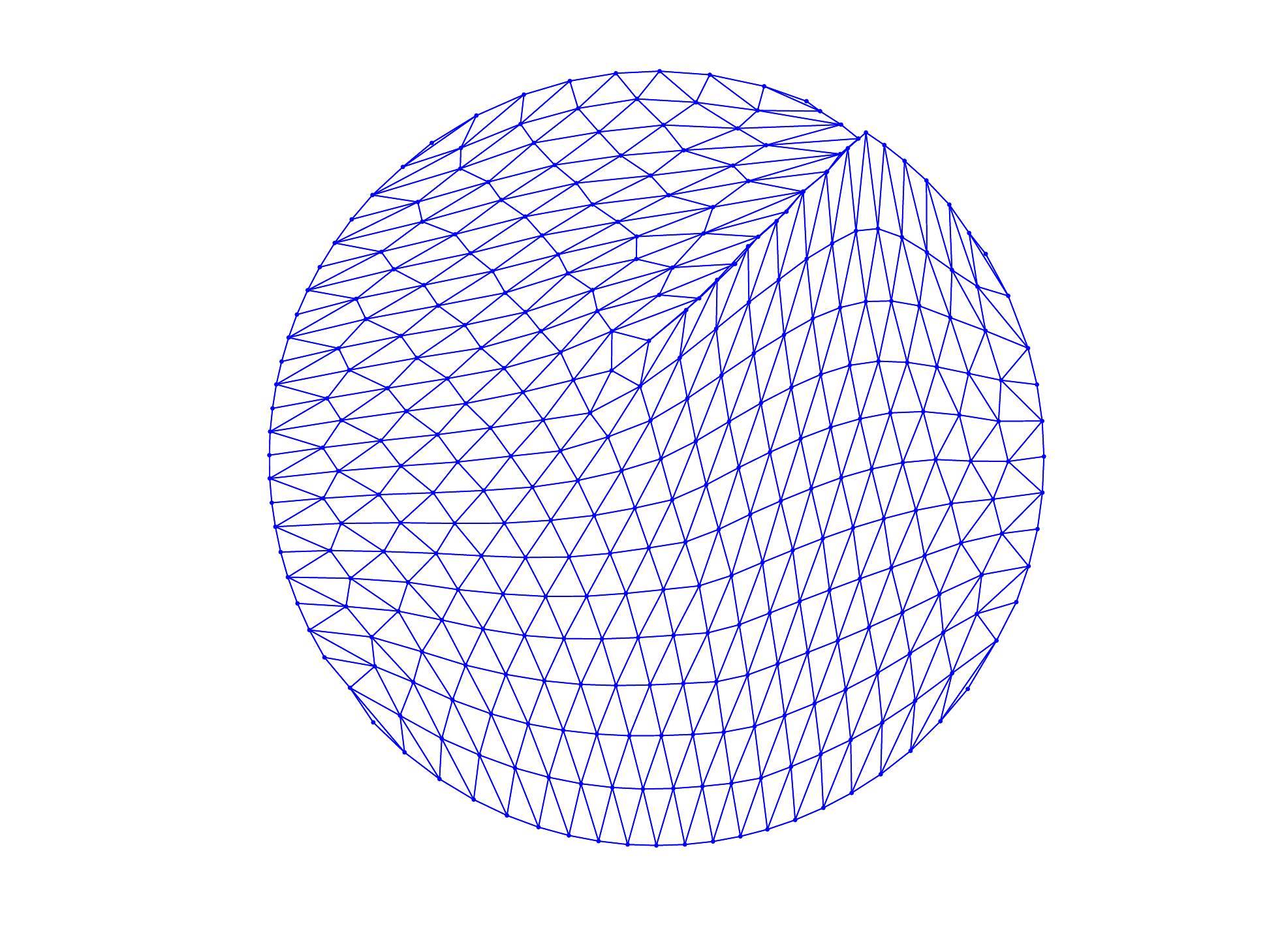}
    \par\end{center}
    
    \caption{Visualization of numerical solution of the dynamical optimal transport problem with source measure taken to be uniform and supported on the domain in the upper left and target measure taken to be uniform on the unit disc. Times $t=0,\frac{1}{3},\frac{2}{3},1$ are depicted at upper left, upper right, lower left, and lower right respectively.}
    
    \label{caf_movie}
\end{figure}
In Figure \ref{convex_potential} we visualize the computed convex potential.

\begin{figure}
    \includegraphics[scale=0.275]{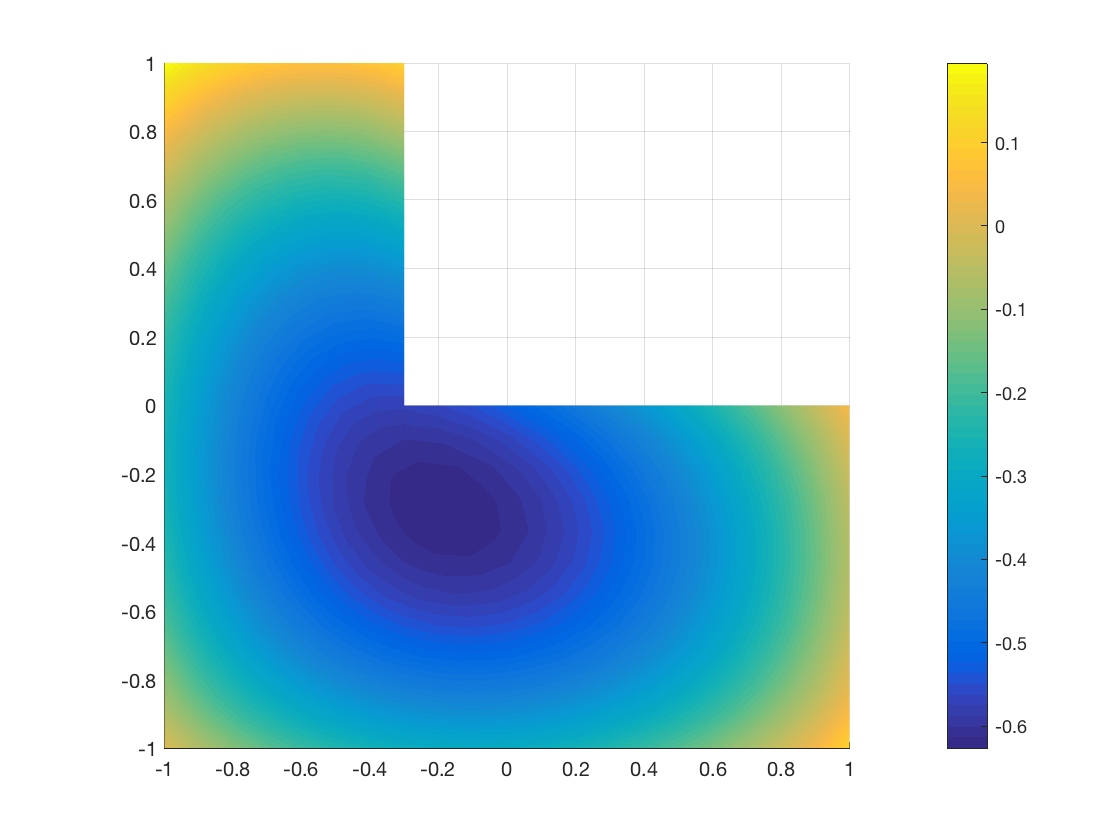}
    \caption{Visualization of the convex potential retrieved by solving the DMAOP associated to the source measure and triangulation in Figure \ref{caf_movie}, with shading corresponding to the value of the potential. (The target measure is the uniform on the unit disc.)}
    \label{convex_potential}
\end{figure}

Lastly we consider an example in which the source measure has highly irregular support (again with uniform density on its support). See Figure \ref{weird_movie} for details. There are 359 points in our triangulation, and solving the DMAOP took 58.9 seconds.

\begin{figure}
    \noindent \begin{center}
    \includegraphics[scale=0.35]{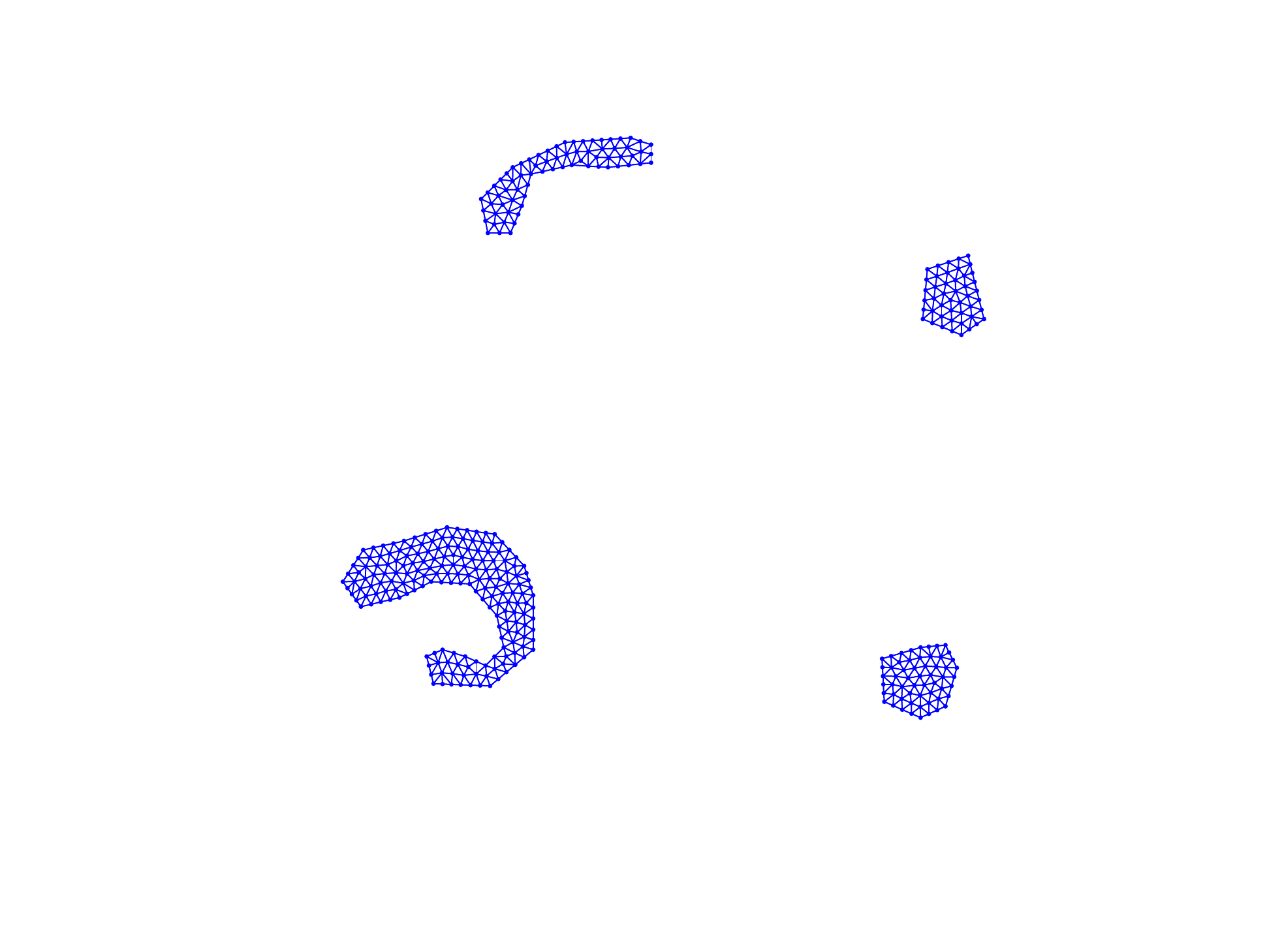}\includegraphics[scale=0.35]{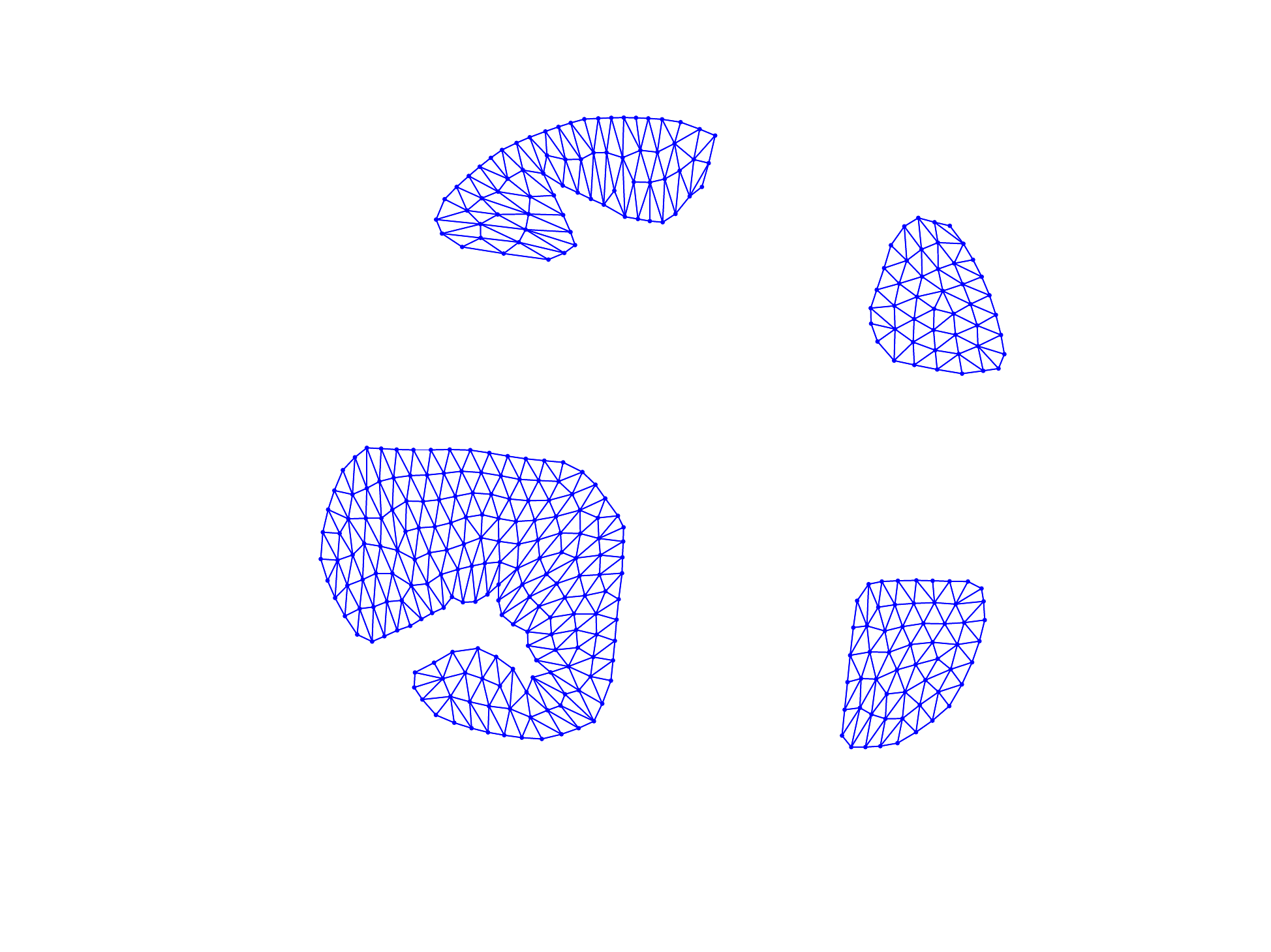}
    \par\end{center}
    
    \noindent \begin{center}
    \includegraphics[scale=0.35]{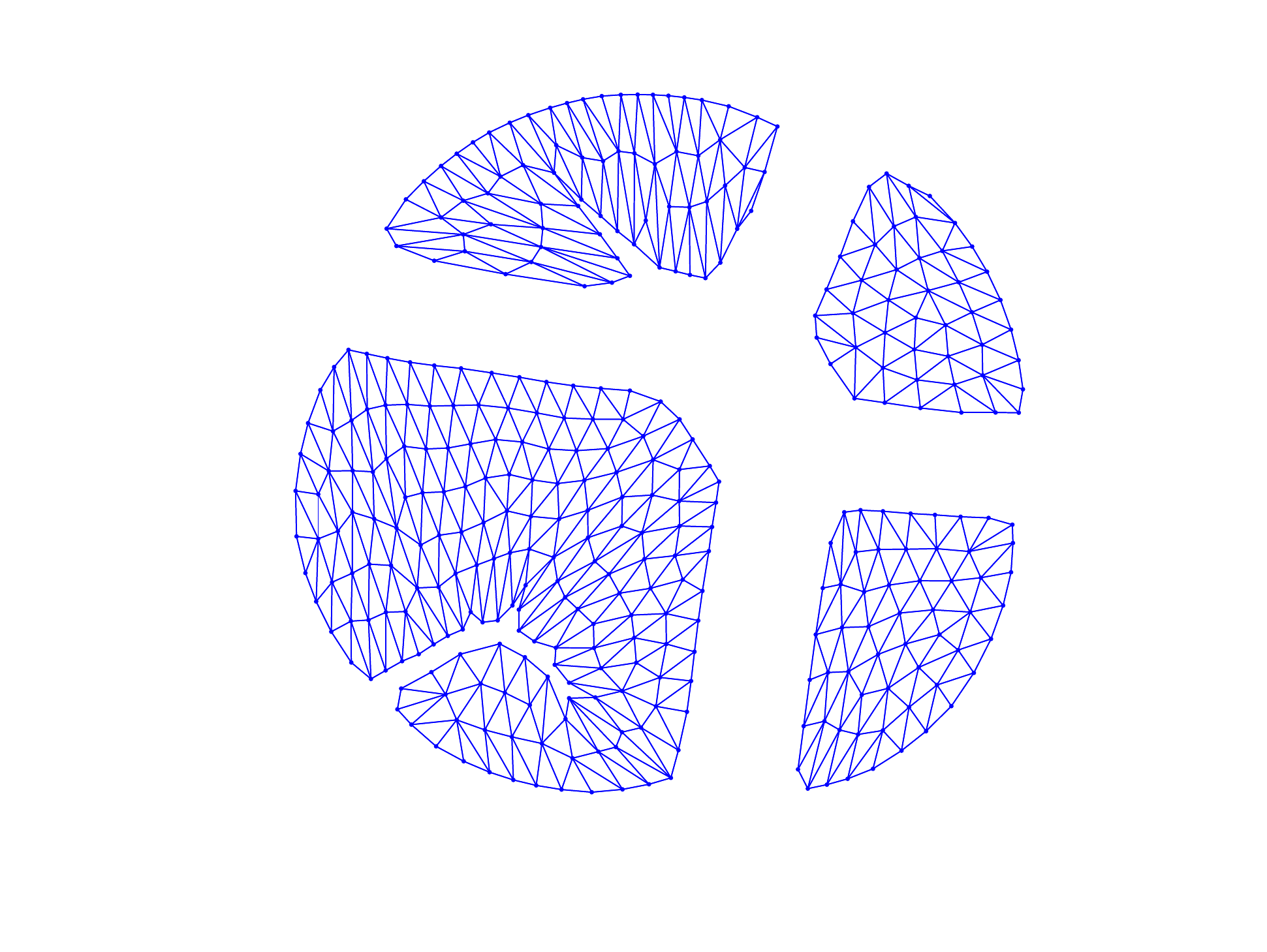}\includegraphics[scale=0.35]{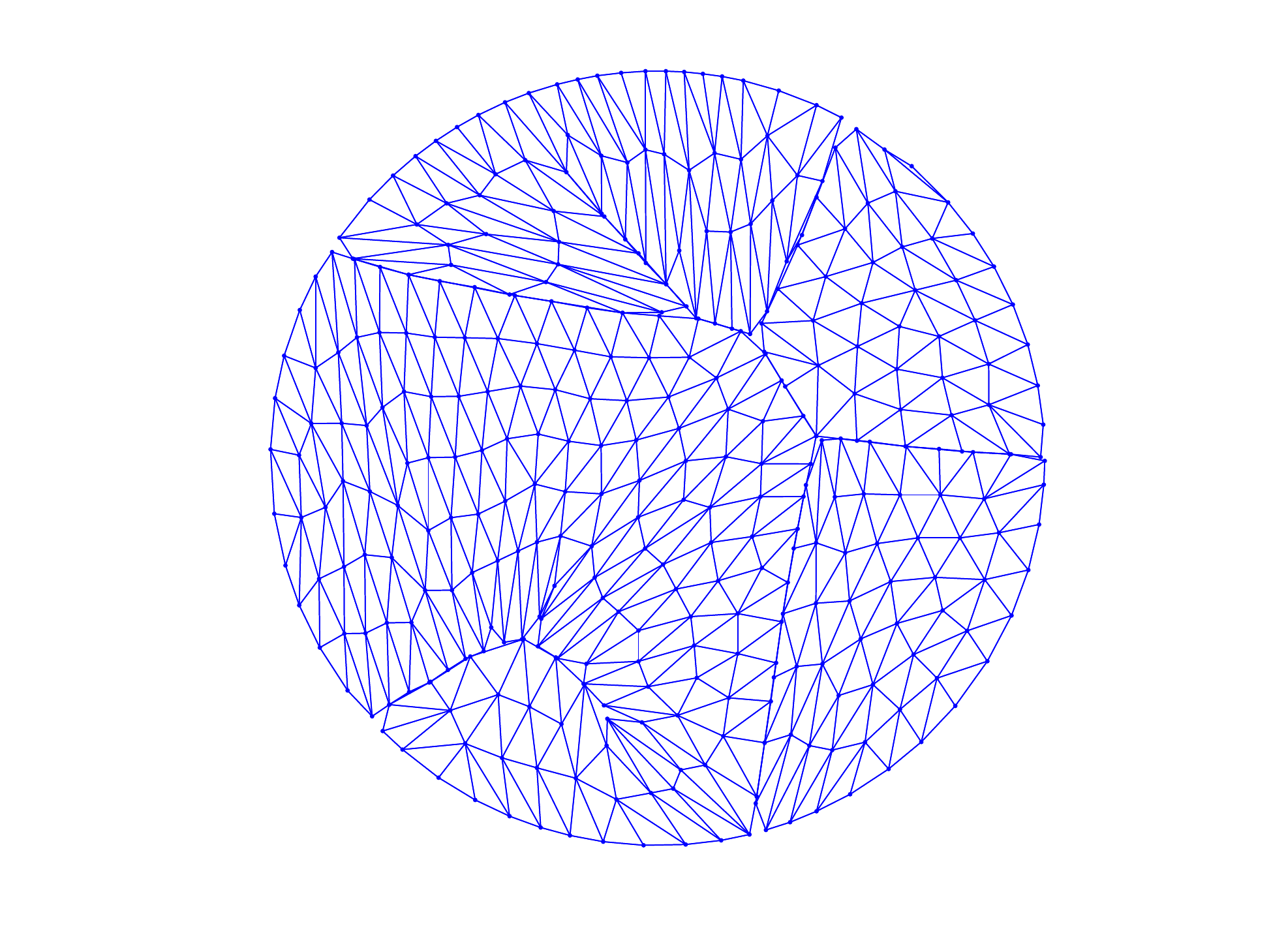}
    \par\end{center}
    
    \caption{Visualization of numerical solution of the dynamical optimal transport problem with source measure taken to be uniform and supported on the domain in the upper left and target measure taken to be uniform on the unit disc. Times $t=0,\frac{1}{3},\frac{2}{3},1$ are depicted at upper left, upper right, lower left, and lower right respectively.}
    
    \label{weird_movie}
\end{figure}

Notice that in the last two examples above, the inverse optimal maps are discontinuous. Nonetheless, we are able to approximate them by calculating the (continuous) forward maps and then inverting. Our method is particularly effective for highlighting the discontinuity sets of these inverse maps.

\subsection{Run time and convergence analysis}
\label{runtime}

We now fix an example problem and analyze the performance of our algorithm on discretizations of varying coarseness. In specific, we analyze the problem in the second example considered above (depicted in Figure \ref{caf_movie}), but we remark that such analysis does not depend noticeably on the choice of problem.

See Figure \ref{timePlot} for the dependence of run time on problem size. (All numerical computations were performed on a 2011 MacBook Pro with a 2.2 GHz Intel Core i7 processor.) The asymptotic behavior of the total run time (which includes mesh generation as well as the costly step of setting up the convex problem in the modeling language YALMIP) is, empirically, no worse than quadratic. However, the time spent by the convex solver (MOSEK) on the actual optimization problem is arguably a more fundamental quantity, and examination of the slope of a log-log plot (not pictured) of solver time against number of discretization points indicates quadratic growth. It is reasonable that this would be the case, since the number of constraints of the DMAOP grows quadratically in the number of discretization points. We expect that more efficient implementations could significantly cut down on time spent outside of the optimization step.

Next we examine the dependence of the cost (as in Definition \ref{DMAOPDef}) of our numerical solution on problem size. Figure \ref{costPlot} indicates that the cost decays as $N^{-\frac{1}{2}}$ (where $N$ is the number of discretization points). Since we are in dimension two, we expect that the mesh scale $h$ decays as $N^{-\frac{1}{2}}$, so in fact the cost decays like $h$. Note that the proof of Lemma \ref{LindseyLemma} guarantees that the optimal cost is $O(h^\alpha)$ whenever the Brenier potential satisfies $\varphi\in C^{2,\alpha}(\overline\O)$ and generally 
one may take $\alpha=1$ if $f$ and $g$ are sufficiently regular.

\begin{figure}
    \centering
    \includegraphics[scale=0.55]{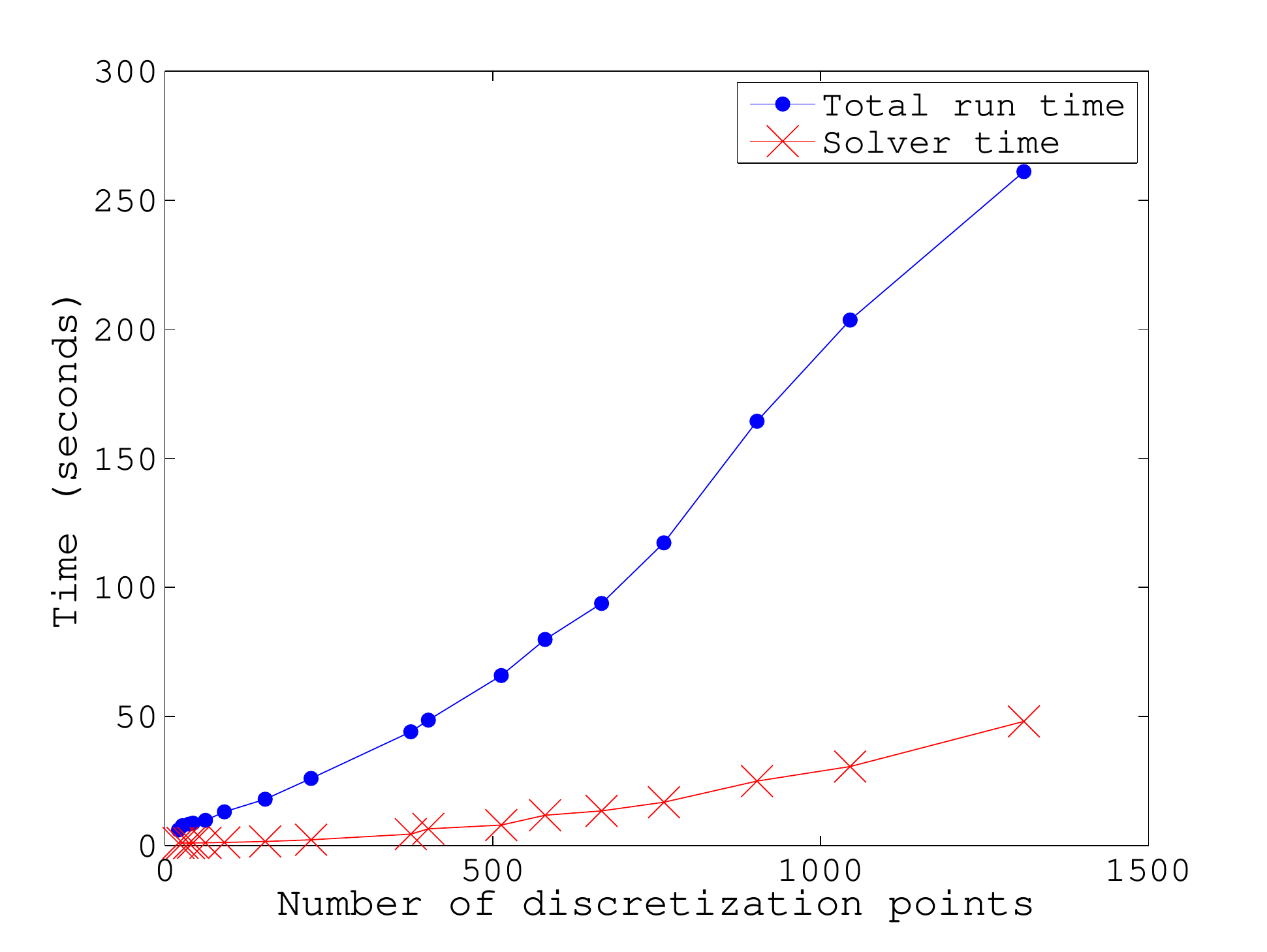}
    \caption{Plot of total run time of the algorithm and run time of the convex solver against the number of discretization points.}
    \label{timePlot}
\end{figure}

We also study the decay of a two-sided cost (not explicitly optimized in the DMAOP) that penalizes both excessive contraction and excessive expansion. With a view toward Definition \ref{mDMAOPDef} consider the quantity
$$
\tilde{c} := \sum_{i=1}^{M} V_i \cdot \Big| -\left(\det H_i\right)^{1/n}
+ \left( f\Big({\textstyle
\frac{ \sum_{j=0}^{n}x_{i_{j}}} {n+1} }\Big) 
/ g\Big({\textstyle
\frac{ \sum_{j=0}^n \eta_{i_j} }{n+1} }\Big) \right)^{1/n}
\Big|,
$$
where the $\eta_j$ are the $\eta_j$ of our solution of the DMAOP, and the $H_i$ are defined with respect to the $\eta_j$ as in \eqref{HiEq}. $\tilde{c}$ can be thought of as the average over the simplices of a two-sided penalty on area distortion. The dependence of $\tilde{c}$ on $N$ is depicted in Figure \ref{costPlot} and does not differ qualitatively from the dependence of the DMAOP cost on $N$.

\begin{figure}
    \centering
    \includegraphics[scale=0.55]{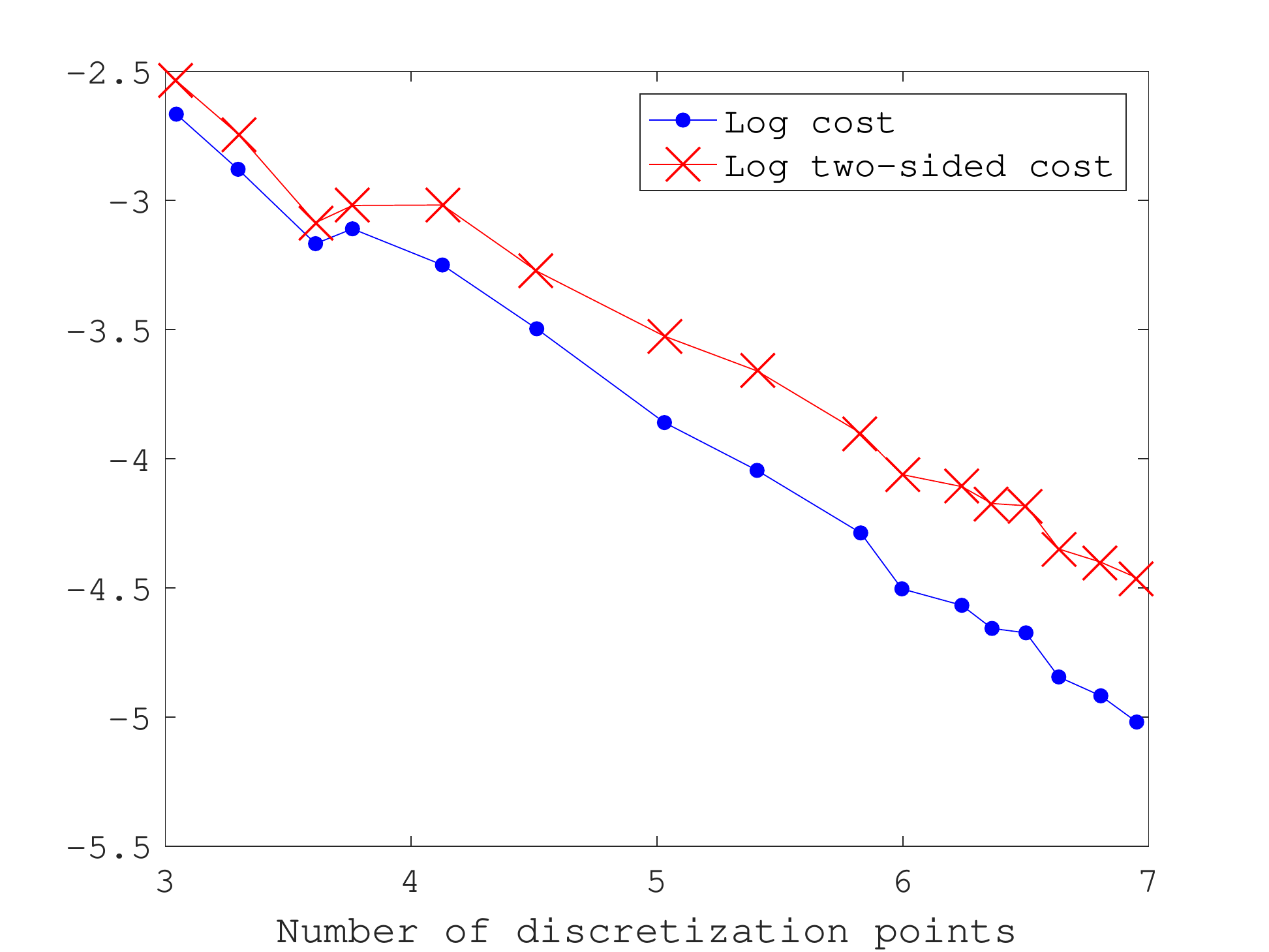}
    \caption{Log-log plot of the DMAOP cost of the numerical solution and a two-sided cost ($\tilde{c}$, introduced in \S\ref{runtime}) against the number of discretization points.}
    \label{costPlot}
\end{figure}

\subsection{Discussion}

We will not undertake a comparison with existing methods for numerical optimal transport, but we do make some remarks on behaviors, advantages, and disadvantages of our numerical method.

First, we comment that the method can be used to compute discontinuous optimal maps by inverting optimal maps from non-convex to convex domains, and the discontinuity sets can be resolved sharply (see \S\ref{numEx} for examples). These examples are in practice no more computationally expensive than convex-to-convex examples.

Also, although we have not taken advantage of this feature in the examples of \S\ref{numEx}, we remark that the method naturally allows for the preferential allocation of computational resources to more `difficult' regions (or any regions of particular interest) within the source domain. Indeed, we may simply solve the DMAOP for a triangulation with a greater density of vertices in desired areas.

In addition, we note that the extension of the implementation to higher dimensions is straightforward.

A significant limitation of our implementation is the requirement that the target measure have density $g$ for which $g^{-1/n}$ is convex. While the DMAOP still admits a minimizer for general target measures, it is only clear a priori that the DMAOP can be practically solved when it is convex (though see Section \ref{futureWork} below).

Another limitation is that the method is only first-order accurate. This is confirmed empirically in \S\ref{runtime}, but it is also to be expected due to the use of first-order finite difference quotients in the definition \eqref{HiEq} of the $H_i$ in terms of the subgradients $\eta_j$. It could be fruitful to replace these with higher-order difference quotients (this could be done easily on a Cartesian grid), though the convergence proof might require nontrivial modification to account for such a change (as the specific form of the $H_i$ is exploited quite directly in the proof), and it is not clear that faster convergence would follow.

Lastly, the quadratic growth of run time in $N$ is not ideal. This growth owes to the quadratic growth of the number of constraints of the DMAOP. We comment that this can be remedied by not including the $\eta_j$ as optimization variables in Definition \ref{DMAOPDef} and instead defining the $H_i$ directly in terms of the $\psi_j$ via finite difference quotients. However, we prefer Definition \ref{DMAOPDef} because it lends itself more naturally to the construction of the optimization potentials \eqref{phiEq}. We will leave the study of such a modified numerical approach to future work.

\section{Future directions}

As remarked in \S\ref{otherMethods}, our convergence result has the potential to be applied to other numerical schemes. Indeed, if one can show that a scheme yields vanishingly small one-sided deviation (in an $L^1$ sense) from satisfaction of the Monge--Amp\'ere equation, then our result is applicable. Such application could be pursued for new and existing methods.

Currently, it is an interesting open question whether our convergence proof can be upgraded to yield error bounds. At this point, we do not know how to achieve such bounds.

Lastly, we mention that though the (L)DMAOP is not a convex problem for general target measures, we may still ask whether there are any local minimizers that are not global. If there are no such `bad' local minima, then practical solution of the DMAOP may well be feasible even in the non-convex case. An investigation into this possibility shall be the subject of future work.

\label{futureWork}

\section*{Acknowledgments}
The authors are grateful to O. Chodosh, V. Jain, and L. Panchev
for many stimulating discussions and initial collaboration on this project
and on \cite{CJLPR}. 
This research was initially supported
by the Stanford University SURIM and OVPUE during Summer 2012, and
subsequently by NSF grants DMS-1206284,1515703 
and a Sloan Research Fellowship. 
Part of this work took place at MSRI (supported by NSF grant DMS-1440140)
during the Spring 2016 semester.


\begin{thebibliography}{99}
\def\bi{\bibitem}


\bi{AngenentHT}S. Angenent, S. Haker, A. Tannenbaum, { Minimizing flows for the Monge--Kantorovich problem,} SIAM J. Math. Anal. 35 (2003), 61--97. 

\bi{BaghW}A. Bagh, Roger J.-B. Wets, 
{ Convergence of set-valued mappings: equi-outer semicontinuity,}
Set-Valued Anal. 4 (1996), 333--360. 

\bi{BaldesWohlrab}A. Baldes, O. Wohlrab,
Computer graphics of solutions of the generalized Monge--Amp\`ere equation,
in: Geometric analysis and computer graphics (Berkeley, CA, 1988), 
Math. Sci. Res. Inst. Publ., Vol. 17, Springer, 1991, pp. 19--30,. 

\bi{BarlesS}G. Barles, P. E. Souganidis, { 
Convergence of approximation schemes for fully nonlinear second order equations,}
Asymptot. Anal., 4 (1991), 271--283.

\bi{BenamouBrenier}
J.-D. Benamou, Y. Brenier, { A computational fluid mechanics solution to the Monge-Kantorovich mass transfer problem,} 
Numer. Math. 84 (2000), 375--393. 

\bi{BenamouFO}
J.-D. Benamou, Jean-David; B.D. Froese, A.M. Oberman, { Numerical solution of the optimal transportation problem using the Monge--Amp\`ere equation,}
 J. Comput. Phys. 260 (2014), 107--126.

\bi{BenamouFO2}
J.-D. Benamou, B.D. Froese, A.M. Oberman, { A viscosity solution approach to the Monge--Amp\`ere formulation of the optimal transportation problem,} preprint,
arxiv:1208.4873.

\bi{Boyd}
S. Boyd, L. Vandenberghe, { Convex optimization,} Cambridge University Press, Cambridge, 2004.

\bi{Brenier}
Y. Brenier, 
{ Polar factorization and monotone rearrangement of vector-valued functions,}
Comm. Pure Appl. Math. 44 (1991), 375--417.

\bibitem{Caf1992}
L.~Caffarelli, {The regularity of mappings with a convex potential,} 
J. Amer. Math. Soc. 5, 99--104, (1992).

\bi{AMSAbstract} 
O. Chodosh, V. Jain, M. Lindsey, L. Panchev, Y.A. Rubinstein, 
{ Visualizing optimal transportation maps,} Abstracts
of papers presented to the Amer. Math. Soc., Vol. 34, No. 1086-28-496, 2013.

\bi{CJLPR}
 O. Chodosh, V. Jain, M. Lindsey, L. Panchev, Y.A. Rubinstein, 
{ On discontinuity of planar optimal transport maps,} 
J. Topology \& Analysis 7 (2015), 239--260.

\bibitem{chuiLandau}
C. K. Chui and P. W. Smith, {A note on Landau's problem for bounded intervals}, Am. Math. Mon. 82 (1975), 927--929.

\bibitem{DePFig}
G.~De~Philippis and A.~Figalli, {The Monge--Amp\`ere equation and its link to 
optimal transport,} preprint, available at {http://arxiv.org/abs/1310.6167} (2013)

\bi{FengGN}
X.-B. Feng, R. Glowinski, M. Neilan, { Recent developments in numerical methods for fully nonlinear second order partial differential equations,} 
SIAM Rev. 55 (2013), 205--267.

\bi{FroeseO} B.D. Froese, A.M. Oberman, { Convergent filtered schemes for the Monge--Amp\`ere partial differential equation,} 
SIAM J. Numer. Anal. 51 (2013), 423--444.

\bi{Guittet}K. Guittet, { On the time-continuous mass transport problem and its approximation by augmented Lagrangian techniques,} 
SIAM J. Numer. Anal. 41 (2003), 382--399.

\bi{HaberRT}
E. Haber, T. Rehman, A. Tannenbaum, 
{ An efficient numerical method for the solution of the $L^2$ 
optimal mass transfer problem, }
SIAM J. Sci. Comput. 32 (2010), 197--211.

\bibitem{HL1} J.-B. Hiriart-Urruty, C. Lemar\'echal, Convex analysis 
and minimization algorithms I, Springer, 1993.

\bi{Kitagawa} J. Kitagawa, 
{ An iterative scheme for solving the optimal transportation problem,} 
Calc. Var. Partial Differential Equations 51 (2014), 243--263. 

\bi{LoeperR}
G. Loeper, F. Rapetti,
{ Numerical solution of the Monge--Amp\`ere equation
by a Newton's algorithm,} 
C. R. Acad. Sci. Paris, Ser. I 340 (2005) 319--324.

\bibitem{YALMIP}
J. L\"{o}fberg, {YALMIP: A toolbox for modeling and optimization in MATLAB}, Proc. CACSD Conf., 2004.

\bi{QuentinOudet}
Q. M\'erigot, \'E. Oudet,
{ Discrete optimal transport: complexity, geometry and applications,}
preprint, available at: http://quentin.mrgt.fr/research/

\bi{Minkowski}H. Minkowski, { Volumen und Oberfl\"achen,} Math. Ann. 57 (1903),
447--495.

\bibitem{mosek}
MOSEK ApS, {The MOSEK optimization toolbox for MATLAB manual. Version 7.1 (Revision 28)} (2015), http://docs.mosek.com/7.1/toolbox/index.html. 

\bi{Neilan}M. Neilan, { Finite element methods for fully nonlinear second order PDEs based on a discrete Hessian with applications to the Monge--Amp\`ere equation,} 
J. Comput. Appl. Math. 263 (2014), 351--369.

\bi{OlikerPrussner}V.I. Oliker, L.D. Prussner, { On the numerical solution of the 
equation $(\partial^2z/\partial x^2)(\partial^2z/\partial y^2)-((\partial^2z/\partial x\partial y))^2=f$ and its discretizations. I,}
Numer. Math. 54 (1988), 271--293. 

\bi{PapadakisPO}N. Papadakis, G. Peyr\'e, E. Oudet,
{
Optimal transport with proximal splitting,}
SIAM J. Imaging Sci. 7 (2014), 212--238. 

\bibitem{distMesh}
P.-O. Persson, G. Strang, {A simple mesh generator in MATLAB}, SIAM Rev. 46 (2004), 329--345.

\bi{Pogorelov}
A.V. Pogorelov, { 
The Minkowski multidimensional problem,}, Wiley, 1978.

\bibitem{RT} J. Rauch, B.A. Taylor, 
{ The Dirichlet problem for 
the multidimensional Monge--Amp\`ere equation,}
Rocky Mountain J. Math. 7 (1977), 345--364.

\bibitem{Rock}
R.~T.~Rockafellar, {Convex analysis,} Princeton University Press, 1970.

\bi{RW} R.~T.~Rockafellar, R. J.-B. Wets, Variational analysis, Springer, 1997.

\bi{SulmanWR}
M.M. Sulman, J.F. Williams, R.D. Russell,
{
An efficient approach for the numerical solution of the Monge--Amp\`ere equation,} 
Appl. Numer. Math. 61 (2011), 298--307. 

\bibitem{Villani}
C.~Villani, {Topics in optimal transportation,}  
American Mathematical Society, 2003.

\bibitem{VillaniOldNew}
C.~Villani, {Optimal transport: old and new},  Springer, 2009.

\bibitem{Zahorski}
Z. Zahorski, 
{Sur l'ensemble des points de non-d\'erivabilit\'e d'une fonction continue,}
Bull.  Soc. Math. France 74 (1946), 147--178.
\end{thebibliography}
\end{document}